\RequirePackage[table]{xcolor}

\documentclass[a4paper,11pt]{amsart}

\usepackage[all,cmtip]{xy}
\usepackage[colorlinks=true, linktocpage=true]{hyperref}
\usepackage[paper=a4paper,left=20mm,right=20mm,top=25mm,bottom=30mm]{geometry}
\usepackage[shortlabels]{enumitem}
\usepackage{amsfonts}
\usepackage{amsmath}
\usepackage{amssymb}
\usepackage{amsthm}
\usepackage{bbm}
\usepackage[backend=biber,style=alphabetic,giveninits=true]{biblatex}
\usepackage{color}
\usepackage{eucal}
\usepackage{mathrsfs}
\usepackage{mathtools}
\usepackage{newpxtext}
\usepackage{newpxtext}
\usepackage{soul}
\usepackage{url}
\usepackage{verbatim}
\usepackage{textcmds} %Anführungszeichen english \qq{}
\usepackage[capitalise]{cleveref} 
\crefformat{equation}{(#2#1#3)}
\crefmultiformat{equation}{(#2#1#3}{ \& #2#1#3)}{, #2#1#3}{ \& #2#1#3)}
\crefname{construction}{Construction}{Constructions}
\crefrangeformat{equation}{(#3#1#4--#5#2#6)}

\usepackage{tikz,pgfplots}
\usetikzlibrary{backgrounds}
\usepackage{tkz-base}
\usepackage{tkz-euclide}
\usepackage{tikz-cd}
\pgfplotsset{compat=1.18}

\DeclareMathOperator{\add}{\mathsf{add}}

\DeclareMathOperator{\Aut}{Aut}

\DeclareMathOperator{\be}{\mathsf{e}}

\DeclareMathOperator{\Char}{char}

\DeclareMathOperator{\Coh}{\mathsf{Coh}}
\DeclareMathOperator{\coker}{coker}

\DeclareMathOperator{\dimu}{\underline{dim}}

\DeclareMathOperator{\EKP}{\mathsf{EKP}}
\DeclareMathOperator{\End}{End}

\DeclareMathOperator{\esp}{esp}

\DeclareMathOperator{\ext}{ext}
\DeclareMathOperator{\Ext}{Ext}

\DeclareMathOperator{\fv}{\mathfrak{v}}
\DeclareMathOperator{\gen}{\mathsf{gen}}

\DeclareMathOperator{\GL}{GL}

\DeclareMathOperator{\Gr}{Gr}

\DeclareMathOperator{\Hom}{Hom}

\DeclareMathOperator{\id}{id}

\DeclareMathOperator{\im}{im}

\DeclareMathOperator{\Inj}{Inj}

\DeclareMathOperator{\Iso}{Iso}

\DeclareMathOperator{\Mat}{Mat}

\DeclareMathOperator{\modd}{mod}

\DeclareMathOperator{\PGL}{PGL}

\DeclareMathOperator{\proj}{proj}

\DeclareMathOperator{\ql}{q\ell}

\DeclareMathOperator{\Rad}{Rad}

\DeclareMathOperator{\reg}{reg}

\DeclareMathOperator{\rep}{rep}
\DeclareMathOperator{\repp}{rep_{proj}}

\DeclareMathOperator{\res}{res}

\DeclareMathOperator{\rk}{rk}

\DeclareMathOperator{\Stab}{Stab}

\DeclareMathOperator{\StVect}{\mathsf{StVect}}
\DeclareMathOperator{\supp}{supp}

\DeclareMathOperator{\TilTheta}{\widetilde{\Theta}}

\DeclareMathOperator{\tr}{tr}

\DeclareMathOperator{\udim}{\underline{\dim}}

\DeclareMathOperator{\Vect}{\mathsf{Vect}}

\DeclareMathOperator{\pinj}{pinj}
\DeclareMathOperator{\pproj}{pproj}

\newcommand{\bmk}{\mathbbm{k}}

\newcommand{\cB}{\mathcal{B}}
\newcommand{\CC}{\mathbb{C}}
\newcommand{\cC}{\mathcal{C}}
\newcommand{\cD}{\mathcal{D}}
\newcommand{\cE}{\mathcal{E}}
\newcommand{\cF}{\mathcal{F}}
\newcommand{\cG}{\mathcal{G}}

\newcommand{\cI}{\mathcal{I}}
\newcommand{\cJ}{\mathcal{J}}
\newcommand{\cK}{\mathcal{K}}
\newcommand{\cL}{\mathcal{L}}

\newcommand{\cO}{\mathcal{O}}
\newcommand{\cP}{\mathcal{P}}

\newcommand{\cR}{\mathcal{R}}

\newcommand{\cT}{\mathcal{T}}

\newcommand{\cV}{\mathcal{V}}

\newcommand{\cX}{\mathcal{X}}

\newcommand{\fm}{\mathfrak{m}}

\newcommand{\fu}{\mathfrak{u}}
\newcommand{\fw}{\mathfrak{w}}

\newcommand{\KK}{\mathbbm{k}}
\newcommand{\lra}{\longrightarrow}

\newcommand{\NN}{\mathbb{N}}
\newcommand{\PP}{\mathbb{P}}

\newcommand{\ZZ}{\mathbb{Z}}

\colorlet{colorDaniel}{blue}

\let\emptyset\varnothing

\theoremstyle{plain}%
\newtheorem{proposition}{Proposition}[subsection]

\newtheorem{Theorem}[proposition]{Theorem}
\newtheorem{Lemma}[proposition]{Lemma}

\newtheorem{corollary}[proposition]{Corollary}

\newtheorem{TheoremA}{Theorem}

\theoremstyle{definition}
\newtheorem{Definition}[proposition]{Definition}

\newtheorem{example}[proposition]{Example}

\theoremstyle{remark}
\newtheorem{Remarks}[proposition]{Remark}
\newtheorem{Remark}[proposition]{Remark}

\newlist{enumerate2}{enumerate}{10}
\setlist[enumerate2]{label={({\alph*})}}
\setlist[enumerate]{label={({\arabic*})}}

\hypersetup{
  linkcolor  = blue, %Color Referenz
  citecolor  = blue, % Color Cite
  urlcolor   = blue, % Color Url
  filecolor = blue,  
}

\begin{document}

\title{Uniform Steiner bundles on $\PP^n$ and reflection functors}
\author{Daniel Bissinger}
\thanks{The author is supported by DFG Grant no. 548677842.}

\begin{abstract}
Let $n \in \NN_{\geq 2}$. We prove that for every $k \geq 4$ there exist uniform but non-homogeneous Steiner bundles on $\PP^n$ of $k$-type with disconnected splitting type, and we further investigate almost-uniform Steiner bundles.
Our approach relies on interpreting Steiner bundles as relative projective Kronecker representations and applying adjoint pairs arising from restriction, inflation, and reflection functors.
\end{abstract}
\maketitle

\section*{Introduction}

Let $\KK$ be an algebraically closed field and $n \in \NN_{\geq 2}$. Restricting vector bundles on $\PP^n$ along linear embeddings $\PP^d \hookrightarrow \PP^n$ for $1 \leq d \leq n$ is a frequently used technique in algebraic geometry, since restrictions of a given vector bundle $\cF \in \Vect(\PP^n)$ often allow conclusions about the bundle itself. For example, by Horrocks' splitting principle  \cite[(2.3.1),(2.3.2)]{OSS80} a vector bundle $\cF \in \Vect(\PP^n)$ splits (as a direct sum of line bundles) if and only if a restriction of $\cF$ to some plane $\PP^2$ splits.

By Grothendieck's Theorem \cite[(2.1.1)]{OSS80}, every vector bundle on $\PP^1$ splits. More precisely, any vector bundle $\cG$ on $\PP^1$ admits a decomposition
\[
\cG \cong \bigoplus_{i \in \ZZ} b_i \, \cO_{\PP^1}(i)
\]
into Serre twists of the structure sheaf $\cO_{\PP^1}$, which is unique up to isomorphism.
As a consequence, the restriction of a bundle $\cF \in \Vect(\PP^n)$ to a line $\PP^1 \cong L \subseteq \PP^n$ splits and yields a decomposition
\[ \cF|_{L} \cong \bigoplus_{i \in \ZZ} b_i(L) \cO_{L}(i).\]
Moreover, there exists a uniquely determined sequence $(b_i(\cF))_{i \in \ZZ}$ of natural numbers such that
\[ O_{\cF} \coloneqq  \{ L \subseteq \PP^n \mid L \ \text{line}, \forall i \in \ZZ \colon  b_i(L) = b_i(\cF) \}\] is an open subset of the Grassmannian $G(\PP^n)$ of lines in $\PP^n$. The corresponding subset 
\[\supp(\cF) \coloneqq \{ i \in \ZZ \mid b_i(\cF) \neq 0\}\] is called \textit{support} of $\cF$. 

The $\cF$ vector bundle is referred to as \textit{uniform} if $O_{\cF} = G(\PP^n)$, and \textit{homogeneous}, provided
\[ t^\ast(\cF) \cong \cF\]
for every projective transformation $t \in \PGL_{n+1}(\KK)$, where $t^\ast(\cF)$ denotes the pullback of $\cF$ along $t$. Since $\PGL_{n+1}(\KK)$ acts transitively on lines, homogeneous vector bundles are always uniform.

In 1961 Schwarzenberger \cite[(1.2)]{Sch61b} raised the question, whether all uniform vector bundles on $\PP^n$ are homogeneous. The first counterexample was given by Elencwajg \cite{Ele79} in 1979. He constructed a uniform but non-homogeneous bundle of rank $4$ on $\PP^2$ over $\CC$. Dr\'ezet  \cite{Dre80} generalized this result by proving the existence of uniform but non-homogeneous vector bundles on $\PP^n$ of rank $s$ for every $n \geq 2$ and $s \geq 2n$. 
It is further conjectured that, in characteristic zero, Dr\'ezet’s bundles are of minimal rank in this context; that is, every uniform vector bundle on $\PP^n$ of rank $< 2n$ is already homogeneous (see \cite[\S 3]{EM16})\footnote{The conjecture has been proven for $n \in \{2,3\}$ \cite{BE83,Ele77}. In positive characteristic, however, the situation is quite different. In this case, Xin has proven \cite{Xin18} the existence of uniform but non-homogeneous bundles of rank $n+1$ on $\PP^n$.}. 

Recently, Marchesi and Mir\'o-Roig \cite{MMR21} initiated the study of uniformity and homogeneity in the category $\StVect(\PP^n) \subseteq \Vect(\PP^n)$ of \textit{Steiner bundles} and proved, by giving explicit defining matrices, that uniform but non-homogeneous vector bundles of rank $s \geq 2(n+1)$ already exist in $\StVect(\PP^n)$. 
Moreover, they initiated a systematic study of the possible supports that can be realized by Steiner bundles.

Since $\cF \in \StVect(\PP(A_r))$ is globally generated, its support satisfies $\supp(\cF) \subseteq \NN_0$ and the element $k \coloneqq \max \supp(\cF)$ defines its \textit{type}. 
The constructed counterexamples in \cite{MMR21} are of $1$-type with support $\{0,1\}$ and the method used for their construction also entails the existence of new uniform $k$-type Steiner bundles for every $k \in \NN_{\geq 2}$ with support $\{0,1,\ldots,k\}$  (see \cite[(4.1)]{MMR21}). 

In this article, we further develop the study of uniform and homogeneous Steiner bundles.
The exact equivalence
\[
\TilTheta \colon \cX \lra \StVect(\PP^{n}),
\]
established in \cite{BF24} between a full subcategory $\cX$ of the category $\rep(K_{n+1})$ of representations of the generalized Kronecker quiver with $n+1$ arrows and the category $\StVect(\PP^{n})$ of Steiner bundles, enables us to study such bundles within the hereditary category $\rep(K_{n+1})$. 
This, in turn, naturally leads to the consideration of \textit{homogeneous} and \textit{uniform} representations, and their \textit{supports} in $\cX$.

The representations in $\cX$ have been characterized (see \cite[(2.1.5)]{BF24}) as being right $\Hom$-orthogonal to a certain algebraic family of representations, referred to as \textit{test representations}. Characterizing a category using such test representations is a well-established technique that has been effectively employed several times in the study of path algebras (cf. \cite{HU91,Wor13a,Bis20}).
We present a new construction of such families, unifying the ad hoc approaches in loc.\ cit., and naturally producing a family of test representations for $1 \leq d \leq n$ and every indecomposable homogeneous $K_d$-representation.

The main ingredients in the construction are adjoint pairs of functors between $\rep(K_{d+1})$ for $1 \leq d \leq n$ and $\rep(K_{n+1})$\footnote{The statement holds in more generality. We refer the interested reader to \cite{Bis25b}.} that are defined as compositions of restriction, inflation, and shift functors (see \cref{S:3} for the definitions):

\bigskip

\begin{TheoremA}\label{ThmA} \phantom{}
Let $1 \leq d \leq n$.
\begin{enumerate}
\item The functor $\sigma_{K_{n+1}}^{-1} \circ \inf \colon \rep(K_{d+1}) \lra \rep(K_{n+1})$ is left adjoint to $\sigma_{K_{d+1}} \circ \res \colon \rep(K_{n+1}) \lra \rep(K_{d+1})$. 
\item The functor $\sigma_{K_{d+1}}^{-1} \circ \res \colon \rep(K_{n+1}) \lra \rep(K_{d+1})$
is left adjoint to
$\sigma_{K_{n+1}} \circ \inf \colon \rep(K_{d+1}) \lra \rep(K_{n+1})$.
\end{enumerate}
\end{TheoremA}

\bigskip
Combining families of test representations, obtained from preprojective indecomposable $K_2$-representations, with an adapted version of Reineke’s result \cite{Rei24,Bis25} on general subrepresentations, we prove the existence of a large class of Kronecker representations in $\mathcal{X}$ corresponding to uniform but non-homogeneous Steiner bundles.
Transporting our findings to the category of Steiner bundles, we arrive at:

\bigskip

\begin{TheoremA}\label{ThmB}
Let $n \in \NN_{\geq 2}$ and $k \in \NN$.
\begin{enumerate}
    \item There exists a simple and uniform Steiner bundle on $\PP^n$ of $k$-type with support $\{k-1,k\}$ that is not homogeneous.
    \item There exists an indecomposable and uniform Steiner bundle on $\PP^n$ of $k$-type with support $\{0,1,k-1,k\}$ that is not homogeneous.
    % \item Let $s \in \NN$ such that $s \geq 2k^2n$ with $s = 2k^2 n + \ell$ and $....$. Then the generic Steiner bundle on $\PP^n$ with first Chern class $c$ of rank $s$ is simple, uniform of $\KK$-type with support $\{k,k+1\}$ and not homogeneous.
    \item Let $k \in \NN_{\geq 2}$, $s \geq 2k^2(n-1)$ and $\ell \coloneqq s -2k^2(n-1) \in \NN_0$. For each $c \in \NN_0$ satisfying  
    \[ (k-1)(k(n-1)+s) \leq c \leq (k-1)(k(n-1)+s)+\ell,\]
    the general Steiner bundle on $\PP^n$ with first Chern class $c$ and of rank $s$ is simple and uniform of $k$-type with support $\{k-1,k\}$, and not homogeneous.
\end{enumerate}
\end{TheoremA}

\bigskip

As a consequence of \cref{ThmB}, we obtain for $k \geq 4$ the existence of indecomposable and uniform Steiner bundles that do not have a \qq{connected} support.

We also apply our approach to the study of \textit{almost-uniform} vector bundles in $\StVect(\PP^n)$. Almost-uniform bundles were introduced by Ellia \cite{Ell17} and are defined by the property that the set of \textit{jumping lines} 
\[
\cJ_{\cF} \coloneqq G(\PP^n) \setminus O_{\cF}
\]  
is finite and non-empty. A main result of \cite{Ell17} asserts that there exists an almost-uniform vector bundle on $\PP^n$ of rank $2n - 1$ with exactly one jumping line. The given examples are, in fact, Steiner bundles. We show how these bundles arise naturally from \cref{ThmA} and prove the following statement, which implies that there is a rich supply of almost-uniform Steiner bundles.

\bigskip

\begin{TheoremA}\label{ThmC}
Let $\emptyset \neq \cL \subseteq G(\PP^n)$ be a finite set of lines. There exists a full subcategory of $\StVect(\PP^n)$ consisting of almost-uniform Steiner bundles with set of jumping lines $\cL$, that corresponds under $\TilTheta$ to a wild subcategory of $\rep({K_{n+1}})$.
\end{TheoremA}

\bigskip

This paper is organized as follows. In Section 1, we recall essential background from the representation theory of algebras and the theory of vector bundles on $\PP^n$. We also explain the connection between Kronecker representations and Steiner bundles, as established in \cite{BF24}. In Section 2, we analyze how uniform and homogeneous Kronecker representations behave differently in terms of Auslander–Reiten theory and within the variety of representations.

Section 3 is devoted to the proof of \cref{ThmA}. Finally, in Sections 4 and 5, we draw consequences from \cref{ThmA} and establish \cref{ThmB} and \cref{ThmC}.

\bigskip

\section{Preliminaries}

Throughout this work, $\KK$ denotes an algebraically closed field and, if not stated otherwise, $\KK$ is of arbitrary characteristic. Moreover, we denote by $\NN$ the set of natural numbers $\{1,2,3,\ldots\}$ and define $\NN_0 \coloneqq \NN \cup \{0\}$. Given a vector space $V \neq \{0\}$ and an integer $1 \leq d \leq \dim_{\KK} V$ we denote by
\[ \Gr_d(V) := \{ \fv \subseteq V \mid \dim_{\KK} \fv = d\} \]
the \textit{Grassmann variety} of $d$-dimensional subspaces of $V$, 
where the case $d = 1$ yields the full projective space $\PP(V) := \Gr_1(V)$ of dimension $\dim_{\KK} V - 1$.
For $r \in \NN$, we denote by $K_r$ the (generalized) \textit{Kronecker quiver} 
\[ \xymatrix{
K_r =& 1 \ar@/^1.4pc/^{\gamma_1}[r] \ar@/^/^{\gamma_2}[r]\ar@/_1pc/_{\gamma_r}[r] \ar@{}[r]|{\vdots} &  2
}
\]
with vertices $1$ and $2$, $r$ arrows $\gamma_i \colon 1 \lra 2$. The $r$-dimensional \textit{arrow space} is denoted by $A_r \coloneqq \bigoplus^r_{i=1}\KK \gamma_i$.

In this section, we fix notations and explain the relationship between Kronecker representations and Steiner bundles on projective space $\PP(A_r)$, focusing on the properties of uniformity and homogeneity for representations and vector bundles. For more details we refer the reader to \cite{BF24}.

\subsection{Homogeneous and uniform vector bundles}\label{S:1.1}

Given a projective variety $X$, we denote $\cO_X$ its structure sheaf and let $\Coh(X)$ be the category of coherent $\cO_{X}$-modules and $\Vect(X)$ be the category of vector bundles, i.e., locally free sheaves of finite rank on $X$. Both $\Coh(X)$ and $\Vect(X)$ are Krull-Schmidt categories (see \cite{Ati56}), while $\Coh(X)$ is also an abelian category \cite[(7.46)]{GW10}.

Given $\KK$-vector spaces $V$ and $W$, we define
\[ \Inj_{\KK}(V,W) := \{ \alpha \in \Hom_{\KK}(V,W) \mid \alpha \ \text{is injective}\}.\]
Every element $\alpha \in \Inj_{\KK}(V,W)$ induces an injective morphism $\hat{\alpha} : \PP(V) \lra \PP(W)$ of varieties, which gives rise to the inverse image functor 
\[ \hat{\alpha}^\ast : \Coh(\PP(W)) \lra \Coh(\PP(V)).\]
The inverse image functor is known to be right exact (cf.\ \cite[(7.11)]{GW10}) and sends vector bundles to vector bundles. Moreover, $\hat{\alpha}^\ast$ commutes with direct sums and the map that sends a fraction $\frac{f}{g}$ of homogeneous polynomials to 
$\frac{f\circ\alpha}{g\circ\alpha}$ induces isomorphisms $\hat{\alpha}^\ast(\cO_{\PP(W)}(j)) \cong \cO_{\PP(V)}(j)$ for every $j \in \ZZ$, where $\cO_{\PP(W)}(j)$ and $\cO_{\PP(V)}(j)$ denote the $j$-th \textit{Serre twisting sheaf} on $\PP(W)$ and $\PP(V)$, respectively.

The general linear group $\GL(A_r)$ acts on $\PP(A_r)$ via automorphisms and hence on $\Coh(\PP(A_r))$: Given $g \in \GL(A_r)$ and $\cF \in \Coh(\PP(A_r))$, we define
\[ (g^\ast.\cF)(U) \coloneqq \cF(g^{-1}.U)\]
for every open subset $U \subseteq \PP(A_r)$. By definition (see \cite[\S 7.8]{GW10}) we have
\[ g^\ast.\cF=\widehat{g^{-1}}^\ast(\cF)\]
for all $g \in \GL(A_r) = \Inj_{\KK}(A_r,A_r)$. 

\bigskip

\begin{Definition}\label{1.1.1}
A vector bundle $\cF \in \Vect(\PP(A_r))$ is called 
\begin{enumerate}
     \item \textit{uniform}, if  $\hat{\alpha}^\ast(\cF) \cong \hat{\beta}^\ast(\cF)$ for all $\alpha,\beta \in \Inj_{\KK}(A_{2},A_r)$.
    \item \textit{homogeneous}, provided $g^\ast.\cF \cong \cF$ for all $g \in \GL(A_r)$.
\end{enumerate}
\end{Definition}

\bigskip

Since $\GL(A_r)$ acts transitively on $\Inj_{\KK}(A_2, A_r) $ via $ g \cdot \alpha := g \circ \alpha$ for all $ g \in \GL(A_r)$ and $ \alpha \in \Inj_{\KK}(A_2, A_r)$, every homogeneous vector bundle on $ \PP(A_r)$ is uniform. However, for $ r \geq 3$, there exist uniform but non-homogeneous vector bundles on $ \PP(A_r)$ \cite{Ele79}.

On $ \PP(A_2) \cong \PP^1$, every vector bundle $ \cG \in \Vect(\PP(A_2))$ is homogeneous. This follows from Grothendieck’s Theorem \cite[(Theorem 2.1.1)]{OSS80}, which asserts that every vector bundle on $ \PP(A_2)$ admits a decomposition as a direct sum of line bundles 
\[
\cG \cong \bigoplus_{i \in \ZZ} b_i \, \cO_{\PP(A_2)}(i),
\]
and the fact that all line bundles $\cO_{\PP(A_2)}(j)$ are homogeneous.

As a consequence of Grothendieck's Theorem, we obtain for every $\cF \in \Vect(\PP(A_r))$ and every $\alpha\in \Inj_{\KK}(A_2,A_r)$ a unique decomposition
\[ \hat{\alpha}^\ast(\cF) \cong \bigoplus_{i \in \ZZ} b_i(\alpha,\cF) \cO_{\PP(A_2)}(i).\]
By the same token, we obtain for every $2$-dimensional subspace $\fv \in \Gr_2(A_r)$ a decomposition 
\[ \bigoplus_{i \in \ZZ} b_i(\fv,\cF) \cO_{\PP(\fv)}(i)\]
of the vector bundle $\cF|_{\fv} \coloneqq \hat{\iota}^\ast(\cF) \in \Vect(\PP(\fv))$, where $\iota \colon \fv \lra A_r$ is the inclusion morphism and $\hat{\iota}^\ast \colon \Coh(\PP(A_r)) \lra \Coh(\PP(\fv))$ the corresponding inverse image functor. 
Let $\alpha_0 \colon A_2 \lra \fv$ be an isomorphism and set $\alpha := \iota \circ \alpha_0 \in \Inj_{\KK}(A_2,A_r)$. By \cite[(7.8.9)]{GW10} we have $\hat{\alpha}^\ast(\cF)\!=\!\widehat{\iota \circ \alpha_0}^\ast(\cF)\!\cong\! \hat{\alpha_0}^\ast(\hat{\iota}^\ast(\cF))$ and conclude
\[ \bigoplus_{i \in \ZZ} b_i(\alpha,\cF) \cO_{\PP(A_2)}(i)\!\cong\! \hat{\alpha}^\ast(\cF)\!\cong\! \hat{\alpha_0}^\ast(\hat{\iota}^\ast(\cF))\!\cong\! \hat{\alpha_0}^\ast(\bigoplus_{i \in \ZZ} b_i(\fv,\cF) \cO_{\PP(\fv)}(i))\!\cong\!\bigoplus_{i \in \ZZ} b_i(\fv,\cF)\cO_{\PP(A_2)}(i).\]
Hence, $b_i(\alpha,\cF) = b_i(\fv,\cF)$ for all $i \in \ZZ$. As a direct consequence we obtain the following result.

\bigskip

\begin{Lemma}\label{1.1.2}
Let $r \geq 3$ and $\cF \in \Vect(\PP(A_r))$. The following statements are equivalent.
\begin{enumerate}
    \item $\cF$ is uniform.
    \item We have $b_i(\alpha,\cF) = b_i(\beta,\cF)$ for all $\alpha,\beta \in \Inj_{\KK}(A_2,A_r)$ and all $i \in \ZZ$.
    \item We have $b_i(\fv,\cF) = b_i(\fw,\cF)$ for all $\fv,\fw \in \Gr_2(A_r)$ and all $i \in \ZZ$.
\end{enumerate}
\end{Lemma}

\bigskip

Let $\cF \in \Vect(\PP(A_r))$ with $r \geq 3$. By \cite[(2.2.3)]{OSS80}, there exists a (uniquely determined) sequence $(b_i(\cF))_{i \in \ZZ} \in \NN_{0}^{\ZZ}$ such that 
    \[ O_{\cF} \coloneqq \{ \fv \in \Gr_2(A_r) \mid \forall i \in \NN_0 \colon b_i(\fv,\cF) = b_i(\cF) \}\]
is a dense open subset of $\Gr_2(A_r)$. We call
    \[ \cF_{\gen} \coloneqq \bigoplus_{i \in \ZZ} b_i(\cF) \cO_{\PP(A_2)}(i)\]
    the \textit{generic decomposition} or \textit{splitting type} of $\cF$. The closed subset $\cJ_{\cF} \coloneqq \Gr_2(A_r) \setminus \cO_{\cF}$ is called set of \textit{jumping lines} of $\cF$. If $\cF$ is uniform, we have $\cO_{\cF}= \Gr_2(A_r)$ and write in this case
    \[ \cF|_{\PP(A_2)} = \bigoplus_{i \in \ZZ} b_i(\cF) \cO_{\PP(A_2)}(i)\]
    to indicate that $b_i(\fv,\cF) = b_i(\cF)$ for every $\fv \in \Gr_2(A_r)$. Moreover, we call $\supp(\cF) \coloneqq \{ i \in \NN_0 \mid b_i(\cF) \neq 0\}$ the \textit{support} of $\cF$.

\bigskip

\subsection{Subcategories of Kronecker representations}

Let $r \geq 1$. We denote by $\rep(K_r)$ the category of representation of the Kronecker quiver $K_r$. Moreover, we let be $\cK_r$ the category whose objects are triples 
\[M = (M_1,M_2,\psi_M \colon A_r \otimes_{\KK} M_1 \lra M_2),\]
where $M_1,M_2$ are finite dimensional $\KK$-vector spaces and $\psi_M$ is a $\bmk$-linear map, called \textit{structure map} of $M$. A morphism $f \colon M \lra N$ in the category $\cK_r$ is a pair of $\bmk$-linear maps $f_i \colon M_i \lra N_i$, $i \in \{1,2\}$ such that the diagram
\[ 
\xymatrix{
A_r \otimes_{\KK} M_1 \ar^{\psi_M}[r] \ar[d]_{\id_{A_r} \otimes f_1} & M_2 \ar[d]^{f_2}\\
A_r \otimes_{\KK} N_1 \ar^{\psi_N}[r] & N_2}
\]
commutes. 
We have an equivalence of categories
\[ \rep(K_r) \lra \cK_r \ \ ; \ \ (M_1,M_2,(M(\gamma_i))_{1\leq i \leq r}) \mapsto (M_1,M_2,\psi_M)\]
that is the identity on morphisms and on the level objects is given by $\psi_M(\gamma_i \otimes m) = M(\gamma_i)(m)$ for all $i \in \{1,\ldots,r\}$ and all $m \in M_1$, where $M(\gamma_i) \colon M_1 \lra M_2$ is the $\KK$-linear map attached to the arrow $\gamma_i$. From now on we identify $\cK_r$ and $\rep(K_r)$.

The category $\rep(K_r)$ has two simple objects (representations) $S(1),S(2)$ with dimension vectors $(1,0)$ and $(0,1)$, respectively. The representation $S(1)$ is injective, while $S(2)$ is projective. We define $P_0(r) \coloneqq S(2)$ and denote by $P_1(r)$ the projective cover of $S(1)$ with dimension vector $(1,r)$ and structure map
\[ \psi_{P_1(r)} \colon A_r\otimes_{\KK} \KK \lra A_r \ ; \ \gamma_i \otimes \lambda \mapsto \lambda \gamma_i.\]
The representations $P_0(r)$ and $P_1(r)$ are a complete representation system of the projective indecomposable representations in $\rep(K_r)$.

Let $1 \leq d \leq r$. For $\alpha \in \Inj_{\KK}(A_d,A_r)$ and $M \in \rep(K_r)$ we define
\[ \psi_{\alpha^\ast(M)} \coloneqq  \psi_M \circ (\alpha \otimes \id_{M_1}) \colon A_d \otimes_{\KK} M_1 \lra M_2.\] 
The representation
\[\alpha^\ast(M) := (M_1,M_2,\psi_{\alpha^\ast(M)}) \in \rep(K_d)\] is referred to as \textit{restriction of $M$ along $\alpha$}.
According to \cite[(2.1.1)]{BF24}, $\alpha^\ast(M)$ is projective if and only if $\beta^\ast(M)$ is projective for all $\beta \in \Inj_{\KK}(A_d,A_r)$ with $\im \beta = \im \alpha$. Hence, we have
\begin{align*}
     \cV(K_r,d)_M \coloneqq & \{ \fv \in \Gr_d(A_r) \mid \alpha^{\ast}(M) \ \text{is not projective for all} \ \alpha \in \Inj_{\KK}(A_d,A_r) \ \text{with} \ \im \alpha = \fv\}\\
=& \{ \fv \in \Gr_d(A_r) \mid \alpha^{\ast}(M) \ \text{is not projective for some} \ \alpha \in \Inj_{\KK}(A_d,A_r) \ \text{with} \ \im \alpha = \fv\}.
\end{align*}
Moreover, for $ \fv \in \Gr_d(A_r)$, we define
\[
\psi_{M,\fv} \coloneqq \psi_{M}|_{\fv \otimes_{\KK} M_1} \colon \fv \otimes_{\KK} M_1 \longrightarrow M_2,
\]
as well as
\begin{align*}
    \cR(K_r,d)_M := \{ \fv \in \Gr_d(A_r) \mid \rk(\psi_{M,\fv}) < d \cdot \dim_{\KK} M_1 \} 
    = \{ \fv \in \Gr_d(A_r) \mid \ker \psi_{M,\fv} \neq \{0\} \},
\end{align*}
and refer to $ \cR(K_r,d)_M$ as the \textit{$ d$-th rank variety}\footnote{In fact, $ \cV(K_r,d)_M = \cR(K_r,d)_M$ is a closed subset of $ \Gr_d(A_r)$; see \cite[(2.1.1)]{BF24}.} of $ M$.
It follows from \cite[(2.1.5)]{BF24} that the two notions agree, that is, 
\[
\cR(K_r,d)_M = \cV(K_r,d)_M.
\]
Given $ \alpha \in \Inj_{\KK}(A_d, A_r)$, the map  
\[
\alpha^{-1} \otimes \id_{M_1} \colon \im \alpha \otimes_{\KK} M_1 \longrightarrow A_d \otimes_{\KK} M_1
\]
is a vector space isomorphism, and therefore
\[
\rk(\psi_{\alpha^\ast(M)}) 
= \rk\left(\psi_M \circ (\alpha \otimes \id_{M_1})\right) 
= \rk\left(\psi_M \circ (\alpha \otimes \id_{M_1}) \circ (\alpha^{-1} \otimes \id_{M_1})\right) 
= \rk(\psi_{M, \im \alpha}).
\]
Hence, we obtain $ \rk(\psi_{M, \im \alpha}) = \rk(\psi_{\alpha^\ast(M)})$, and in particular,
\[
\rk(\psi_{\alpha^\ast(M)}) = \rk(\psi_{\beta^\ast(M)})
\]
for all $ \alpha, \beta \in \Inj_{\KK}(A_d, A_r)$ such that $ \im \alpha = \im \beta$. In view of the fact that $ \dimu P_0(d) = (0,1)$ and $ \dimu P_1(d) = (1,d)$, we arrive at the following result; see also \cite[\S 2.2]{BF24}.

\bigskip

\begin{proposition}\label{1.2.1}
Let $1 \leq d < r$. Let $M \in \rep(K_r)$ be a representation and $\fv \in \Gr_d(A_r)$.
The following statements are equivalent.
    \begin{enumerate}
        \item[(i)] $\rk(\psi_{M,\fv}) = d \dim_{\KK} M_1$, i.e., $\psi_{M,\fv}$ is injective.
        \item[(ii)] $\fv \not\in \cV(K_r,d)_M = \cR(K_r,d)_M$.
        \item[(iii)] $\alpha^{\ast}(M) \cong (\dim_{\KK} M_2\!-\!d\dim_{\KK} M_1)P_0(d)\!\oplus\!(\dim_{\KK} M_1)P_1(d)$ for all $\alpha \in \Inj_{\KK}(A_d,A_r)$ with $\im \alpha = \fv$.   
    \end{enumerate}
\end{proposition}

\bigskip

\noindent \cref{1.2.1} motivates the following definition.

\bigskip

\begin{Definition}\label{1.2.2} We say that $M \in \rep(K_r)$ is \textit{relative $d$-projective} if 
$\cV(K_r,d)_M = \emptyset$, and denote by $\rep_{\proj}(K_r,d)$ the full subcategory of $\rep(K_r)$ consisting of relative $d$-projective representations.
\end{Definition}

\bigskip
\noindent Let $1 \leq d \leq r$. For $(x,y) \in \NN_0$, we set $\Delta_{(x,y)} \coloneqq y - x$ and define more generally
\[ \Delta_{(x,y)}(d) \coloneqq y - dx.\]
Given a representation $M \in \rep(K_r)$, or vector spaces $M_1,M_2 \in \modd \KK$, we define
\[ \Delta_M(d) \coloneqq \Delta_{(M_1,M_2)}(d) \coloneqq \Delta_{(\dim_{\KK} M_1,\dim_{\KK} M_2)}(d).\]
\noindent As a consequence of \cref{1.2.1} we have 
 \begin{align*}
\rep_{\proj}(K_r,d) &= \{ M \in \rep(K_r) \mid \forall \alpha \in \Inj_{\KK}(A_d,A_r) \colon \alpha^\ast(M) \in \rep(K_d) \ \text{is projective}\} \\
&= \{ M \in \rep(K_r) \mid \forall \alpha \in \Inj_{\KK}(A_d,A_r) \colon \alpha^\ast(M) \cong \Delta_M(d) P_0(d) \oplus (\dim_{\KK} M_1) P_1(d)\} \\
&= \{ M \in \rep(K_r) \mid \forall \alpha \in \Inj_{\KK}(A_d,A_r) \colon \ker \psi_{\alpha^\ast(M)} = \{0\}\}\\
&=\{ M \in \rep(K_r) \mid \forall \fv \in \Gr_d(A_r) \colon \ker \psi_{M,\fv} = \{0\}\}
\end{align*}
for each $d \in \{1,\ldots,r-1\}$ and obtain a nested sequence 
\[ \rep_{\proj}(K_r,r-1) \subseteq \rep_{\proj}(K_r,r-2) \subseteq \cdots \subseteq \rep_{\proj}(K_r,1).\]
Moreover, this description shows that $ \rep_{\proj}(K_r,1)$ coincides with the category $ \EKP(K_r)$ of \textit{equal kernels representations}. The latter was introduced in \cite{Wor13a}, building on the framework developed in \cite{CFS11}.

\bigskip

\subsection{Steiner bundles on projective space}

Steiner bundles on projective space were first systematically studied in the foundational work of Dolgachev and Kapranov \cite{DK93}, where their connections to hyperplane arrangements were explored. Since then, they have been the subject of extensive study \cite{Bra04,CHS22,MMR21}, and the definition has been generalized to arbitrary smooth irreducible varieties $X$; see, for example, \cite{MS09}.

\bigskip

\begin{Definition}\label{1.3.1} A vector bundle $\cF \in \Vect(\PP(A_r))$ is referred to as a \textit{Steiner bundle} if there exist vector spaces $V_1,V_2$ and an exact sequence
\[ 0\lra V_1 \otimes_{\KK} \cO_{\PP(A_r)}(-1) \lra V_2 \otimes_{\KK}\cO_{\PP(A_r)} \lra \cF \lra 0.\] 
We denote by $\StVect(\PP(A_r))$ the full subcategory of $\Vect(\PP(A_r))$ whose objects are Steiner bundles.
\end{Definition}

\bigskip

\begin{Remark}\label{1.3.2}
It follows from \cite[(3.2)]{DK93} that
$\StVect(\PP(A_r))$ is closed under extensions, and that for every exact sequence in $\Vect(\PP(A_r))$ \[ 0 \lra \cF_1 \lra \cF_2 \lra \cF_3 \lra 0\]
 with $\cF_1,\cF_2 \in \StVect(\PP(A_r))$ one also has $\cF_3 \in \StVect(\PP(A_r))$.
\end{Remark}

\bigskip

Connections between Kronecker representations and Steiner bundles on $\PP(A_r)$ have been known for some time (see \cite{Bra05}, \cite{Hul80}). A categorical equivalence between a suitable subcategory of $\rep(K_r)$ and the category of Steiner bundles - along with a generalized version for  Steiner bundles on Grassmannians $\Gr_d(A_r)$ - can be found in \cite[(3.2.3)]{BF24}. We only recall the result for $\PP(A_r) = \Gr_1(A_r)$.

\bigskip

\begin{Theorem} \label{1.3.3}   There exists a right exact functor
$\TilTheta : \rep(K_r) \lra \Coh(\PP(A_r))$
such that the following statements hold.
\begin{enumerate}
\item A vector bundle $\cF \in \Vect(\PP(A_r))$ is a Steiner bundle if and only if there is $M \in \repp(K_r,1)$ such that $\cF \cong \TilTheta(M)$, and $\TilTheta : \repp(K_r,1) \lra \StVect(\PP(A_r))$ is an equivalence of categories.
\item Let $M \in \repp(K_r,1)$. Then $\rk(\TilTheta(M))= \Delta_M$ and  $c_1(\TilTheta(M)) = \dim_\KK M_1$,
where $c_1(\TilTheta(M))$ denotes the first Chern class of  $\TilTheta(M)$. 
\end{enumerate}
\end{Theorem}

\bigskip

\begin{example}\label{1.3.4}
The \textit{Euler sequence} $0 \lra \cO_{\PP(A_r)}(-1) \lra \cO_{\PP(A_r)}^r \lra \cT_{\PP(A_r)}(-1) \lra 0$  (cf. \cite[(6.4.2)]{Ben16}) shows that the $-1$-twist of the \textit{tangent bundle} $\cT_{\PP(A_r)}$ is a Steiner bundle of rank $r-1$ with first Chern class $1$. According to \cref{1.3.3}, there exists an indecomposable representation $X \in \rep_{\proj}(K_r,1)$ with $\dimu X = (1,r)$ such that $\widetilde{\Theta}(X)  \cong \cT_{\PP(A_r)}(-1)$. A straightforward computation now shows $X \cong P_1(r)$.
\end{example}

\bigskip

\subsection{Auslander-Reiten theory}

Let $r \geq 2$. We denote by $\Gamma(K_r)$ the \textit{Auslander-Reiten quiver} of $K_r$. Its vertices correspond to the isomorphism classes of indecomposable representations, and the arrows represent the so-called \textit{irreducible} morphisms. By abuse of notation, we write $M$ instead of $[M]$ for an indecomposable $M \in \rep(K_r)$, and identify $M$ with its isomorphism class.

In the following, we recall the basic definitions and results needed later, and refer the reader to \cite{Ker94} for further details and unexplained terminology.

Let $r \geq 2$. We denote by $\tau_{K_r} \colon \rep(K_r) \lra \rep(K_r)$ the \textit{Auslander-Reiten translation}. The Auslander-Reiten quiver $\Gamma(K_r)$ of $K_r$ consists of infinitely many components. We denote by $\cP$ and $\cI$ the uniquely determined components of the Auslander-Reiten quiver $\Gamma(K_r)$ containing $S(2) = P_0(r)$ and $S(1)$, respectively. The components $\cP,\cI$ are called \textit{preprojective} and  \textit{preinjective} component, respectively. All other components are \textit{regular}. The (indecomposable) representations in $\cP$ and $\cI$ are called \textit{preprojective} and $\cI$ \textit{preinjective}, respectively, while representations in a regular component are called \textit{regular}. An arbitrary non-zero representation $M \in \rep(K_r)$ is called \textit{preprojective, preinjective} or \textit{regular}, provided all its indecomposable direct summands are preprojective, preinjective or regular, respectively. By definition, the zero representation is preprojective, preinjective and regular.
Recall that the representations
\[ P_{2\ell}(r) \coloneqq \tau^{-\ell}_{K_r}(P_0(r)) \ \text{and} \  P_{2\ell+1}(r) \coloneqq \tau^{-\ell}_{K_r}(P_1(r))\]
for all $\ell \in \NN_0$ form a complete list of representatives of the isomorphism classes of indecomposable preprojective Kronecker representations and for every $i \in \NN_0$ there is an almost split sequence
\[ 0 \lra P_i(r) \lra rP_{i+1}(r) \lra P_{i+2}(r) \lra 0.\]

By the same token, a complete list of representatives of the isomorphism classes of indecomposable preinjective representations is given by $I_i(r) \coloneqq D_{K_r}(P_i(r))$, $i \in \NN_0$, where $D_{K_r} \colon \rep(K_r) \lra \rep(K_r)$ is the standard duality, defined by $D_{K_r}(M) = (M_2^\ast,M_1^\ast,\psi_{D_{K_r}(M)})$ with structure map
\[ \psi_{D_{K_r}(M)}(a \otimes h) \coloneqq h \circ \psi_M(a \otimes -) \colon M_1 \lra \KK \]
for all $a \in A_r$ and $h \in M_2^\ast$. We obtain almost split sequences
\[ 0 \lra I_{i+2}(r) \lra rI_{i+1}(r) \lra I_{i}(r) \lra 0\]
for all $i \in \NN_0$. Note that $I_0(r) = S(1)$ and $I_1(r)$ is the injective hull of $P_0(r) = S(2)$. 
Figure \ref{Fig:1} illustrates the Auslander-Reiten quiver of $K_r$.
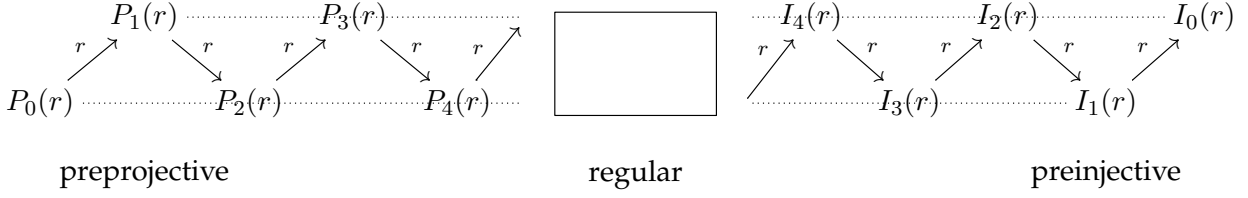
\begin{figure}[ht]
    \centering
  \[ 
\xymatrix @C=7pt@R=15pt{
& P_1(r) \ar^{r}[dr] \ar@{..}[rrrr]&  & P_3(r) \ar^{r}[dr] &   & &   & &  &  &  & \ar@{..}[rrrrr] & I_4(r) \ar^{r}[dr]  & & I_2(r) \ar^{r}[dr]& &  I_0(r)\\
P_0(r)   \ar@{..}[rrrrr] \ar^{r}[ur]&  & P_2(r) \ar^{r}[ur] & & P_4(r) \ar^{r}[ur]  & &  & & & & & \ar^{r}[ur] \ar@{..}[rrrr]  & & I_3(r) \ar^{r}[ur] & & I_1(r) \ar^{r}[ur] \\
& & & & & & & & & & & & & & & &
\save "1,7"."2,11"*[F]\frm{} \restore  
\save "3,2"."3,7" *\txt{preprojective} \restore
\save "3,9"."3,10" *\txt{regular} \restore
\save "3,16"."3,17" *\txt{preinjective} \restore}
\] 
    \caption{Auslander-Reiten quiver $\Gamma(K_r)$ for $r \geq 2$.}
    \label{Fig:1}
\end{figure}

Recall from \cite[(VIII.2.7)]{ASS06} that every indecomposable preprojective (preinjective) representation $M \in \rep(K_r)$ is a \textit{brick}\footnote{The notion \textit{brick} is special to the representation theory of finite dimensional algebras. A more common notion in other fields of algebra is that of a \textit{Schurian} representation. A vector bundle with the corresponding property is called \textit{simple}.}, i.e., $\End_{K_r}(M) = \KK \id_{M}$. 

For $r = 2$, every regular component is a homogeneous tube \cite[(VIII.7)]{ARS95},  
whereas for $r \geq 3$ regular components are of type $\ZZ A_{\infty}$ by \cite{Rin78}.  
The shape of such a component is illustrated in Figure~\ref{Fig:2}. Note that these components are only bounded at the "bottom".

\begin{figure*}[!ht]
\centering 
\begin{tikzpicture}[very thick, scale=1]
                    [every node/.style={fill, circle, inner sep = 1pt}]

\def \n {9} % #Knoten Reihe  - 1
\def \m {2} % #Knoten Spalte - 1
\def \translation {1} % 1 Für Translation

\def \ab {0.15} % Abstand Pfeil und Knoten
\def \Pab {0.6} % Halber Abstand Horizontal

\def \rcone {0} % 1 für rechten Kegel
\def \rdist {2} % Anzahl der quasi-einfachen die eingeschlossen werden - 1
\def \rcolor {gray} %  white für keine Farbe

\def \rrcone {0} %1 für einen zweiten rechten Kegel links von rcone
\def \rrdist {3} % Anzahl der quasi-einfachen die eingeschlossen werden - 1

\foreach \a in {0,...,\n}{
\foreach \b in {0,...,\m}{
  
   \ifthenelse{\a = \n \and \b < \m}{
   \node[color=black] ({\a,\b,5})at ({\a*2*\Pab},{\b*2*\Pab}) {$\circ$};
     }
     {
      \ifthenelse{\b = \m \and \a < \n}{
      \node[color=black] ({\a,\b}) at ({\a*2*\Pab+\Pab},{\b*2*\Pab+\Pab}) {$\circ$};
      \node[color=black] ({\a,\b,5})at ({\a*2*\Pab},{\b*2*\Pab}) {$\circ$};
      }
      {
    
     \ifthenelse{\b = \m \and \a = \n}
     {\node[color=black] ({\a,\b,5})at ({\a*2*\Pab},{\b*2*\Pab}) {$\circ$};}
    {\node[color=black] ({\a,\b}) at ({\a*2*\Pab+\Pab},{\b*2*\Pab+\Pab}) {$\circ$};
    \node[color=black] ({\a,\b,5})at ({\a*2*\Pab},{\b*2*\Pab}) {$\circ$};

      }
      }
      }
    }
    }

\foreach \s in {0,...,\n}{
\foreach \t in {0,...,\m}
{  
 \ifthenelse{\t = \m \and \s < \n}{
    \draw[->] (\s*2*\Pab+\ab,\t*2*\Pab+\ab) to (\s*2*\Pab+\Pab-\ab,\t*2*\Pab+\Pab-\ab); 
    \draw[->] (\s*2*\Pab+\Pab+\ab,\t*2*\Pab+\Pab-\ab) to (\s*2*\Pab+2*\Pab-\ab,\t*2*\Pab+\ab); 

  }{
  
  \ifthenelse{\s = \n \and \t < \m}{
  
  }
  {
  \ifthenelse{\s = \n \and \t = \m}{
   
  }{
   \draw[->] (\s*2*\Pab+\ab,\t*2*\Pab+\ab) to (\s*2*\Pab+\Pab-\ab,\t*2*\Pab+\Pab-\ab); 
   \draw[->] (\s*2*\Pab+\Pab+\ab,\t*2*\Pab+\Pab+\ab) to (\s*2*\Pab+2*\Pab-\ab,\t*2*\Pab+2*\Pab-\ab);
   \draw[->] (\s*2*\Pab+\ab,\t*2*\Pab+2*\Pab-\ab) to (\s*2*\Pab+\Pab-\ab,\t*2*\Pab+\Pab+\ab); 
   \draw[->] (\s*2*\Pab+\Pab+\ab,\t*2*\Pab+\Pab-\ab) to (\s*2*\Pab+2*\Pab-\ab,\t*2*\Pab+\ab);    
   }
   }
  
    }
    }
    }

\ifthenelse{\isodd{\m}}
%% IF
 { 
  \node[color=black] (Dots1) at (0,\m*\Pab+2*\Pab) {$\cdots$};
  \node[color=black] (Dots2) at (\n*2*\Pab,\m*\Pab+2*\Pab) {$\cdots$};
   \ifthenelse{\isodd{\n}}{
  \node[color=black] (Dots3) at (0.5*\n*2*\Pab,2*\m*\Pab+2*\Pab) {$\vdots$};}
  {\node[color=black] (Dots3) at (0.5*\n*2*\Pab,2*\m*\Pab+\Pab) {$\vdots$};} 
  }
%% Else
  {
  \node[color=black] (Dots1) at (0,\m*\Pab+\Pab) {$\cdots$};
  \node[color=black] (Dots2) at (\n*2*\Pab,\m*\Pab+\Pab) {$\cdots$};
  \ifthenelse{\isodd{\n}}{
  \node[color=black] (Dots3) at (0.5*\n*2*\Pab,2*\m*\Pab+2*\Pab) {$\vdots$};}
  {\node[color=black] (Dots3) at (0.5*\n*2*\Pab,2*\m*\Pab+\Pab) {$\vdots$};}
  }
 
\ifthenelse{\translation = 1}{
   \foreach \s in {0,...,\n}{
   \foreach \t in {0,...,\m}{ 
   \ifthenelse{\s = 0}{}{
      \ifthenelse{\s = \n}{\draw[->,dotted,thin] (\s*2*\Pab-\ab,\t*2*\Pab) to (\s*2*\Pab-2*\Pab+\ab,\t*2*\Pab); }{
   \draw[->,dotted,thin] (\s*2*\Pab-\ab,\t*2*\Pab) to (\s*2*\Pab-2*\Pab+\ab,\t*2*\Pab); 
   \draw[->,dotted,thin] (\s*2*\Pab-\ab+\Pab,\t*2*\Pab+\Pab) to (\s*2*\Pab-2*\Pab+\Pab+\ab,\t*2*\Pab+\Pab); 
   }
   }}
}}
{}  %ELSE

\begin{scope}[on background layer]

\ifthenelse{\rrcone = 1}{
        \draw[fill = \rcolor!20] (\n*\Pab*2+\ab,-\ab) node[anchor=north]{}
  -- (\n*\Pab*2-\rrdist*\Pab*2-0.7*\Pab,-\ab) node[anchor=north]{}
  -- (\n*\Pab*2+\ab,\rrdist*\Pab*2+0.7*\Pab) node[anchor=south]{};
    }
  {}
\ifthenelse{\rcone = 1}{
        \draw[fill= \rcolor!40](\n*\Pab*2+\ab,-\ab) node[anchor=north]{}
  -- (\n*\Pab*2-\rdist*\Pab*2-0.7*\Pab,-\ab) node[anchor=north]{}
  -- (\n*\Pab*2+\ab,\rdist*\Pab*2+0.7*\Pab) node[anchor=south]{};
    }
  {}   
  
\end{scope}
\end{tikzpicture}
\caption{Regular component of $\Gamma(K_r)$ for $r \geq 3$.}
\label{Fig:2}
\end{figure*}
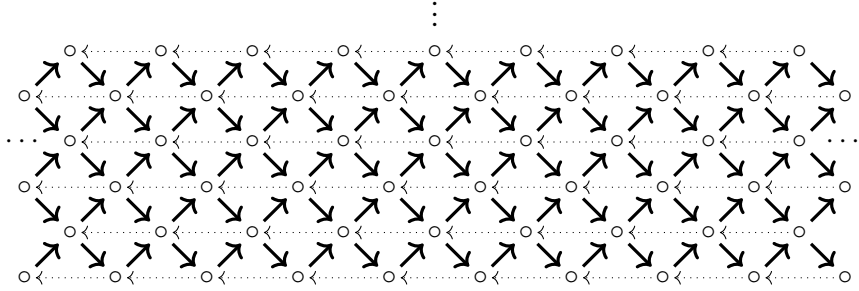

A representation $M$ in a component $\cC$ of type $\ZZ A_{\infty}$ is called \textit{quasi-simple} whenever it has precisely one direct predecessor.  
Equivalently, the quasi-simple representations are those located at the bottom of the component.

The irreducible morphism given by an arrow $M \to N$ in a regular component is injective if the corresponding arrow is uprising and surjective otherwise. 
For each $M \in \cC$, there is therefore a unique directed path
\[ M_{[1]} \rightarrow M_{[2]} \rightarrow \cdots \rightarrow M_{[n-1]} \rightarrow M_{[n]}=M\]
in $\cC$ such that $M_{[1]}$ is quasi-simple and for every $i \in \{1,\ldots,n-1\}$ the irreducible morphism $M_{[i]} \lra M_{[i+1]}$ is injective. In this case, $\ql(M)\coloneqq n$ is called the \textit{quasi-length} of $M$.

\bigskip

\subsection{Homogeneous and uniform Steiner bundles}

\label{S:1.5}
Given $g \in \GL(A_r)$ and $M \in \rep(K_r)$, we define the representation $g.M = (M_1,M_2,\psi_{g.M}) \in \rep(K_r)$ with structure map
\[ \psi_{g.M} \coloneqq   \psi_M \circ (g^{-1} \otimes \id_{M_1}) \colon A_r \otimes_{\KK} M_1 \lra M_2.\]
Moreover, we define $g.f \coloneqq  f$ for every morphism $f \in \Hom_{K_r}(M,N)$. The group $\GL(A_r)$ acts in this way on $\rep(K_r)$  via auto-equivalences. 

\bigskip

\begin{Definition}\label{1.5.1}
We call $M \in \rep(K_r)$
\begin{enumerate}
    \item \textit{uniform}, provided $\alpha^\ast(M) \cong \beta^\ast(M)$ for all $\alpha,\beta \in \Inj_{\KK}(A_2,A_r)$.
    \item \textit{homogeneous}, provided $g.M \cong M$ for all $g \in \GL(A_r)$.
\end{enumerate}
\end{Definition}

\bigskip

\begin{Remarks}\label{1.5.2} \phantom{.}
\begin{enumerate}
\item For $g \in \GL(A_r) = \Inj_{\KK}(A_r,A_r)$ we have an isomorphism $g.M \cong (g^{-1})^\ast(M)$.
\item Since $\GL(A_r)$ acts transitively on $\Gr_d(A_r)$, every homogeneous representation satisfies $\rk(\psi_{M,\fv}) = \rk(\psi_{M,\fw})$ for all $1 \leq d < r$ and all $\fv,\fw\in \Gr_d(A_r)$. In particular, the rank of $\sum^r_{i=1} \alpha_i M(\gamma_i) \colon M_1 \lra M_2$ does not depend on $\alpha \in \KK^r\setminus \{0\}$.
\item If $M$ is uniform, then we either have $\cV(K_r,2)_M=\emptyset$ or $\cV(K_r,2)_M=\Gr_2(A_r)$.
\end{enumerate}
\end{Remarks}

\bigskip

\begin{Lemma} Let $M \in \rep(K_r)$ be a representation. \label{1.5.3}
\begin{enumerate}
    \item For $\alpha \in \Inj_\KK(A_2,A_r)$ we have $\TilTheta(\alpha^\ast(M)) \cong \hat{\alpha}^\ast(\TilTheta(M))$\footnote{Note that $\widetilde{\Theta}$ has two different meanings.}. 
\item Let $M \in \rep_{\proj}(K_r,1)$.
\begin{enumerate}
    \item[(i)] We have $\alpha^\ast(M) \cong \beta^\ast(M)$ for all $\alpha,\beta \in \Inj_{\KK}(A_2,A_r)$ such that $\im \alpha = \im \beta$.
    \item[(ii)] The representation $M$ is homogeneous if and only if $\TilTheta(M) \in \StVect(\PP(A_r))$ is homogeneous.
    \item[(iii)] The representation $M$ is uniform if and only if $\TilTheta(M) \in \StVect(\PP(A_r))$ is uniform.

\end{enumerate}
\end{enumerate}
\end{Lemma}
\begin{proof}
\begin{enumerate}
    \item See \cite[(5.1.1)]{BF24}.
    \item Statement (i) is \cite[(1.3.1)]{BF24}. Statements (ii) and (iii) follow from (1) in conjunction with \cref{1.3.3} since $\alpha^\ast(M) \in \rep_{\proj}(K_r,d)$ for every $d \in \{2,r\}$ and every $\alpha \in \Inj_{\KK}(A_d,A_r)$.
    \end{enumerate}
\end{proof}
 
\bigskip

Having established the correspondence between uniform (homogeneous) Steiner bundles and uniform (homogeneous) representations in $\rep_{\proj}(K_r,1)$, we now proceed to examine how the splitting type of a Steiner bundle is reflected within the category $\rep(K_2)$.

Let $M \in \rep_{\proj}(K_r,1)$ and $\alpha \in \Inj_{\KK}(A_2,A_r)$. Then $\alpha^\ast(M) \in \rep_{\proj}(K_2,1)$ and \cite[(4.3)]{Wor13b} implies that $\alpha^\ast(M)$ is preprojective. Hence, we find a (uniquely determined) decomposition 
\[ \alpha^\ast(M) \cong \bigoplus_{i \in \NN_0} b_i(\alpha,M) P_i(2).\]
In view of the foregoing result, we have $\alpha^\ast(M) \cong \beta^\ast(M)$ for every $\beta \in \Inj_{\KK}(A_2,A_r)$ such that $\fv := \im \alpha = \im \beta$ and obtain therefore a well-defined sequence
\[ (b_i(\fv,M))_{i \in \NN_0} := (b_i(\alpha,M))_{i \in \NN_0}.\]
Accordingly, we put
\[
M|_{\fv} := \bigoplus_{i \in \NN_0} b_i(\fv,M)P_i(2) \in \rep(K_2).
\]

Let $i \in \NN_0$, then \cref{1.3.3} gives us $\rk(\widetilde{\Theta}(P_i(2))) = 1$ and $c_1(\widetilde{\Theta}(P_i(2))) = i$, which implies $\TilTheta(P_i(2)) \cong \cO_{\PP(A_2)}(i)$ (cf. \cite[(6.3.1), (7.3.7)]{Ben16}). Given $\fv \in \Gr_2(A_r)$ and $\alpha \in \Inj_{\KK}(A_2,A_r)$ such that $\im \alpha = \fv$, we conclude with \cref{1.5.3}:
\[
    \bigoplus_{i \in \NN_0} b_i(\alpha,M) \cO_{\PP(A_2)}(i) \cong \TilTheta({\alpha}^\ast(M)) \cong \hat{\alpha}^\ast(\TilTheta(M)) 
   = \bigoplus_{i \in \ZZ} b_i(\alpha,\TilTheta(M)) \cO_{\PP(A_2)}(i).
\]
In summary, we have:

\bigskip

\begin{Lemma}\label{1.5.4}
Let $M \in \rep_{\proj}(K_r,1)$, $\alpha \in \Inj_{\KK}(A_2,A_r)$ and $\fv \coloneqq \im \alpha$.
\begin{enumerate}
 \item We have $b_i(\fv,\TilTheta(M)) = b_i(\alpha,\TilTheta(M)) = 0$ for all $i < 0$.
    \item We have $b_i(\fv,M) = b_i(\alpha,M) = b_i(\alpha,\TilTheta(M)) = b_i(\fv,\TilTheta(M))$ for all $i \in \NN_0$.
    \item We have $O_{\Theta(M)} = \{ \fv \in \Gr_2(A_r) \mid M|_{\fv} = \bigoplus_{i \in \NN_0} b_i(\widetilde{\Theta}(M)) P_i(2)\}$. %or $b_i(M) \coloneqq b_i(\widetilde{\Theta}(M))$,
   % \item We have $b_i(\alpha,M) = b_i(\beta,M)$ for all $\beta \in \Inj_{\KK}(A_d,A_r)$ with $\im \beta = \fv$.
\end{enumerate}
\end{Lemma}

\bigskip

In view of \cref{1.5.4}, for $M \in \rep_{\proj}(K_r,1)$ we set 
\[
O_{M} := O_{\widetilde{\Theta}(M)}, \qquad 
\cJ_{M} := \cJ_{\widetilde{\Theta}(M)}, \qquad 
b_i(M) := b_i(\widetilde{\Theta}(M)).
\]
We then call
\[
M_{\gen} \coloneqq \bigoplus_{i \in \NN_0} b_i(M) P_i(2)
\]
the \emph{generic decomposition} (or \emph{splitting type}) of $M$. If $M$ is uniform, we just write 
\[ M|_{K_2} = \bigoplus_{i \in \NN_0} b_i(M) P_i(2)\]
to indicate that $M$ has  splitting type $\bigoplus_{i \in \NN_0} b_i(M) P_i(2)$ with $\cJ_M = \emptyset$. 

\bigskip

The final result of this section shows that the category $\rep_{\proj}(K_r,2)$ which is equivalent to the category of Steiner bundles on $\Gr_2(A_r)$ (cf. \cite[(3.2.3)]{BF24}), can also be interpreted as the category of uniform Steiner bundles on $\PP(A_r)$ with support $\{0,1\}$.

\bigskip

\begin{proposition}\label{1.5.5}
Let $r \geq 3$, $\cF$ be a Steiner bundle on $\PP(A_r)$ and $M \in \rep_{\proj}(K_r,1)$ be such that $\cF \cong \widetilde{\Theta}(M)$. The following statements are equivalent.
\begin{enumerate}
    \item The bundle $\cF$ is uniform with $\supp(\cF) = \{0,1\}$ (resp. $\supp(\cF) = \{0\}$).
    \item The representation $M$ is in $\rep_{\proj}(K_r,2)$ and $M_1 \neq \{0\}$ (resp. $M_1 = \{0\}$).
\end{enumerate}
In this case, $\cF|_{\PP(A_2)} = \Delta_M(2) \cO_{\PP(A_2)} \oplus (\dim_{\KK} M_1) \cO_{\PP(A_2)}(1)$.
\end{proposition}
\begin{proof}
We only consider the case $M_1\neq \{0\}$. 

(1) $\Rightarrow$ (2). We have $\cF|_{\PP(A_2)} = b_0(\cF) \cO_{\PP(A_2)} \oplus b_1(\cF) \cO_{\PP(A_2)}(1)$ with $b_0(\cF) \neq 0 \neq b_1(\cF)$ and conclude
\[ M|_{K_2} = b_0(\cF) P_0(2) \oplus b_1(\cF) P_1(2).\]
By definition, we therefore have $\alpha^\ast(M) \cong b_0(\cF) P_0(2) \oplus b_1(\cF) P_1(2)$ for all $\alpha \in \Inj_{\KK}(A_2,A_r)$.  
Since $b_0(\cF) P_0(2) \oplus b_1(\cF) P_1(2)$ is projective, the characterization of $\rep_{\proj}(K_r,2)$ following \cref{1.2.2} yields $M \in \rep_{\proj}(K_r,2)$ and $0 \neq b_1(\cF) = \dim_{\KK} M_1$.

(2) $\Rightarrow$ (1). Since $M \in \rep_{\proj}(K_r,2)$, we have
\[ \alpha^\ast(M) \cong \Delta_M(2) P_0(2) \oplus (\dim_{\KK} M_1) P_1(2)\]
for all $\alpha \in \Inj_{\KK}(A_2,A_r)$. \cref{1.5.4} implies $O_{\widetilde{\Theta}(M)} = \Gr_2(A_r)$ with $b_0(\widetilde{\Theta}(M)) = \Delta_M(2)$, $b_1(\widetilde{\Theta}(M)) = \dim_{\KK} M_1$, and $b_i(\widetilde{\Theta}(M)) = 0$ for all other $i$. Moreover, $M \in \rep_{\proj}(K_r,2)$ and \cite[(2.3.2)]{BF24} give $\Delta_M(2) \geq (r-2)\min\{2,\dim_{\KK} M_1\} > 0$. 
Hence, $\supp(\cF) = \{0,1\}$.
\end{proof}
      
\bigskip

\section{Homogeneous and uniform representations}

Let $M \in \rep(K_r)$ be homogeneous and $\alpha,\beta \in \Inj_{\KK}(A_2,A_r)$. We find $g \in \GL(A_r)$ such that $\alpha = g \circ \beta$ and conclude
\[ \alpha^{\ast}(M) = (g \circ \beta)^{\ast}(M) \cong \beta^\ast(g^\ast(M)) \cong \beta^\ast(g^{-1}.M) \cong \beta^\ast(M).\]
Hence, homogeneous representations are uniform. Based on the equivalence $\TilTheta \colon \rep_{\proj}(K_r,1) \lra \StVect(\PP(A_r))$, \cref{1.5.3} and the findings in \cite{MMR21}, it is already clear that not all uniform representations are homogeneous. 

In what follows, we show that the examples constructed in \cite{MMR21}, which provide uniform but non-homogeneous representations, are not isolated phenomena. In fact, we establish the existence of a broad class of such representations. This is achieved by proving that homogeneous and uniform representations exhibit different behavior with respect to Auslander-Reiten theory and within the variety of representations.

Throughout this section we assume that $(V_1,V_2)$ is a pair of finite-dimensional $\KK$-vector spaces such that $V_1 \oplus V_2 \neq \{0\}$. 

\subsection{The variety of representations and Kac's Theorem}
We denote by
\[\cV(K_r;V_1,V_2) \coloneqq \Hom_\bmk(A_r \otimes_{\KK} V_1,V_2)\]
the (irreducible) affine variety of representations of $K_r$ on $(V_1,V_2)$. A point $\psi$ in the variety $\cV(K_r;V_1,V_2)$ can be identified with the representation $V_{\psi} = (V_1,V_2,\psi \colon A_r \otimes_{\KK} V_1 \lra V_2)\in   \cK_r \cong \rep(K_r)$. In the following we will freely use this identification. 

The \textit{Euler-Ringel form} of $K_r$ is given by
\[ \langle - , - \rangle_{r} \colon \ZZ^2 \times \ZZ^2 \lra \ZZ \ ; \ (x,y) \mapsto x_1 y_1 + x_2 y_2 - r x_1 y_2 \]
and satisfies $\langle \udim M,\udim N\rangle_r =\dim_\KK\Hom_{K_r}(M,N)-\dim_\KK\Ext^1_{K_r}(M,N)$ for all $M,N \in \rep(K_r)$. The corresponding \textit{Tits quadratic} form is denoted by $q_r \colon \ZZ^2 \lra \ZZ \ ; \ x \mapsto \langle x,x \rangle_r = x_1^2 + x_2^2 - r x_1 x_2$.  

According to \cite[(Thm B)]{Kac82}, an indecomposable representation $M \in \rep(K_r)$ satisfies $q_r(\dimu M) \leq 1$. An element $(x,y) \in \NN_0^2$ is called \textit{Schur root}, provided there exists a brick $M \in \rep(K_r)$ such that $\dimu M = (x,y)$. 
We denote by
\[\cB(V_1,V_2) := \{ \psi \in \cV(K_r;V_1,V_2) \mid V_{\psi} \ \text{is a brick}\} \subseteq \cV(K_r;V_1,V_2)\]
the open subset of bricks in $\cV(K_r;V_1,V_2)$.
Note that $\dimu(V_1,V_2)$ is a Schur root if and only if $\cB(V_1,V_2) \neq \emptyset$. 
Kac’s Theorem provides a complete answer to the question of whether $(x,y)$ is a Schur root (see \cite[(6(c), p.159)]{Kac82}). We present a simplified version sufficient for the present discussion.

\bigskip

\begin{Theorem} \label{2.1.1}
Let $r \geq 2$. The following statements hold.
\begin{enumerate}
\item[(1)] The subset $\cB(V_1,V_2)$
is non-empty if and only if $q_r(\dimu(V_1,V_2)) \leq 1$.
\item[(2)] The following statements are equivalent.
\begin{enumerate}
\item[(i)] $\cV(K_r;V_1,V_2)$ contains a regular representation.
\item[(ii)]  Every indecomposable representation in $\cV(K_r;V_1,V_2)$ is regular.
\item[(iii)]  $q_r(\dimu(V_1,V_2)) \leq 0$.
\end{enumerate}
\item[(3)] If $q_{r}(\dimu(V_1,V_2)) = 1$ and $M \in \cV(K_r;V_1,V_2)$ is indecomposable, then there exists a unique $i \in \NN_0$ such that $M \cong P_i(r)$ or $M \cong I_i(r)$.
\end{enumerate}
\end{Theorem}

\bigskip

As a consequence of \cref{2.1.1}, we refer to any $(x,y) \in \NN^2$ satisfying $q_r(x,y) \leq 0$ as a \textit{regular dimension vector}.

\bigskip

\subsection{Homogeneous representations}

This section is devoted to the study of homogeneous representations in $\rep(K_r)$ for $r \geq 2$.  
To set the stage, we begin with a general structural result that will play a key role in our analysis.  
Although the statement holds in greater generality for any associative algebra and any connected algebraic group \cite[(2.1), (2.2)]{Far11}, we include a complete proof in our setting for the reader’s convenience.

Recall that the algebraic group
\[ G(A_r;V_1,V_2) := \GL(A_r) \times \GL(V_2) \times \GL(V_1)\]
acts on the variety $\cV(K_r;V_1,V_2)$ via
\[ (a,g_2,g_1).\psi := g_2 \circ \psi \circ (a^{-1} \otimes g_1^{-1})\]
for all $a \in A_r, g_1 \in \GL(V_1), g_2 \in \GL(V_2)$, and that $\psi,\psi'$ are in the same orbit under this action if and only if there exists $g \in \GL(A_r)$ such that $g.V_{\psi} \cong V_{\psi'}.$

\bigskip

\begin{Lemma}\phantom{}\label{2.2.1}
\begin{enumerate}
\item Let $M \in \rep(K_r)$, then $\GL(A_r)_M := \{ g \in \GL(A_r) \mid g.M \cong M\}$ is a closed subgroup of $\GL(A_r)$.
\item Homogeneous representations are closed under direct sums and summands.
\item A regular Auslander-Reiten component $\cC$ contains a homogeneous representation if and only if every representation in $\cC$ is homogeneous.
\item All preprojective and preinjective representations are homogeneous.
\item Every non-zero regular representation $M \in \rep(K_2)$ is not homogeneous.
\end{enumerate}

\end{Lemma}
\begin{proof}
\begin{enumerate}
    \item We set $G := G(A_r;M_2,M_1)$. Restriction of the canonical projection $\pi \colon G \lra \GL(A_r)$ yields the morphism of algebraic groups
    \[ \pi_M \colon \Stab_{G}(\psi_M) \lra \GL(A_r) \ ; \ (g,a_2,a_1) \mapsto g.\]
    Note that $\GL(A_r)_M = \im \pi_M \subseteq \GL(A_r)$ is closed, cf. \cite[(2.2.5)]{Spr98}.
    \item Clearly, homogeneity is closed under direct sums. Let $X_1,\ldots,X_n \in \rep(K_r)$ be indecomposable and assume that $M := \bigoplus^n_{i=1} X_i$ is homogeneous. We have $\bigoplus^n_{i=1} g.X_i \cong g.M \cong M \cong \bigoplus^n_{i=1} X_i$ for all $g \in G$. By the Theorem of Krull-Remak-Schmidt, we therefore obtain an action $\GL(A_r)$ on the set of isomorphism classes $\{[X_1],\ldots,[X_n]\}$ with $\Stab_{\GL(A_r)}([X_i]) = \GL(A_r)_{X_i}$. Since $|\GL(A_r)/\Stab_{\GL(A_r)}([X_i])| = |\GL(A_r).[X_i]|$ is finite and $\GL(A_r)$ is connected, we conclude with \cite[(2.2.1)(iii)]{Spr98} that $\Stab_{\GL(A_r)}([X_i]) = \GL(A_r)$. Hence, $g.X_i \cong X_i$ for all $g \in \GL(A_r)$ and every $i \in \{1,\ldots,n\}$.
  \item Let $\cC$ be a regular component of $K_r$ and assume that $\cC$ contains a homogeneous representation. Hence, $\cD := \{ X \in \cC \mid X \ \text{is homogeneous}\}$
is non-empty. Let $Y \in \cC$. Since $Y$ is not injective, there exists an almost split sequence 
\[
0 \longrightarrow \tau_{K_r}(Y) \longrightarrow E \longrightarrow Y \longrightarrow 0.
\] 
As $Y$ is homogeneous and $\GL(A_r)$ acts on $\rep(K_r)$ by auto-equivalences, we obtain for each $g \in \GL(A_r)$ an almost split sequence 
\[
0 \longrightarrow g.\tau_{K_r}(Y) \longrightarrow g.E \longrightarrow Y \longrightarrow 0.
\] 
Uniqueness of almost split sequences ending in $Y$ \cite[(V.1.16)]{ARS95} implies that $\tau_{K_r}(Y)$ and $E$ are homogeneous. In particular, $\cD$ is closed under the $\tau_{K_r}$. Moreover,  (2) implies that $\cD$ is closed under predecessors.
By the same token, $\cD$ is closed under successors in $\cC$ and $\tau_{K_r}^{-1}$. Since $\cC$ is a connected translation quiver, we conclude $\cC = \cD$.
\item In view of (2), it suffices to prove the statement for every indecomposable representations. Since every representation $P_i(r)$ in the preprojective component is a successor of $S(2) = P_0(r)$ and $P_0(r) = S(2)$ is homogeneous, we can apply the arguments given in (3). The proof for $I_i(r) \in \cI$ is similar. We remark that the result also follows from \cref{2.1.1}.
\item It is well-known (see for example \cite[(VIII.7.4)]{ARS95}) that every regular component of $K_2$ contains an indecomposable representation with dimension vector $(1,1)$. It is easy to see that such representations can not be homogeneous. Now we apply (2) and (3). 
\end{enumerate}
\end{proof}

\bigskip

In contrast to the tame case $r = 2$, there exist regular, indecomposable representations (even bricks) that are homogeneous. Explicit examples can be found in~\cite[(2.12), (3.22), (3.24), (4.3)]{Wor13b}. However, these examples are exceptional: they do not reflect the generic behavior of regular representations in the variety of representations. The goal of the remainder of this section is to establish the following theorem, which shows that a general representation of a fixed regular dimension vector is a non-homogeneous brick.

\bigskip

\begin{Theorem}\label{2.2.2}
Let $r \geq 3$. The following statements are equivalent.
\begin{enumerate}
    \item $q_r(\dimu(V_1,V_2)) \leq 0$.
    \item $\cV(K_r;V_1,V_2)$ contains a non-empty open set consisting of bricks that are non-homogeneous.
    \item $\cV(K_r;V_1,V_2)$ contains a non-empty open set consisting of indecomposable representations that are non-homogeneous.
\end{enumerate}
\end{Theorem}

\bigskip
\noindent The following result plays a central role in the proof of \cref{2.2.2}.
\bigskip

\begin{proposition}\label{2.2.3}
Let $\psi \in \cV(K_r;V_1,V_2)$ be such that $V_\psi$ is a brick and set  $G := G(A_r;V_2,V_1)$. We have $\dim \Stab_G(\psi) \leq r^2+1$ and the following statements are equivalent.
\begin{enumerate}
    \item $V_\psi$ is homogeneous.
    \item For every $g \in \GL(A_r)$, there is $(a_{g,2},a_{g,1}) \in \GL(V_2) \times \GL(V_1)$ such that 
    \[\Stab_G(\psi) = \{ (g,\lambda a_{g,2}, \lambda a_{g,1}) \mid \lambda \in \KK^\times, g \in \GL(A_r)\}.\]
    \item $\dim \Stab_G(\psi) = r^2+1$.
\end{enumerate}
\end{proposition}
\begin{proof} 
As in \cref{2.2.1}, we consider the morphism of algebraic groups
\[ \pi_{\psi} \colon \Stab_G(\psi) \lra \GL(A_r) \ ; \ (g,a_2,a_1) \mapsto g.\] 
According to \cite[(2.2.5)]{Spr98}, $\im  \pi_{\psi}$ is a closed subgroup of $\GL(A_r)$. Let $g \in \im \pi_{\psi}$. By definition of $\pi_{\psi}$, we find $(a_{g,2},a_{g,1}) \in \GL(V_2) \times \GL(V_1)$ such that $(g,a_{g,2},a_{g,1}) \in \Stab_G(\psi)$. For $(g,b_2,b_1) \in \Stab_G(\psi)$ we have $(g^{-1},b_2^{-1},b_1^{-1}) \in \Stab_G(\psi)$ and conclude
 \begin{align*}
  \psi \circ (\id_{A_r} \otimes a_{g,1}^{-1} \circ b_1) &= \psi \circ (g^{-1} \otimes a_{g,1}^{-1}) \circ (g \otimes b_1) \\
  &=a_{g,2}^{-1} \circ \psi \circ (g \otimes b_1) \\
  &= a_{g,2}^{-1} \circ b_2 \circ \psi.
 \end{align*}
 Hence, $(a_{g,1}^{-1} \circ b_1,a_{g,2}^{-1} \circ b_2) \in \End_{K_r}(V_{\psi}) \setminus \{0\}$ and $V_{\psi}$ being a brick implies the existence of $\lambda \in \KK^\times$ such that $(b_2,b_1) = (\lambda  a_{g,2},\lambda a_{g,1})$.
 In conclusion, we find a map 
 \[ \im \pi_{\psi} \lra \GL(V_2) \times \GL(V_1) \ ; \ g \mapsto  (a_{g,2},a_{g,1})\] 
 such that
\[ \Stab_G(\psi) = \{ (g,\lambda a_{g,2}, \lambda a_{g,1}) \mid \lambda \in \KK^\times, g \in \im  \pi_{\psi}\}.\]
Let $Y_1,\ldots,Y_n$ be the irreducible components of $\im  \pi_{\psi}$. We obtain surjective morphisms
\[ \varphi_i :=  \pi_{\psi}|_{ \pi_{\psi}^{-1}(Y_i)} \colon  \pi_{\psi}^{-1}(Y_i) \lra Y_i \ ; \ (g,a_2,a_1) \mapsto g.\]
Since $\overline{\im \varphi_i}  = Y_i$ is irreducible, we can apply \cite[(6.4.1)]{Kem93} and find $U_i \subseteq \im \varphi_i = Y_i$ such that $U_i$ is open and dense in the closure $\overline{\im \varphi_i} = Y_i$, and every $g \in U_i$ satisfies 
\[\dim \varphi_i^{-1}(g) = \dim \varphi^{-1}_i(Y_i) - \dim \overline{\im \varphi_i} = \dim \varphi^{-1}_i(Y_i) - \dim Y_i.\] 
Let $g \in U_i$, then $\varphi^{-1}_i(g) = \{(g,\lambda a_{g,2},\lambda a_{g,1}) \mid \lambda \in \KK^\times\} \cong \KK^\times$ and we obtain
\[ (\ast) \quad 1 = \dim \varphi^{-1}_i(g) = \dim \varphi^{-1}_i(Y) - \dim Y_i.\]
Since $\im  \pi_{\psi} = \bigcup^n_{i=1} Y_i$ and $\Stab_{G}(\psi) = \bigcup^n_{i=1} \varphi_i^{-1}(Y_i)$ are finite unions of closed subsets, we conclude with ($\ast$)
\[ \dim \im  \pi_{\psi} = \max\{\dim Y_i\} = \max\{\dim \varphi^{-1}_i(Y_i) - 1\} = -1 + \max\{ \dim \varphi^{-1}_i(Y_i)\} =  \Stab_G(\psi)-1.\]
Hence, $\dim \Stab_G(\psi) \leq \dim \GL(A_r) + 1 = r^2+1$.

Since $\GL(A_r)$ is an irreducible variety (see \cite[\S 7.3]{Hum12}) and $\im  \pi_{\psi} \subseteq \GL(A_r)$ is closed, we also have (see \cite[(I.1.10(d)]{Har77}) $\dim \im  \pi_{\psi} = \dim \GL(A_r) = r^2$ if and only if $\im  \pi_{\psi} = \GL(A_r)$. In summary, 
we have
\[ \psi \ \text{is homogeneous} \Leftrightarrow \im  \pi_{\psi} = \GL(A_r) \Leftrightarrow \dim \im  \pi_{\psi} = r^2 \Leftrightarrow \dim \Stab_{G}(\psi) = r^2+1.\]
This proves the equivalence of (1),(2) and (3).
\end{proof}

\bigskip

\begin{corollary}\label{2.2.4}
Let $r \geq 3$. The following statements are equivalent.
\begin{enumerate}
    \item There exists a brick $\psi \in \cV(K_r;V_1,V_2)$ that is not homogeneous.
    \item The set $\cB(V_1,V_2) \subseteq  \cV(K_r;V_1,V_2)$ contains a non-empty open subset of $\cV(K_r;V_1,V_2)$ consisting of non-homogeneous representations.
\end{enumerate}
\end{corollary}
\begin{proof} 
Let $G := G(A_r;V_1,V_2)$. It follows from \cite[(6.4.5)]{Kem93} that the map
\[ \varphi \colon \cV(K_r;V_1,V_2) \lra \NN_0 \ ; \ \psi \mapsto \dim \Stab_G(\psi)\]
is upper semicontinuous, i.e., $\varphi^{-1}(\NN_{\geq n})$ is closed for all $n \in \NN_0$. This implies that for $s := \min \varphi(\cV(K_r;V_1,V_2))$ the set
\[ \cO_G(V_1,V_2)  :=  \{ \psi \in \cV(K_r;V_1,V_2) \mid \dim \Stab_G(\psi) = s\} = \varphi^{-1}(\{s\}) = \cV(K_r;V_1,V_2) \setminus \varphi^{-1}(\NN_{\geq s+1})\]
is open and non-empty.

(1) $\Rightarrow$ (2). Since $\psi$ is not homogeneous, we conclude with \cref{2.2.3} $\dim \Stab_G(\psi) \leq r^2$. Moreover,  \cref{2.1.1} implies that $O := \cO_G(V_1,V_2) \cap \cB(V_1,V_2)$ is open in $\cV(K_r;V_1,V_2)$ and non-empty. Let $\psi' \in O$, then $\dim \Stab_G(\psi') = s \leq \dim \Stab_G(\psi) \leq r^2$. We conclude with \cref{2.2.3} that $\psi'$ is not homogeneous.

(2) $\Rightarrow$ (1). This is clear.
\end{proof}

\bigskip

The application of \cref{2.2.4} requires the existence of a non-homogeneous brick for each regular dimension vector.
This will be established in the following via a theorem of Chen \cite{Che13}, which explicitly constructs such bricks.

\bigskip

\begin{Lemma}\label{2.2.5}
Let $m,n,i \in \NN$ such that $2 \leq i \leq n-m+1$\footnote{Note that this implies $n \geq m+1$.}.
Consider the representation $M \in \rep(K_2)$ with $M_1 = \KK^m$, $M_2 = \KK^{n}$, $M(\gamma_1) \colon \KK^m \lra \KK^{n} \ ; \ x \mapsto I(1)x$ and  $M(\gamma_2) \colon \KK^m \lra \KK^{n} \ ; \ x \mapsto I(i)x$,
where 
\[ I(1) := \begin{pmatrix}
   I_{m \times m} \\
   0_{n-m \times m}
\end{pmatrix} \in \Mat_{n \times m}(\KK) \ \text{and} \ I(i) := \begin{pmatrix}
   0_{i-1 \times m} \\
   I_{m \times m} \\
   0_{n-m-(i-1) \times m}
\end{pmatrix} \in \Mat_{n \times m}(\KK).\]
Then the radical $\Rad(M) \subseteq M$ satisfies \[ \dim_{\KK} \Rad(M)_2 =  m + \min\{i-1,m\}.\]
\end{Lemma}
\begin{proof}
We denote by $\{e'_1,\ldots,e'_m\}$ and \{$e_1,\ldots,e_{n}\}$ the standard bases of $\KK^m$ and $\KK^{n}$, respectively. We have
\[ M(\gamma_1)(e'_j) = e_j \ \text{and} \  M(\gamma_2)(e'_j) = e_{j+i-1} \]
for all $j \in \{1,\ldots,m\}$. Hence, \cite[(III.2.2)]{ASS06} implies 
\[\Rad(M)_2 = \bigoplus_{j=1}^m \KK e_j + \bigoplus_{j=1}^m \KK e_{j+i-1}.\]
Since $|\{i,\ldots,i+(m-1)\} \setminus \{1,\ldots,m\}| = \min \{m,i-1\}$, we conclude $\dim_{\KK} \Rad(M)_2 = m +  \min\{m,i-1\}$.
\end{proof}

\bigskip

We denote by $\sigma_{K_r},\sigma^{-1}_{K_r}: \rep(K_r) \lra \rep(K_r)$ the shift functors on $\rep(K_r)$. They correspond reflection functors but take into account that the opposite quiver of $K_r$ is isomorphic to $K_r$, i.e., $D_{K_r} \circ \sigma_{K_r} \cong \sigma_{K_r}^{-1} \circ D_{K_r}$. For a precise definition we refer to \cref{S:3.2}.

\bigskip

\begin{proposition}\label{2.2.6}
Let $r \geq 3$ and $q_r(\dimu(V_1,V_2)) \leq 0$. There exists $\psi \in \cV(K_r;V_1,V_2)$ such that $V_{\psi}$ is a non-homogeneous brick.
\end{proposition}
\begin{proof}
Recall that the shift functors $\sigma_{K_r},\sigma_{K_r}^{-1} \colon \rep(K_r) \lra \rep(K_r)$ and duality induce  auto-equivalences on the category of regular representations. On the level of the Grothendieck group $K_0(K_r) \cong \ZZ^2$ we obtain an action of the group $H \subseteq \Aut(\ZZ^2)$ 
generated by $\sigma_r \colon \ZZ^2 \lra \ZZ^2 ; (x,y) \mapsto (rx-y,x)$ and the twist $\delta \colon \ZZ^2 \lra \ZZ^2 ; (x,y) \mapsto (y,x)$ acting on the set 
\[\cR := \{ (x,y) \in \NN^2 \mid q_r(x,y) \leq 0 \}.\]
A fundamental action for this group is given by 
\[\cF_r := \{ (m,n) \in \NN^2 \mid \frac{2}{r}n \leq m \leq n\}\] (see \cite[\S 2.6]{Kac80}). Since $\sigma_{K_r},\sigma_{K_r}^{-1}$ and duality respect homogeneity (see \cite[(5.1.3)]{BF24}), we may therefore assume that $V_1 = \KK^m$ and $V_2 = \KK^n$ with $(m,n) \in \cF_r$. We write
\[ n = qm + s \ \text{with} \ q \in \NN_0,s \in \{0,\ldots,m-1\}.\]
Since $(m,n) \in \cF_r,$ we also have $0 \neq q$ and $qm \leq qm+s = n \leq \frac{r}{2}m$ implies $q \leq \frac{r}{2} < r-1$. Hence, we are left to consider the following three cases. We stick to the notation introduced in \cref{2.2.5}.
\begin{enumerate}
    \item[(i)] $q = 1$ and $s = 0$: We are in case (1) of the proof of \cite[(3.6)]{Che13}\footnote{Note that Chen multiplies matrices from the right.} and find a brick $M \in \rep(K_r)$ with dimension vector $(m,m)$ such that $M$ such that that $\rk(M(\gamma_1)) = m-1$ and $\rk(M(\gamma_2)) = m$. Hence, $M$ is not homogeneous by \cref{1.5.2}.
    \item[(ii)] $q = 1$ and $0 < s < m$: We are in case (2) of the proof of \cite[(3.6)]{Che13} and find a brick $M \in \rep(K_r)$ with  dimension vector $(m,n)$ such that $M(\gamma_1)(x) = I(1)x, M(\gamma_2)(x) = I(s+1)x$ and $M(\gamma_3)(x) = I(2)x$ for all $x \in \KK^m$. 
    If $s = 1$ we have $M(\gamma_2) = M(\gamma_3)$ and therefore $\rk(M(\gamma_1)) = m$ and $\rk(M(\gamma_2)- M(\gamma_3)) = 0 \neq m$. Hence, $M$ is not homogeneous by \cref{1.5.2}. 
    For $s > 1$ we consider the injective $\KK$-linear maps 
    \[ \alpha \colon A_2 \lra A_r ; \gamma_1 \mapsto \gamma_1, \gamma_2 \mapsto \gamma_2 \ \text{and} \ \beta \colon A_2 \lra A_r ; \gamma_1 \mapsto \gamma_1, \gamma_2 \mapsto \gamma_3.\]
    We conclude with \cref{2.2.5} that $\dim_{\KK} 
  \Rad(\alpha^\ast(M))_2 = m+1$ and $\dim_{\KK} 
  \Rad(\beta^\ast(M))_2 = m+s > m+1$. Hence, $\alpha^\ast(M) \not\cong \beta^\ast(M)$. Therefore, $M$ is not uniform and in particular not homogeneous.
    \item[(iii)] $2 \leq q \leq r-2$ and $0 \leq s < m$: We are in the case (3) or (4) of the proof of \cite[(3.6)]{Che13} and find a brick $M \in \rep(K_r)$ with dimension vector $(m,n)$ such that
    $M(\gamma_1)(x) = I(1)x, M(\gamma_2)(x) = I(m+1)x$ and $M(\gamma_3)(x) = I(2)x$ for all $x \in \KK^m$. For $m=1$, we have $M(\gamma_2) = M(\gamma_3)$, and argue as in (ii). Hence, we may assume that $m > 1$. We consider the injective $\KK$-linear maps 
    \[ \alpha \colon A_2 \lra A_r ; \gamma_1 \mapsto \gamma_1, \gamma_2 \mapsto \gamma_2 \ \text{and} \ \beta \colon A_2 \lra A_r ; \gamma_1 \mapsto \gamma_1, \gamma_2 \mapsto \gamma_3.\]
    We conclude with \cref{2.2.5} that $\dim_{\KK} 
  \Rad(\alpha^\ast(M))_2 = 2m$ and $\dim_{\KK} 
  \Rad(\beta^\ast(M))_2 = m+1 < 2m$. Hence, $\alpha^\ast(M) \not\cong \beta^\ast(M)$.  Therefore, $M$ is not uniform and in particular not homogeneous.
\end{enumerate}
\end{proof}

\bigskip 

With all the necessary ingredients at hand, we are now in the position to prove \cref{2.2.2}.

\bigskip

\begin{proof}[Proof of \cref{2.2.2}]
(1) $\Rightarrow$ (2) This follows from \cref{2.2.4} and \cref{2.2.6}.

(2) $\Rightarrow$ (3) This is clear.

(3) $\Rightarrow$ (1) We have a non-empty set $\cO \subseteq \cV(K_r;V_1,V_2)$ consisting of indecomposable representations that are not homogeneous. We conclude
with \cref{2.1.1} that $q_r(\dimu(V_1,V_2)) \leq 1$. Hence, it remains to rule out $q_r(\dimu(V_1,V_2)) = 1$. We consider the action of $G \coloneqq \GL(V_2) \times \GL(V_1)$ on $\cV(K_r;V_1,V_2)$ given by $(g_2,g_1).\psi := g_2 \circ \psi \circ (g_1^{-1} \otimes \id_{V_1})$. Note that the orbits correspond to the isomorphisms classes of representations in $\cV(K_r;V_1,V_2)$. 
Hence, according to \cref{2.1.1}(3) and \cref{2.2.1}(4), there exists $\psi \in \cV(K_r;V_1,V_2)$ homogeneous such that $\cO \subseteq G.\psi$. This is a contradiction since every representation in $G.\psi$ is homogeneous. 
\end{proof}

\bigskip

\subsection{Uniform representations} 
Our interest lies in the investigation of uniform Steiner bundles, and by \cref{1.5.3} we may instead focus on uniform representations in $\rep_{\proj}(K_r,1)$.  
We begin this section by recalling the following result, whose proof can be found in \cite[(2.3.2)]{BF24} and \cite[(2.1.4)]{Bis25}.

\bigskip

\begin{Theorem}\label{2.3.1}
Let $r \geq 2$ and $1 \leq d < r$. The following statements hold.
\begin{enumerate}
    \item $\rep_{\proj}(K_r,d) \cap \cV(K_r;V_1,V_2)$ is an open subset of $\cV(K_r;V_1,V_2)$.
    \item The following statements are equivalent.
    \begin{enumerate}
        \item[(i)] $\rep_{\proj}(K_r,d) \cap \cV(K_r;V_1,V_2) \neq \emptyset$.
        \item[(ii)] $\Delta_{(V_1,V_2)}(d) \geq (r-d)\min \{d,\dim_{\KK} M_1\}$.
    \end{enumerate}
\end{enumerate}
\end{Theorem}

\bigskip

Application of \cref{2.3.1} for $d = 2$ in conjunction with \cref{1.5.5} shows that, unlike the situation in \cref{2.2.2}, the variety of representations corresponding to a regular dimension vector may contain a non-empty open subset of uniform representations. However, as the following results demonstrate, the existence of a uniform representation depends critically on the choice of the dimension vector $\dimu(V_1,V_2)$.

\bigskip

\begin{proposition}\label{2.3.2}
The variety $\rep_{\proj}(K_r,1) \cap \cV(K_r;V_1,V_2) \neq \emptyset$ need not to contain any uniform representation. 
\end{proposition}
\begin{proof}
We consider $r = 3$ and $\Char(\KK) = 0$\footnote{The assumption $\Char(\KK) = 0$ is not necessary, cf. \cref{6.3.2} and \cite[(2.4)]{Lan79}.}. Let $(V_1,V_2)$ be a pair of vector spaces with dimension vector $\dimu(V_1,V_2) \neq (1,3)$ such that $\Delta_{(V_1,V_2)} = 2$. We apply \cref{2.3.1} and conclude $\rep_{\proj}(K_r,1) \cap \cV(K_r;V_1,V_2) \neq \emptyset$. Assume to the contrary that there exists a representation $M \in \rep_{\proj}(K_r,1) \cap \cV(K_3;V_1,V_2)$ that is uniform. Then \cref{1.3.3} implies that $\TilTheta(M) \in \StVect(\PP(A_3))$ is uniform and of rank $\Delta_{(V_1,V_2)} = 2$. According to a Theorem of Van de Van (see \cite[(2.2.2)]{OSS80}), we find $a \in \ZZ$ such that $\TilTheta(M) \cong \cT_{\PP(A_r)}(a)$ is homogeneous. One readily checks that this implies $a = -1$ since $\widetilde{\Theta}(M)$ is a Steiner bundle. \cref{1.3.3} and \cref{1.3.4} imply $M \cong P_1(3)$. This is impossible since $\dimu P_1(3) = (1,3) \neq \dimu(V_1,V_2)$.   
\end{proof}

\bigskip

\begin{proposition}\label{2.3.3}
Let $r \geq 2$, $1 \leq d < r$ and let $M \in \rep(K_r)$ be an indecomposable representation such that $q_r(\dimu M) + \Delta_{M}(d) \geq 1$. Then $M \in \rep_{\proj}(K_r,d)$.
\end{proposition}
\begin{proof}
The proof follows from a straightforward adaptation of the arguments in \cite[Theorem~A]{Wie08b} and \cite[(3.1)]{Bis23} (see also \cite[(4.2.2)]{Wie08a}). We only sketch the main ideas.

In view of \cref{1.2.1}, it suffices to show that $\psi_{M,\fv} \colon \fv \otimes_{\KK} M_1 \lra M_2$ is injective for every $\fv \in \Gr_d(A_r)$. Since $g.M$ is indecomposable for all $g \in \GL(A_r)$ and $\GL(A_r)$ acts with one orbit on $\Gr_d(A_r)$, we may assume $\fv = A_d$.  We denote by $\widehat{K_r}$ the quiver

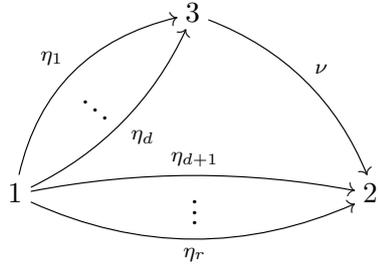
\[
\xymatrix{
& & 3 \ar^{\nu}@/^1pc/[rrdd] & & \\
& & & & \\
1 \ar^{\eta_{d+1}}_{\vdots}@/^0.6pc/[rrrr] \ar_{\eta_r}@/_1.4pc/[rrrr]  \ar^{\eta_1}_{\ddots}@/^1.4pc/[rruu] \ar_{\eta_d}@/_0.8pc/[rruu]& & & & 2.
}
\]
We define a representation $\widehat{M} \in \rep(\widehat{K_r})$ given by the following data:
$\widehat{M}_i = M_i$ for $i \in \{1,2\}$, 
$\widehat{M}_3 = \im \psi_{M,A_d}$, 
$\widehat{M}(\eta_i) = M(\gamma_i)$ for $i \in \{1,\ldots,r\}$, 
and $\widehat{M}(\nu)$ is the inclusion morphism.
This representation is indecomposable. Hence, \cite[(Thm B)]{Kac82} implies  
\begin{align*}
1 \geq q_{\widehat{K_r}}(\dimu \widehat{M}) &= q_r(\dimu M) + (d \dim_{\KK} M_1 - \rk(\psi_{M,A_d}))(\dim_{\KK} M_2 - \rk(\psi_{M,A_d}))\\
&\geq -\Delta_M(d)+1 + (d \dim_{\KK} M_1 - \rk(\psi_{M,A_d}))(\dim_{\KK} M_2 - \rk(\psi_{M,A_d})). 
\end{align*}
Therefore,
\[ \Delta_{M}(d) \geq (d \dim_{\KK} M_1 - \rk(\psi_{M,A_d}))(\dim_{\KK} M_2 - \rk(\psi_{M,A_d}))\]
and the assumption $d \dim_{\KK} M_1 > \rk(\psi_{M,A_d})$ implies
\[ \Delta_M(d) \geq 1 \cdot (\dim_{\KK} M_2 - \rk(\psi_{M,A_d})) > \dim_{\KK} M_2 - d \dim_{\KK} M_1 = \Delta_M(d),\]
in contradiction.
\end{proof}

\bigskip

\noindent The following examples show how to apply the preceding results.

\bigskip

\begin{example}\label{2.3.4}
Let $r = 3$. For $(m,n) \in \NN^2$ we set $V_1 := \KK^m$ and $V_2 := \KK^n$. Note that all of the following pairs $(m,n)$ satisfy $q_3(m,n) \leq 0$; hence, by \cref{2.1.1}, the variety $\cV(K_r;V_1,V_2)$ contains the non-empty open subset $\cB(V_1,V_2)$ consisting of regular bricks.

\begin{enumerate}
    \item For $(m,n) = (4,10)$ we have $\Delta_{(V_1,V_2)}(2) = 10 - 2 \cdot 4 = 2 \geq  (3-2)\min \{2,4\}$. Hence, \cref{2.3.1} and \cref{1.5.5} imply that $\rep_{\proj}(K_3,2) \cap \cV(K_3;V_1,V_2)$ is non-empty and consists of uniform representations. Moreover, \cref{2.2.2} implies that $\rep_{\proj}(K_3,2) \cap \cV(K_3;V_1,V_2)$ contains a non-empty open subset of non-homogeneous representations.
    \item Set $(m,n) := (2,4)$. According to \cref{2.3.2}, the space $\rep_{\proj}(K_3,1) \cap \cV(K_3;V_1,V_2)$ does not contain any uniform representation.
    \item Set $(m,n) := (5,13)$. Then $q_3(\dimu(V_1,V_2)) + \Delta_{(V_1,V_2)}(2) = 2 \geq 1$. We apply \cref{2.3.3} and conclude that every indecomposable representation in $\cV(K_3;V_1,V_2)$ is contained in $\rep_{\proj}(K_3,2)$. In particular, $\cB(V_1,V_2) \subseteq \rep_{\proj}(K_3,2) \cap \cV(K_3;V_1,V_2)$.
\end{enumerate}
\end{example}

\bigskip

We now turn our attention to the distribution of uniform representations within regular Auslander–Reiten components. To this end, we begin with the following definition and recall a main result of \cite{BF24} that generalizes \cite[(3.3)]{Wor13a}.

\bigskip

\begin{Definition}\label{2.3.5}
Let $r \geq 3$ and $\cC$ be a regular component of the Auslander-Reiten of $\rep(K_r)$. Given a quasi-simple representation $M \in \cC$, we denote by
\[ (M \to) \coloneqq \{ \tau_{K_r}^{-n}(M)_{[i]} \mid n \in \NN_0, i \in \NN \}\]
the \textit{cone of successors} of $M$ in $\cC$.
\end{Definition}

\bigskip

\begin{Theorem}\label{2.3.6}
Let $r \geq 2$ and $d \in \{1,\ldots,r-1\}$. 
\begin{enumerate}
    \item The category $\rep_{\proj}(K_r,d)$ is a torsion-free class closed under direct summands, $\sigma_{K_r}^{-1}$ and $\tau_{K_r}^{-1}$. 
    \item The category $\rep_{\proj}(K_r,d)$ contains all preprojective representations and no non-zero preinjective representation.
    \item Let $r \geq 3$ and $\cC$ be a regular component of the Auslander-Reiten component. There exists a quasi-simple representations $M_{\cC,d} \in \cC$ such that $\cC \cap \rep_{\proj}(K_r,d) = (M_{\cC,d} \to)$.
    \item Let $r \geq 3$ and $\cC$ be a regular component of the Auslander-Reiten component. Then either $M_{\cC,1} = M_{\cC,2}$ or $M_{\cC,1} = \tau_{K_r}(M_{\cC,2})$.
\end{enumerate}
\end{Theorem}

\bigskip

Figure~\ref{Fig:3} illustrates the preceding result. As a direct consequence of \cref{2.3.6} and \cref{1.5.5}, every regular component $\cC$ contains infinitely many uniform representations.

\bigskip

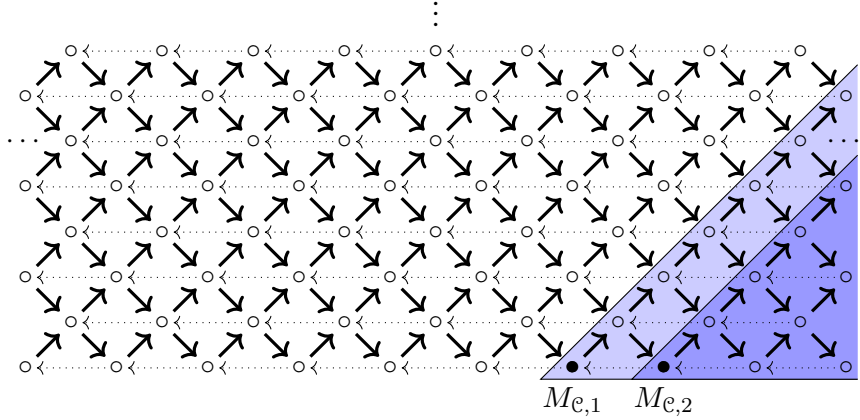
\begin{figure*}[!ht]
\centering 
\begin{tikzpicture}[very thick, scale=1]
                    [every node/.style={fill, circle, inner sep = 1pt}]

\def \n {9} % #Knoten Reihe  - 1
\def \m {3} % #Knoten Spalte - 1
\def \translation {1} % 1 Für Translation

\def \ab {0.15} % Abstand Pfeil und Knoten
\def \Pab {0.6} % Halber Abstand Horizontal

\def \rcone {1} % 1 für rechten Kegel
\def \rdist {2} % Anzahl der quasi-einfachen die eingeschlossen werden - 1
\def \rcolor {blue} %  white für keine Farbe

\def \rrcone {1} %1 für einen zweiten rechten Kegel links von rcone
\def \rrdist {3} % Anzahl der quasi-einfachen die eingeschlossen werden - 1

\node[color=black]  at (\n*\Pab*2-\Pab*2*3,-\ab-0.3) {$M_{\mathcal{C},1}$};
\node[color=black]  at (\n*\Pab*2-\Pab*2*2,-\ab-0.3) {$M_{\mathcal{C},2}$};
\node[color=black]  at (\n*\Pab*2-\Pab*2*4+1.2,-\ab+0.16) {$\bullet$};
\node[color=black]  at (\n*\Pab*2-\Pab*2*4+2.4,-\ab+0.16) {$\bullet$};
\foreach \a in {0,...,\n}{
\foreach \b in {0,...,\m}{
  
   \ifthenelse{\a = \n \and \b < \m}{
   \node[color=black] ({\a,\b,5})at ({\a*2*\Pab},{\b*2*\Pab}) {$\circ$};
     }
     {
      \ifthenelse{\b = \m \and \a < \n}{
      \node[color=black] ({\a,\b}) at ({\a*2*\Pab+\Pab},{\b*2*\Pab+\Pab}) {$\circ$};
      \node[color=black] ({\a,\b,5})at ({\a*2*\Pab},{\b*2*\Pab}) {$\circ$};
      }
      {
    
     \ifthenelse{\b = \m \and \a = \n}
     {\node[color=black] ({\a,\b,5})at ({\a*2*\Pab},{\b*2*\Pab}) {$\circ$};}
    {\node[color=black] ({\a,\b}) at ({\a*2*\Pab+\Pab},{\b*2*\Pab+\Pab}) {$\circ$};
    \node[color=black] ({\a,\b,5})at ({\a*2*\Pab},{\b*2*\Pab}) {$\circ$};

      }
      }
      }
    }
    }

\foreach \s in {0,...,\n}{
\foreach \t in {0,...,\m}
{  
 \ifthenelse{\t = \m \and \s < \n}{
    \draw[->] (\s*2*\Pab+\ab,\t*2*\Pab+\ab) to (\s*2*\Pab+\Pab-\ab,\t*2*\Pab+\Pab-\ab); 
    \draw[->] (\s*2*\Pab+\Pab+\ab,\t*2*\Pab+\Pab-\ab) to (\s*2*\Pab+2*\Pab-\ab,\t*2*\Pab+\ab); 

  }{
  
  \ifthenelse{\s = \n \and \t < \m}{
  
  }
  {
  \ifthenelse{\s = \n \and \t = \m}{
   
  }{
   \draw[->] (\s*2*\Pab+\ab,\t*2*\Pab+\ab) to (\s*2*\Pab+\Pab-\ab,\t*2*\Pab+\Pab-\ab); 
   \draw[->] (\s*2*\Pab+\Pab+\ab,\t*2*\Pab+\Pab+\ab) to (\s*2*\Pab+2*\Pab-\ab,\t*2*\Pab+2*\Pab-\ab);
   \draw[->] (\s*2*\Pab+\ab,\t*2*\Pab+2*\Pab-\ab) to (\s*2*\Pab+\Pab-\ab,\t*2*\Pab+\Pab+\ab); 
   \draw[->] (\s*2*\Pab+\Pab+\ab,\t*2*\Pab+\Pab-\ab) to (\s*2*\Pab+2*\Pab-\ab,\t*2*\Pab+\ab);    
   }
   }
  
    }
    }
    }

\ifthenelse{\isodd{\m}}
%% IF
 { 
  \node[color=black] (Dots1) at (0,\m*\Pab+2*\Pab) {$\cdots$};
  \node[color=black] (Dots2) at (\n*2*\Pab,\m*\Pab+2*\Pab) {$\cdots$};
   \ifthenelse{\isodd{\n}}{
  \node[color=black] (Dots3) at (0.5*\n*2*\Pab,2*\m*\Pab+2*\Pab) {$\vdots$};}
  {\node[color=black] (Dots3) at (0.5*\n*2*\Pab,2*\m*\Pab+\Pab) {$\vdots$};} 
  }
%% Else
  {
  \node[color=black] (Dots1) at (0,\m*\Pab+\Pab) {$\cdots$};
  \node[color=black] (Dots2) at (\n*2*\Pab,\m*\Pab+\Pab) {$\cdots$};
  \ifthenelse{\isodd{\n}}{
  \node[color=black] (Dots3) at (0.5*\n*2*\Pab,2*\m*\Pab+2*\Pab) {$\vdots$};}
  {\node[color=black] (Dots3) at (0.5*\n*2*\Pab,2*\m*\Pab+\Pab) {$\vdots$};}
  }
 
\ifthenelse{\translation = 1}{
   \foreach \s in {0,...,\n}{
   \foreach \t in {0,...,\m}{ 
   \ifthenelse{\s = 0}{}{
      \ifthenelse{\s = \n}{\draw[->,dotted,thin] (\s*2*\Pab-\ab,\t*2*\Pab) to (\s*2*\Pab-2*\Pab+\ab,\t*2*\Pab); }{
   \draw[->,dotted,thin] (\s*2*\Pab-\ab,\t*2*\Pab) to (\s*2*\Pab-2*\Pab+\ab,\t*2*\Pab); 
   \draw[->,dotted,thin] (\s*2*\Pab-\ab+\Pab,\t*2*\Pab+\Pab) to (\s*2*\Pab-2*\Pab+\Pab+\ab,\t*2*\Pab+\Pab); 
   }
   }}
}}
{}  %ELSE

\begin{scope}[on background layer]

\ifthenelse{\rrcone = 1}{
        \draw[fill = \rcolor!20] (\n*\Pab*2+\ab,-\ab) node[anchor=north]{}
  -- (\n*\Pab*2-\rrdist*\Pab*2-0.7*\Pab,-\ab) node[anchor=north]{}
  -- (\n*\Pab*2+\ab,\rrdist*\Pab*2+0.7*\Pab) node[anchor=south]{};
    }
  {}

\ifthenelse{\rcone = 1}{
        \draw[fill= \rcolor!40](\n*\Pab*2+\ab,-\ab) node[anchor=north]{}
  -- (\n*\Pab*2-\rdist*\Pab*2-0.7*\Pab,-\ab) node[anchor=north]{}
  -- (\n*\Pab*2+\ab,\rdist*\Pab*2+0.7*\Pab) node[anchor=south]{};
    }
  {}   
   
\end{scope}
\end{tikzpicture}
\caption{Cones of successors of $M_{\mathcal{C},1}$ and $M_{\mathcal{C},2}$ in a regular component $\mathcal{C} \subseteq \Gamma(K_r)$ with $M_{\mathcal{C},1} \neq M_{\mathcal{C},2}$. The Auslander-Reiten translation is indicated by dotted arrows.}
\label{Fig:3}
\end{figure*}

\bigskip

\begin{example}\label{2.3.7}
Let $r = 3$, $V_1 := \KK^{13}$ and $V_2 := \KK^{34}$. Then $q_3(\dimu(V_1,V_2)) = -1 \leq 0$ and \cref{2.2.2} gives us a non-empty open subset $O \subseteq \cV(K_3;V_1,V_2)$ such that for every $\psi \in O$ the representation $M := V_{\psi}$ is a quasi-simple brick (see \cite[(9.2),(9.4)]{Ker94}) in a regular component $\cC = \cC_{\psi}$ and not homogeneous. Due to \cref{2.2.1}, every representation in $\cC$ is non-homogeneous. 
We have $\dimu \tau_{K_3}(M) = (2,5)$. Application of \cref{2.3.3} for $d= 1$ yields $\tau_{K_3}(M) \in \rep_{\proj}(K_r,1)$ and \cref{2.3.1} shows that $\tau_{K_3}(M) \not\in  \rep_{\proj}(K_r,2)$. In summary, $\tau_{K_3}(M) = M_{\cC,1}$ and $M = M_{\cC,2}$.
\end{example}

\bigskip

\section{Restriction, inflation and shifts functors}\label{S:3}

Let $1 \leq d \leq r$. In this section, we recall the definitions of restriction, inflation, and shift functors, and prove that these functors give rise to adjoint pairs between $\rep(K_d)$ and $\rep(K_r)$. 

\subsection{Restriction and inflation}

We fix $1 \leq d \leq r$. Let $\iota \colon A_d \lra A_r$ be the canonical embedding sending $\gamma_i$ to $\gamma_i$ for all $1 \leq i \leq d$. We define functors that act as the identity on morphisms:
\[ \res \colon \rep(K_r) \lra \rep(K_d) \ ; \ M \mapsto \res(M) := \iota^\ast(M)\] with
\[ \psi_{\res(M)} = \psi_{M} \circ (\iota \otimes \id_{M_1}),\]
and 
\[  \inf \colon \rep(K_d) \lra \rep(K_r) \ ; \ X \mapsto \inf(X),\]
where
\[\psi_{\inf(X)}(\gamma_i \otimes x) = \begin{cases}
\psi_X(\gamma_i \otimes x), & 1 \leq i \leq d \\
0, & i > d.
\end{cases}\]

These two functors allow us to move between $\rep(K_d)$ and $\rep(K_r)$.
The following  example show that $(\inf,\res)$ and $(\res,\inf)$ are not adjoint pairs.

\bigskip

\begin{example}\label{3.1.1}
Let $1 \leq d < r$. We have $\res(P_1(r)) \cong P_1(d) \oplus (r-d)P_0(d)$ and 
\[ 1 = \dim_{\KK} \res(P_1(d))_1 = \dim_{\KK} \Hom_{K_d}(P_1(d),\res(P_1(r))).\]
The representation $\inf(P_1(d)) \in \rep(K_r)$ is indecomposable and regular (see \cref{2.1.1}). Since there are no non-zero morphisms from regular to projective representations (see \cite[(VIII.2.13)]{ASS06}), we conclude 
\[  \Hom_{K_r}(\inf(P_1(d)),P_1(r)) = 0 \neq \Hom_{K_d}(P_1(d),\res(P_1(r))).\]
By the same token, we have
\[  \Hom_{K_r}(I_1(r),\inf(I_1(d))) = 0 \neq \Hom_{K_d}(\res(I_1(r)),I_1(d)).\]
This shows that we can not hope for natural isomorphisms of the form
\[ \Hom_{K_r}(\inf(X),M) \cong \Hom_{K_d}(X,\res(M)) \ \text{and} \  \Hom_{K_r}(M,\inf(X)) \cong \Hom_{K_d}(\res(M),X).\]
\end{example}

\bigskip

\subsection{Shift functors}\label{S:3.2}

In the sequel the \textit{shift functors} $\sigma_{K_r},\sigma^{-1}_{K_r}: \rep(K_r) \lra \rep(K_r)$  will be of major importance. These functors correspond reflection functors but take into account that the opposite quiver of $K_r$ is isomorphic to $K_r$, i.e., $D_{K_r} \circ \sigma_{K_r} \cong \sigma_{K_r}^{-1} \circ D_{K_r}$.
We recall the definitions of $\sigma_{K_r}$ and $\sigma_{K_r}^{-1}$ in terms of the corresponding structure maps.  
Given a representation $M \in \rep(K_r)$, the structure map 
\[ \psi_{\sigma_{K_r}(M)} \colon A_r \otimes_{\KK} \ker \psi_M \lra M_1 \]
of $\sigma_{K_r}(M) \in \rep(K_r)$ is the restriction of the $\KK$-linear map 
\[A_r \otimes_{\KK} (A_r \otimes_{\KK} M_1) \lra M_1 \ ; \ \gamma_i \otimes (\gamma_j \otimes m) \mapsto \delta_{ij} m,\]
where $\delta_{ij}$ denotes the Kronecker delta.
If $f \in \Hom_{K_r}(M,N)$, then $\sigma_{K_r}(f)_1 : \sigma_{K_r}(M)_1 \lra \sigma_{K_r}(N)_1$ is the restriction of 
\[\id_{A_r} \otimes f_1 \colon A_r \otimes_{\KK}M_1 \lra A_r \otimes_{\KK}N_1\]
to $\ker \psi_{M}$ while $\sigma_{K_r}(f)_2\coloneqq f_1$. 

The representation $\sigma_{K_r}^{-1}(M)$ is given by $(M_2,\coker \eta_{M})$, where 
\[ \eta_{M} \colon M_1 \lra A_r \otimes_{\KK} M_2 \ ; \ m \mapsto \sum^r_{i=1} \gamma_i \otimes \psi_{M}(\gamma_i \otimes m)\]
with structure map
\[ \psi_{\sigma_{K_r}^{-1}(M)} \colon A_r \otimes_{\KK} M_2 \lra \coker \eta_{M} \ ; \  a \otimes m \mapsto a \otimes m + \im \eta_{M}.\]
If $f \in \Hom_{K_r}(M,N)$, then $\sigma^{-1}_{K_r}(f)_1 = f_2$ and $\sigma^{-1}_{K_r}(f)_2 : \sigma_{K_r}(M)_2 \lra \sigma_{K_r}(N)_1$ is the unique $\KK$-linear map making the diagram
\[
\xymatrix{
A_r \otimes_{\KK} M_2 \ar^{ \psi_{\sigma_{K_r}^{-1}(M)}}[r]  \ar_{\id_{A_r} \otimes f_2}[d] &\coker \eta_{M}  \ar^{\sigma^{-1}_{K_r}(f)_2}[d] \\
A_r \otimes_{\KK} N_2  \ar^{ \psi_{\sigma_{K_r}^{-1}(N)}}[r] & \coker \eta_{N} 
}
\]
commute.
As $(\sigma_{K_r}^{-1},\sigma_{K_r})$ is an adjoint pair \cite[(VII.5.7)]{ASS06}, $\sigma_{K_r}$ is left exact, while $\sigma_{K_r}^{-1}$ is right exact.

For $i \in \{1,2\}$ we denote by $\rep_i(K_r)$ the full subcategory of $\rep(K_r)$, whose objects do not have any direct summands isomorphic to the simple representation $S(i)$ and set $\rep_{1,2}(K_r) \coloneqq \rep_1(K_r) \cap \rep_2(K_r)$.
We also note that $D_{K_r}(\rep_i(K_r)) = \rep_{3-i}(K_r)$ for every $i \in \{1,2\}$. By \cite[(VII.5.3)]{ASS06}, the functor $\sigma_{K_r}$ induces an equivalence
\[ \sigma_{K_r} : \rep_2(K_r) \lra \rep_1(K_r).\]
By the same token, $\sigma_{K_r}^{-1} \colon \rep_1(K_r) \lra \rep_2(K_r)$ is a quasi-inverse of $\sigma_{K_r}$.
We 
also note that 
\[\dimu \sigma_{K_r}(M) =(r \dim_\KK M_1 - \dim_\KK M_2, \dim_\KK M_1)\] 
for $M$ in $\rep_{2}(K_r)$, while $\sigma_{K_r}(S(2))=0$.  In conjunction with the left exactness of $\sigma_{K_r}$ this implies that $\sigma_{K_r} \colon \rep_2(K_r) \lra \rep_1(K_r)$ is exact. By the same token, $\sigma_{K_r}^{-1} \colon \rep_1(K_r) \lra \rep_2(K_r)$ is exact.
The map
\[ \sigma_r \colon \ZZ^2 \lra \ZZ^2 \ ; \ (x,y) \mapsto (rx - y,x)\]
is invertible and satisfies
\[ \dimu \sigma_{K_r}(M) = \sigma_{r}(\dimu M) \ \text{and} \ \dimu \sigma^{-1}_{K_r}(N) = \sigma^{-1}_{r}(\dimu N)\]
for all $M \in \rep_2(K_r)$ and $N \in \rep_1(K_r)$. Finally, we recall (see \cite[(VII.5.8)]{ASS06}) that $\sigma_{K_r} \circ \sigma_{K_r}$ is just the Auslander-Reiten translation $\tau_{K_r}$. By the same token, we have $\sigma_{K_r}^{-1} \circ \sigma_{K_r}^{-1} \cong \tau_{K_r}^{-1}$.

\bigskip

\subsection{The adjoint pairs for Kronecker quivers}\label{S:3.3}
This section is devoted to the proof of the following result. We remark that this result holds in greater generality for any connected quiver $Q$ with a sink. The precise statement and the proof may be found in \cite[(2.2.2)]{Bis25b}.

\bigskip

\begin{Theorem}\label{3.3.1}
The functor
\[ \sigma_{K_r}^{-1} \circ \inf \colon \rep(K_d) \lra \rep(K_r) \ \text{is left adjoint to} \ \sigma_{K_d} \circ \res \colon \rep(K_r) \lra \rep(K_d).\]
\end{Theorem}

\begin{proof} Let $X \in \rep(K_d)$ and $M \in \rep(K_r)$. We fix $(f_1,f_2) \in \Hom_{K_r}((\sigma_{K_r}^{-1} \circ \inf)(X),M)$. 
Then $\psi_{(\sigma_{K_r}^{-1} \circ \inf)(X)} $ is given by
\[ \psi_{(\sigma_{K_r}^{-1} \circ \inf)(X)} \colon A_r \otimes_{\KK} X_2 \lra \coker \eta_{\inf(X)} \ ; \ a \otimes x \mapsto a \otimes x + \im \eta_{\inf(X)}.\]
By definition, we have $f_1 \in \Hom_{\KK}(X_2,M_1)$, $f_2 \in \Hom_{\KK}(\coker \eta_{\inf(X)},M_2)$ and
\[ (+) \quad \psi_M \circ (\id_{A_r} \otimes f_1) = f_2 \circ \psi_{(\sigma_{K_r}^{-1} \circ \inf)(X)}.\]
For $x \in X_1 = \inf(X)_1$ we have
\[(\ast) \quad \sum^d_{i=1} \gamma_i \otimes \psi_X(\gamma_i \otimes x) = \sum^r_{i=1} \gamma_i \otimes \psi_{\inf(X)} (\gamma_i \otimes x) = \eta_{\inf(X)}(x)\] and conclude
\begin{align*}
    (\psi_{\res(M)} \circ (\id_{A_d} \otimes f_1) \circ \eta_{X})(x) &=  \psi_{\res(M)}(\sum^d_{i=1} \gamma_i \otimes f_1(\psi_X(\gamma_i \otimes x))) \\
    &= \sum^d_{i=1} (\psi_M \circ (\id_{A_r} \otimes f_1))(\gamma_i \otimes \psi_X(\gamma_i \otimes x))\\ 
    &\stackrel{(+)}{=} \sum^d_{i=1} (f_2 \circ \psi_{(\sigma^{-1}_{K_r} \circ \inf)(X)})(\gamma_i \otimes \psi_X(\gamma_i \otimes x))\\
    &= f_2(\sum^d_{i=1} \gamma_i \otimes \psi_X(\gamma_i \otimes x) + \im \eta_{\inf(X)}) \\
    &\stackrel{(\ast)}{=} f_2(\eta_{\inf(X)}(x)+\im \eta_{\inf(X)}) = f_2(0) = 0.
\end{align*}
Hence, $\im((\id_{A_d} \otimes f_1) \circ \eta_{X}) \subseteq \ker \psi_{\res(M)}$ and we obtain a diagram
\[
\xymatrix{
A_d \otimes_{\KK} X_1 \ar^{\psi_X}[rrr] \ar^{\id_{A_d} \otimes (\id_{A_d} \otimes f_1) \circ \eta_{X}}[d]&&& X_2 \ar^{f_1}[d]\\
A_d \otimes_{\KK} \ker \psi_{\res(M)} \ar_{\psi_{(\sigma_{K_d} \circ \res)(M)}}[rrr] &&& M_1.
}
\]
Given $i \in \{1,\ldots,d\}$ and $x \in X_1$, we have
\begin{align*}
     (\psi_{(\sigma_{K_d} \circ \res)(M)} \circ (\id_{A_d} \otimes (\id_{A_d} \otimes f_1) \circ \eta_{X}))(\gamma_i \otimes x) &= \psi_{(\sigma_{K_d} \circ \res)(M)}(\gamma_i \otimes \sum^d_{j=1} \gamma_j \otimes (f_1 \circ \psi_X)(\gamma_j \otimes x)) \\
     &= (f_1 \circ \psi_X)(\gamma_i \otimes x).
\end{align*} 
This shows $((\id_{A_d} \otimes f_1) \circ \eta_{X},f_1) \in \Hom_{K_d}(X,(\sigma_{K_d} \circ \res)(M))$. We obtain a $\KK$-linear map
\[ \tau_{X,M} \colon \Hom_{K_r}((\sigma_{K_r}^{-1} \circ \inf)(X),M) \lra \Hom_{K_d}(X,(\sigma_{K_d} \circ \res)(M)) \ ; \ (f_1,f_2) \mapsto ((\id_{A_d} \otimes f_1) \circ \eta_{X},f_1).\]
Now we proceed in steps.
\begin{enumerate}
    \item[(i)] $\tau_{X,M}$ is injective: Let $(f_1,f_2) \in \ker \tau_{X,M}$, then $f_1 = 0$. Since $Y \in \rep(K_r) \in \rep_2(K_r)$ for every $Y \in \rep(K_r)$, $\psi_{(\sigma_{K_r}^{-1} \circ \inf)(X)}$ 
 is surjective and the equality $f_2 \circ \psi_{(\sigma_{K_r}^{-1} \circ \inf)(X)} = \psi_M \circ (\id_{A_r} \otimes f_1)$ implies $f_2 = 0$.
    \item[(ii)] $\tau_{X,M}$ is surjective: Let $(g_1,g_2) \in \Hom_{K_d}(X, (\sigma_{K_d} \circ \res)(M))$ and $x \in X_1 = \inf(X)_1$. We write $g_1(x) = \sum^d_{j=1} \gamma_j \otimes m_j \in (\sigma_{K_d} \circ \res(M))_1 = \ker \psi_{\res(M)} \subseteq A_d \otimes_{\KK} M_1$.  Recall that $\iota \colon A_d \lra A_r$ denotes the canonical embedding. We have 
\begin{align*}
    (\psi_M \circ (\id_{A_r} \otimes g_2))(\eta_{\inf(X)}(x))&=  (\psi_M \circ (\id_{A_r} \otimes g_2))(\sum^r_{i=1} \gamma_i \otimes \psi_{\inf(X)}(\gamma_i \otimes x)) \\
    &= (\psi_M \circ (\id_{A_r} \otimes g_2))(\sum^d_{i=1} \gamma_i \otimes \psi_{X}(\gamma_i \otimes x)) \\
    &= \psi_M(\sum^d_{i=1} \gamma_i \otimes g_2(\psi_X(\gamma_i \otimes x)))\\
    &=\psi_M(\sum^d_{i=1} \gamma_i \otimes \psi_{(\sigma_{K_d} \circ \res)(M)}(\gamma_i \otimes g_1(x))) \\
     &=\psi_M(\sum^d_{i=1} \gamma_i \otimes \psi_{(\sigma_{K_d} \circ \res)(M)}(\gamma_i \otimes \sum^d_{j=1} \gamma_j \otimes m_j)) \\
     &=\psi_M(\sum^d_{i=1} \gamma_i \otimes m_i) = (\psi_{M} \circ (\iota \otimes \id_{M_1}))(\sum^d_{i=1} \gamma_i \otimes m_i) \\
    &=\psi_{\res(M)}(g_1(x)) = 0,
\end{align*}
since $g_1(x) \in \ker \psi_{\res(M)}$. The universal property of $\coker \eta_{\inf(X)}$ gives us a unique $\KK$-linear map $h \colon \coker \eta_{\inf(X)} \lra M_2$, making the following diagram commute:
\[
\xymatrix{
X_1 \ar^{\eta_{\inf(X)}}[rr] && A_r \otimes_{\KK} X_2 \ar^{\psi_{(\sigma_{K_r}^{-1} \circ \inf)(X)}}[rr] \ar^{\id_{A_r} \otimes g_2}[d]&& \coker \eta_{\inf(X)} \ar^{h}@{..>}[d] \ar[rr] && 0\\
 && A_r \otimes_{\KK} M_1 \ar_{\psi_M}[rr] && M_2.
}
\]
Hence, $(g_2,h) \in \Hom_{K_r}((\sigma_{K_r}^{-1} \circ \inf)(X),M)$. Application of $\tau_{X,M}$ yields the morphism $\tau_{X,M}(g_2,h) = ((\id_{A_d} \otimes g_2) \circ \eta_{X},g_2) \in \Hom_{K_d}(X,(\sigma_{K_d} \circ \res)(M))$. In conclusion,
\[ (g_1 -(\id_{A_d} \otimes g_2) \circ \eta_{X},0)  = (g_1,g_2) - ((\id_{A_d} \otimes g_2) \circ \eta_{X},g_2)  \in \Hom_{K_d}(X,(\sigma_{K_d} \circ \res)(M))\]
and therefore 
\[0 = 0 \circ \psi_X = \psi_{(\sigma_{K_d} \circ \res)(M)} \circ (\id_{A_d} \otimes \underbrace{(g_1 -(\id_{A_d} \otimes g_2) \circ \eta_{X})}_{s\colon X_1 \lra \ker \psi_{\res(M)} \subseteq A_d \otimes M_1}).\]
Let $x \in X_1$ and write $s(x) = \sum^d_{i=1} \gamma_i \otimes m_i$ with $m_1,\ldots,m_d \in M_1$. We conclude for $j \in \{1,\ldots,d\}$
\[ 0 = \psi_{(\sigma_{K_d} \circ \res)(M)}(\gamma_j \otimes s(x)) =  \psi_{\sigma_{K_d}(\res(M))}(\gamma_j \otimes \sum^d_{i=1} \gamma_i \otimes m_i) = m_j.\]
In conclusion,  $s = 0$, $g_1 = (\id_{A_d} \otimes g_2) \circ \eta_{X}$ and $(g_1,g_2)  = \tau_{X,M}(g_2,h)$.
\item[(iii)] $\tau$ is natural in the first component: 
We let $X,Y \in \rep(K_d)$, $M \in \rep(K_r)$ be representations and $(g_1,g_2) \in  \Hom_{K_d}(X,Y)$, $(f_1,f_2) \in\Hom_{K_r}((\sigma_{K_r}^{-1} \circ \inf)(Y),M)$ be morphisms. We have $\eta_{Y} \circ g_1 = (\id_{A_d} \otimes g_2) \circ \eta_{X}$, $(\sigma_{K_r}^{-1} \circ \inf)(g)_1= g_2$ and conclude
\begin{align*}
[\Hom_{K_d}(g,(\sigma_{K_d} \circ \res)(M)) \circ \tau_{Y,M}](f_1,f_2) &= \Hom_{K_d}(g,(\sigma_{K_d} \circ \res)(M))((\id_{A_d} \otimes f_1) \circ \eta_{Y},f_1) \\
&= ((\id_{A_d} \otimes f_1) \circ \eta_{Y} \circ g_1,f_1 \circ g_2) \\
&= ((\id_{A_d} \otimes f_1) \circ (\id_{A_d} \otimes g_2) \circ \eta_{X},f_1 \circ g_2)  \\
&= ((\id_{A_d} \otimes f_1 \circ g_2) \circ \eta_{X},f_1 \circ g_2)  \\
&= \tau_{X,M}((f_1 \circ g_2,f_2 \circ (\sigma_{K_r}^{-1} \circ \inf)(g)_2)\\
&= \tau_{X,M}((f_1 \circ (\sigma_{K_r}^{-1} \circ \inf)(g)_1,f_2 \circ (\sigma_{K_r}^{-1} \circ \inf)(g)_2)\\
&= [(\tau_{X,M} \circ\Hom_{K_r}((\sigma_{K_r}^{-1} \circ \inf)(g),M)](f_1,f_2).
\end{align*}
Hence, we have a commutative diagram
\[
\xymatrix{
\Hom_{K_r}((\sigma_{K_r}^{-1} \circ \inf)(Y),M) \ar^{\tau_{Y,M}}[rrr] \ar^{\Hom_{K_r}((\sigma_{K_r}^{-1} \circ \inf)(g),M)}[d]&&& \Hom_{K_d}(Y,(\sigma_{K_d} \circ \res)(M)) \ar^{\Hom_{K_d}(g,(\sigma_{K_d} \circ \res)(M))}[d]\\
\Hom_{K_r}((\sigma_{K_r}^{-1} \circ \inf)(X),M) \ar^{\tau_{X,M}}[rrr] &&& \Hom_{K_d}(X,(\sigma_{K_d} \circ \res)(M)).
}
\]
\item[(iv)] $\tau$ is natural in the second component: Let $X \in \rep(K_d)$, $M,N \in \rep(K_r)$ be representations and $f \in \Hom_{K_r}(M,N)$,  $(g_1,g_2) \in \Hom_{K_r}((\sigma_{K_r}^{-1} \circ \inf)(X),M)$ be morphisms. 
We have 
\begin{align*}
    [\tau_{X,N} \circ \Hom_{K_r}((\sigma^{-1}_{K_r} \circ \inf)(X),f)](g_1,g_2) &= \tau_{X,N}(f_1 \circ g_1,f_2 \circ g_2) \\
    &= ((\id_{A_d} \otimes (f_1 \circ g_1)) \circ \eta_{X},f_1 \circ g_1) \\
    &=((\id_{A_d} \otimes f_1) \circ (\id_{A_d} \otimes g_1) \circ \eta_{X},f_1 \circ g_1)\\
    &= [\Hom_{K_d}(X,(\sigma_{K_d} \circ \res)(f)) \circ \tau_{X,M}](g_1,g_2).
\end{align*}
Hence, we have a commutative diagram
\[
\xymatrix{
\Hom_{K_r}((\sigma_{K_r}^{-1} \circ \inf)(X),M) \ar^{\tau_{X,M}}[rrr] \ar^{\Hom_{K_r}((\sigma^{-1}_{K_r} \circ \inf)(X),f)}[d]&&& \Hom_{K_d}(X,(\sigma_{K_d} \circ \res)(M)) \ar^{\Hom_{K_d}(X,(\sigma_{K_d} \circ \res)(f))}[d]\\
\Hom_{K_r}((\sigma_{K_r}^{-1} \circ \inf)(X),N) \ar^{\tau_{X,N}}[rrr] &&& \Hom_{K_d}(X,(\sigma_{K_d} \circ \res)(N)).
}
\]

\end{enumerate}
\end{proof}

\bigskip
\noindent By applying duality, we obtain the following result.
\bigskip

\begin{corollary}\label{3.3.2}
The functor
\[ \sigma_{K_d}^{-1} \circ \res \colon \rep(K_r) \lra \rep(K_d)\]
is left adjoint to
\[ \sigma_{K_r} \circ \inf \colon \rep(K_d) \lra \rep(K_r)\]
\end{corollary}
\begin{proof}
By \cref{3.3.1}, we have an adjoint pair $(\sigma_{K_r}^{-1} \circ \inf,\sigma_{K_d} \circ \res)$. Let $Y \in \rep(K_d)$ and $N \in \rep(K_r)$. We write $Y \cong D_{K_d}(X)$ and $N \cong D_{K_r}(M)$ with $X \in \rep(K_d)$ and $M \in \rep(K_r)$. Since we have natural equivalences $D_{K_r} \circ \inf \cong \inf \circ D_{K_d}$, $D_{K_d} \circ \res \cong \res \circ D_{K_r}$, $D_{K_r} \circ\sigma_{K_r} \cong \sigma_{K_r}^{-1} \circ D_{K_r}$ and $D_{K_d} \circ\sigma_{K_d} \cong \sigma_{K_d}^{-1} \circ D_{K_d}$, we obtain a sequence of natural isomorphisms
\begin{align*}
\Hom_{K_r}(N,(\sigma_{K_r} \circ \inf)(Y)) &\cong \Hom_{K_r}((D_{K_r} \circ \sigma_{K_r} \circ \inf)(Y),D_{K_r}(N))\\
&\cong \Hom_{K_r}((\sigma^{-1}_{K_r} \circ \inf)(D_{K_d}(Y)),D_{K_r}(N)) \\
&\cong \Hom_{K_d}(D_{K_d}(Y),(\sigma_{K_d} \circ \res)(D_{K_r}(N))) \\
&\cong \Hom_{K_d}((D_{K_d} \circ \sigma_{K_d} \circ \res)(D_{K_r}(M)),D_{K_d}(D_{K_d}(Y)))\\
&\cong \Hom_{K_d}((\sigma_{K_d}^{-1} \circ \res)(N),Y).
\end{align*}
\end{proof}

\begin{Remark}\label{3.3.3}
If we consider $d = r$, i.e., we do not remove any arrows, we obtain the classical result that $\sigma_{K_r}^{-1}$ is left adjoint to $\sigma_{K_r}$ (cf. \cite[(VII.5.7)]{ASS06}).
\end{Remark}

\bigskip

\section{Families of test representations}\label{S:4}

Characterizing a full subcategory of $\rep(K_r)$ as being right $\Hom$- or $\Ext^1$-orthogonal to certain algebraic families of \textit{test representations} is a well-established technique, employed repeatedly in the study of Kronecker representations (cf.\ \cite{HU91,Wor13a,Bis20}).
Recently, test representations have been applied to the study of relative projective representations, as will be outlined in what follows.

Let $1 \leq d < r$, $\fv \in \Gr_d(A_r)$ and $\alpha \in \Inj_{\KK}(A_d,A_r)$ with $\im \alpha = \fv$. We consider the projective Kronecker representations $(0,A_d) \cong d P_0(d)$ and $(\KK,A_d) \cong P_1(d)$, with structure map $\psi_{P_1(d)}(a \otimes t)  = ta$. The morphism $\alpha$ induces a morphism of representations 
\[ \overline{\alpha} \colon (0,A_d) \lra (\KK,A_r)\]
by setting $\overline{\alpha}_1 = 0$ and $\overline{\alpha}_2 \coloneqq \alpha$. We define 
\[ E(\fv) := D_{K_r}(\tau_{K_r}(\coker \overline{\alpha})).\]
In fact, up to isomorphism, this definition does not depend on the choice of $\alpha \in \Inj_{\KK}(A_d,A_r)$, cf. \cite[(2.1.3)]{BF24} and the following statements hold (see \cite[(2.1.3),(2.1.5)]{BF24}):

\medskip

\begin{enumerate}
   % \item $E(\fv)$ does not depend $\alpha \in \Inj_{\KK}(A_d,A_r)$ with $\im \alpha =\fv$.
    \item For $\fv,\fw \in \Gr_d(A_r)$ we have $E(\fv) \cong E(\fw)$ if and only if $\fv = \fw$. 
    \item $\rep_{\proj}(K_r,d) = \{ M \in \rep(K_r) \mid \forall \fv \in \Gr_d(A_r) \colon \Hom_{K_r}(E(\fv),M) = 0 \}$.
\end{enumerate}

\medskip

In this section, we show that the adjoint pairs constructed in \cref{S:3} give rise to a family of test representations in $\rep(K_r)$ for every non-semisimple homogeneous representation $X \in \rep(K_d)$. We then apply this construction to the preprojective indecomposables $P_n(d) \in \rep(K_d)$ for all $n \in \NN_0$, focusing on the case $d = 2$ to derive consequences for uniform Kronecker representations. Moreover, we show that the family obtained from $P_1(d)$ coincides with $(E(\fv))_{\fv \in \Gr_d(A_r)}$.

\subsection{Constructing test representations}

Let $1 \leq d < r$. In the following we construct two families of test representations, $(X^+(\alpha))_{\alpha \in \Inj_{\KK}(A_d, A_r)}$ and $(X^-(\alpha))_{\alpha \in \Inj_{\KK}(A_d, A_r)}$, for each $X \in \rep(K_d)$, and show in the next section that this construction can be lifted to $\Gr_d(A_r)$ whenever $X$ is homogeneous. As a first step, we introduce for $X \in \rep(K_r)$ the representations
\[
X_{d,r}^{-} \coloneqq (\sigma_{K_r}^{-1} \circ \inf)(X) 
\quad \text{and} \quad 
X_{d,r}^+ \coloneqq (\sigma_{K_r} \circ \inf)(X),
\]
whose definition will be generalized after the necessary prerequisites have been established in the subsequent result. Recall that $\iota \colon A_d \lra A_r$ denotes the canonical embedding.

\bigskip

\begin{Lemma}\label{4.1.1} Let $X \in \rep(K_d)$.
\begin{enumerate}
\item For $X\in\rep_{1,2}(K_d)$, we have an isomorphism $X^+_{d,r} \cong \tau_{K_r}(X^{-}_{d,r})$.
\item Let $g \in \GL(A_r)$ be such that $g \circ \iota = \iota$ holds, then $g.X_{d,r}^- \cong X_{d,r}^-$ and $g.X_{d,r}^+ \cong X_{d,r}^+$.
\end{enumerate}
\end{Lemma}
\begin{proof}
\begin{enumerate}
\item Since $\sigma_{K_r}^{-1},\sigma_{K_r}$ and $\inf$ commute with finite direct sums, we may assume that $X$ is indecomposable and not simple. By \cite[(3.2.1),(3.2.2)]{Bis20}, the representation $\inf(X)$ is therefore regular. Hence, we have $\sigma_{K_r}(\sigma_{K_r}^{-1}(\inf(X))) \cong \inf(X)$ and 
\[ \tau_{K_r}(X_{d,r}^-) \cong \sigma_{K_r}^2(\sigma_{K_r}^{-1}(\inf(X))) \cong (\sigma_{K_r} \circ \inf)(X) = X_{d,r}^+.\]
\item We have
\[ \psi_{X_{d,r}^-} = \psi_{\sigma_{K_r}^{-1}(\inf(X))} \colon A_r \otimes_{\KK} X_2 \lra \coker \eta_{\inf(X)} \ ; \  a \otimes x \mapsto a \otimes x + \im \eta_{\inf(X)}\]
with 
\[ \eta_{\inf(X)} \colon X_1 \lra A_r \otimes_{\KK} X_2 \ ; \ x \mapsto \sum^r_{i=1} \gamma_i \otimes \psi_{\inf(X)}(\gamma_i \otimes x) = \sum^d_{i=1} \gamma_i \otimes \psi_X(\gamma_i \otimes x)\]
and
\[ \psi_{g.X_{d,r}^-} = \psi_{X_{d,r}^-} \circ (g^{-1} \otimes \id_{X_2}).\]
Since $g \circ \iota = \iota$, we have $g^{-1}(\gamma_i) = \gamma_i$ for all $i \in \{1,\ldots,d\}$ and conclude for $x \in X_1$:
\begin{align*}
    \psi_{g.X_{d,r}^-}(\eta_{\inf(X)}(x)) &=  [\psi_{X_{d,r}^-} \circ (g^{-1} \otimes \id_{X_2})](\sum^d_{i=1} \gamma_i \otimes \psi_X(\gamma_i \otimes x))\\
    &= \psi_{X_{d,r}^-}(\sum^d_{i=1} \gamma_i \otimes \psi_X(\gamma_i \otimes x)) = \psi_{X_{d,r}^-}(\eta_{\inf(X)}(x)) = 0.
\end{align*}
This shows that there exists a unique $\KK$-linear map $h_2$ making the diagram
    \[
    \xymatrix{
    X_1 \ar^{\eta_{\inf(X)}}[rr] && A_r \otimes_{\KK} X_2 \ar^{\psi_{X_{d,r}^-}}[rr] \ar^{\id_{A_r} \otimes \id_{X_2}}[d]&& \coker \eta_{\inf(X)} \ar@{..>}^{h_2}[d]  \ar[rr] && 0 \\
        &&  A_r \otimes_{\KK} X_2\ar^{\psi_{g.X_{d,r}^-}}[rr]&& \coker \eta_{\inf(X)}  \ar[rr] && 0  
    }
    \]
    commute. Note that $h_2$ is an isomorphism. Hence, $(\id_{X_2},h_2)$ is the desired isomorphism and $g.X_{d,r}^- \cong X_{d,r}^-$.
    
    Now we prove the second isomorphism.  
    We write $X = Y \oplus P \oplus I$ with $Y \in \rep_{1,2}(K_d), P \in \add(P_0(d))$, $I \in \add(I_0(d))$ and prove the isomorphism for each direct summand. We have $\inf(P_0(d)) = P_0(r)$, $\inf(I_0(d)) = I_0(r)$ and conclude $P_0(d)^+_{d,r} = 0$ and $I_0(d)^+_{d,r} \cong I_1(r)$. Since $g.I_1(r) = I_1(r)$ (see \cref{2.2.1}), the statement follows for $P$ and $I$. Twofold application of \cref{4.1.1}(1) in conjunction with $g.Y^-_{d,r} \cong Y^{-}_{d,r}$ yields 
    \[g.Y_{d,r}^+ \cong g.(\tau_{K_r}(Y_{d,r}^-)) \cong \tau_{K_r}(g.Y_{d,r}^-) \cong \tau_{K_r}(Y_{d,r}^-) \cong Y_{d,r}^+.\]
\end{enumerate}
\end{proof}

\bigskip

\noindent For $g \in \GL(A_r)$, we define
\[ g|_{A_d} \coloneqq g \circ \iota \in \Inj_{\KK}(A_d,A_r).\]
Let $\alpha \in \Inj_{\KK}(A_d,A_r)$ and $g,h \in \GL(A_r)$ be such that $g|_{A_d} = \alpha = h|_{A_d}$. Then $(h^{-1} \circ g)|_{A_d} = \iota$ and \cref{4.1.1} implies
\[ g.X^{+}_{d,r} \cong h.X^{+}_{d,r} \ \text{and} \ g.X^{-}_{d,r} \cong h.X^{-}_{d,r}.\]
%\[ (h^{-1} \circ g).X^{\pm}_{d,r} \cong X^{\pm}_{d,r}\]
%and $g.X^{\pm}_{d,r} \cong h.X^{\pm}_{d,r}$.
Therefore, the following definition is well-defined up to isomorphism of representations. 

\bigskip

\begin{Definition}\label{4.1.2}
Let $\alpha \in \Inj_{\KK}(A_d,A_r)$ and $X \in \rep(K_d)$. Given $g \in \GL(A_r)$ such that $g|_{A_d} = \alpha$, we define 
\[ X^-(\alpha) \coloneqq  g.X_{d,r}^- \ \text{and}  \ X^+(\alpha) \coloneqq  g.X_{d,r}^+.\]
\end{Definition}

\bigskip

\begin{proposition}\label{4.1.3}
Let $M \in \rep(K_r)$, $X \in \rep(K_d)$ and $\alpha \in \Inj_{\KK}(A_d,A_r)$.    
\begin{enumerate}
\item We have isomorphisms of $\KK$-vector spaces
\[ \Hom_{K_r}(X^-(\alpha),M) \cong \Hom_{K_d}(X,(\sigma_{K_d} \circ \alpha^\ast)(M)) \cong \Hom_{K_d}(\sigma^{-1}_{K_d}(X),\alpha^\ast(M)).\]
\item We have isomorphisms of $\KK$-vector spaces
\[ \Hom_{K_r}(M,X^+(\alpha)) \cong \Hom_{K_d}((\sigma_{K_d}^{-1} \circ \alpha^\ast)(M),X) \cong \Hom_{K_d}(\alpha^\ast(M),\sigma_{K_d}(X)).\]
\end{enumerate}
\end{proposition}
\begin{proof}
We fix $g \in \GL(A_r)$ such that $g|_{A_d} = \alpha$ and note that
\begin{align*}
    \psi_{\res(g^{-1}.M)} &= \psi_{g^{-1}.M} \circ (\iota \otimes \id_{M_1}) = \psi_M \circ (g \otimes \id_{M_1}) \circ (\iota \otimes \id_{M_1})\\
    &= \psi_M \circ (g|_{A_d}\otimes \id_{M_1}) = \psi_{\alpha^\ast(M)}.
\end{align*}
\begin{enumerate}
\item We have
\begin{align*}
   \Hom_{K_r}(X^-(\alpha),M) & \ \cong \Hom_{K_r}(g.X_{d,r}^-,M) \cong \Hom_{K_r}(X_{d,r}^-,g^{-1}.M)\\
   &\stackrel{\ref{3.3.1}}{\cong} \Hom_{K_d}(X,(\sigma_{K_d} \circ \res)(g^{-1}.M)) \\
   & \ \cong  \Hom_{K_d}(X,(\sigma_{K_d} \circ \alpha^\ast)(M)).
\end{align*}
The second isomorphism holds since $(\sigma_{K_d}^{-1},\sigma_{K_d})$ is an adjoint pair.
\item We have 
\begin{align*}
    \Hom_{K_r}(M,X^+(\alpha)) &\cong \Hom_{K_r}(g^{-1}.M,X_{d,r}^+) \stackrel{\ref{3.3.2}}{\cong} \Hom_{K_d}((\sigma_{K_d}^{-1} \circ \res)(g^{-1}.M),X) \\
    &\cong \Hom_{K_d}((\sigma_{K_d}^{-1} \circ \alpha^{\ast})(M),X).
\end{align*}
The second isomorphism holds since $(\sigma_{K_d}^{-1},\sigma_{K_d})$ is an adjoint pair.
\end{enumerate}
\end{proof}

\bigskip

\bigskip
\noindent In the next result, we collect important properties of $X^-$ and $X^+$ needed in the subsequent sections. Whenever a statement applies to both $X^-$ and $X^+$, we abbreviate it by writing $X^{\pm}$.

\bigskip

\begin{Lemma}\label{4.1.4} Let $\alpha \in \Inj_{\KK}(A_d,A_r)$ and $X,Y \in \rep(K_d)$.
\begin{enumerate}
    \item For $h \in \GL(A_r)$ we have $h.X^{\pm}(\alpha) \cong X^{\pm}(h \circ \alpha)$.
    \item We have $(X \oplus Y)^{\pm}(\alpha) \cong X^{\pm}(\alpha) \oplus Y^{\pm}(\alpha)$.
    \item If $X$ is indecomposable and $X \not\cong I_0(d), P_0(d)$, then $X^{\pm}(\alpha)$ is regular indecomposable.
    \item If $\inf(X)$ and $\inf(Y)$ are regular, then $\Hom_{K_r}(X^{\pm}(\alpha),Y^{\pm}(\alpha)) \cong \Hom_{K_d}(X,Y)$.
    \item For $X \in \rep_{1,2}(K_d)$ we have $\tau_{K_r}(X^-(\alpha)) \cong X^+(\alpha)$.
\end{enumerate}
\end{Lemma}
\begin{proof}
We fix $g \in \GL(A_r)$ such that $g|_{A_d} = \alpha$.
\begin{enumerate}
\item Let $h \in \GL(A_r)$. We have $(h \circ g)|_{A_d} = h \circ \alpha$ and conclude \[h.X^{\pm}(\alpha) \cong h.(g.X^{\pm}_{d,r}) = (h \circ g).X^{\pm}_{d,r} \cong X^{\pm}(h \circ \alpha).\]
\item Clear.
\item By \cite[(3.2.1),(3.2.2)]{Bis20}, the representation $X_{d,r}^{\pm}$ is regular and indecomposable. Hence, the same is true for $g.X_{d,r}^{\pm} \cong X^{\pm}(\alpha)$.
\item For $Z \in \{X,Y\}$ we have $Z^{\pm}(\alpha) \cong g.Z^{\pm}_{d,r}$ and $\inf(Z)$ being regular gives us $\sigma_{K_r}^{-1}(Z^+_{d,r}) \cong \inf(Z)$. Since $\sigma_{K_r}^{-1}$ induces an equivalence on the category of regular representations, we conclude with $\inf \colon \rep(K_d) \lra \rep(K_r)$ being full and faithful
\begin{align*}
     \Hom_{K_r}(X^+(\alpha),Y^+(\alpha)) &\cong \Hom_{K_r}(g.X_{d,r}^+,g.Y_{d,r}^+) = \Hom_{K_r}(X_{d,r}^+,Y_{d,r}^+) \\
     &\cong \Hom_{K_r}(\sigma_{K_r}^{-1}(X_{d,r}^+),\sigma_{K_r}^{-1}(Y_{d,r}^+)) \cong \Hom_{K_r}(\inf(X),\inf(Y)) \\
     &\cong \Hom_{K_d}(X,Y).
\end{align*}
The isomorphism $\Hom_{K_r}(X^-(\alpha),Y^-(\alpha)) \cong \Hom_{K_d}(X,Y)$ follows in the same fashion.
\item \cref{4.1.1} implies $\tau_{K_r}(X^-(\alpha)) = \tau_{K_r}(g.X^-_{d,r}) \cong g.\tau_{K_r}(X^-_{d,r}) \cong g.X^{+}_{d,r} \cong X^+(\alpha)$.
\end{enumerate}
\end{proof}

\bigskip

\subsection{Constructing test representations}
The aim of this section is to show that, for a homogeneous and non-semisimple representation $X \in \rep(K_d)$ and $\alpha,\beta \in \Inj_{\KK}(A_d,A_r)$, we have
\[
X^{\pm}(\alpha) \cong X^{\pm}(\beta) \quad \text{if and only if} \quad \im \alpha = \im \beta.
\]

\bigskip

\begin{Definition}\label{4.2.1}
Let $s \in \NN$,  $\cB = (\gamma_1,\ldots,\gamma_s)$ be the standard basis of $A_s$ and $f \in \End_{\KK}(A_s)$. We denote by $f^{\tr} \in  \End_{\KK}(A_s)$ the unique endomorphism that satisfies the equation $\Mat_{\cB}(f^{\tr}) = \Mat_{\cB}(f)^{\tr}$.
\end{Definition}

\bigskip
\noindent 
The proof of the following result may be found in \cite[(5.1.3)]{BF24}.

\bigskip

\begin{Lemma}\label{4.2.2}
Let $g \in \GL(A_r)$ and $M \in \rep(K_r)$. We have
\[ \sigma^{-1}_{K_r}(g.M) \cong (g^{-1})^{\tr}.\sigma^{-1}_{K_r}(M) \ \text{and} \ \sigma_{K_r}(g.M) \cong (g^{-1})^{\tr}.\sigma_{K_r}(M).\]
\end{Lemma}

\bigskip

\begin{proposition}\label{4.2.3}
Let $X \in \rep(K_d)$ be homogeneous. Then $X^{\pm}(\alpha)$ only depends on $\im \alpha$. 
\end{proposition}
\begin{proof}
Let $\alpha,\alpha' \in \Inj_{\KK}(A_d,A_r)$ such that $\fv \coloneqq  \im \alpha = \im \alpha'$. We define $A_{d}^\perp \coloneqq  \bigoplus^r_{i > d} \KK \gamma_i$ and fix a $\KK$-complement $\fu \subseteq A_r$ of $\fv$ and an isomorphism $k \colon A_d^\perp \lra \fu$. Let $\pi \colon A_r = \fv \oplus \fu \lra \fv$ be the canonical projection and $\beta \in \{\alpha,\alpha'\}$. Then $\pi \circ \beta \in A_d \lra \fv$ is an isomorphism and we define
\[g_{\beta} \coloneqq  \begin{pmatrix}
\pi \circ \beta & 0 \\
0 & k
\end{pmatrix} \colon A_d \oplus A_{d}^\perp \lra \fv \oplus \fu \in \GL(A_r).\] 
By definition, we have $g_{\beta}|_{A_d} = \beta$ and therefore $X^{\pm}(\beta) \cong g_{\beta}.X_{d,r}^{\pm}$. We have to show that $X^{\pm}(\alpha) \cong X^{\pm}(\alpha')$ which is equivalent to $[g_{\alpha'}^{-1} \circ g_{\alpha}].X_{d,r}^{\pm} = X_{d,r}^{\pm}$. Let $h \coloneqq  (\pi \circ \alpha')^{-1} \circ (\pi \circ \alpha) \in \GL(A_d)$, then
\[ g \coloneqq  g_{\alpha'}^{-1} \circ g_{\alpha} = \begin{pmatrix}
(\pi \circ \alpha')^{-1} \circ (\pi \circ \alpha) & 0 \\
0 & \id_{A_d^\perp}
\end{pmatrix} =\begin{pmatrix}
h & 0\\
0 & \id_{A_d^\perp}
\end{pmatrix}.\]
Hence, 
\[ g^{\tr} = \begin{pmatrix}
h^{\tr} & 0\\
0 & \id_{A_d^\perp}
\end{pmatrix}.\]
Since $X \in \rep(K_d)$ is homogeneous, there is an isomorphism $f \in \Hom_{K_d}(h^{-\tr}.X,X)$. We claim that the diagram
     \[ 
     \xymatrix{
     A_r \otimes_{\KK} X_1 \ar^{\psi_{\inf(X)} \circ (g^{\tr} \otimes \id_{X_1})}[rrr] \ar^{\id_{A_r} \otimes f_1}[d] &&& X_2\ar^{f_2}[d] \\
      A_r \otimes_{\KK} X_1 \ar_{\psi_{\inf(X)}}[rrr] &&& X_2 
     }
     \]
     commutes.
Let $x \in X_1$ and $i \in \{1,\ldots,d\}$, then $h^{\tr}(\gamma_i) \in A_d$ and
\begin{align*}
[f_2 \circ \psi_{\inf(X)} \circ (g^{\tr} \otimes \id_{X_1})](\gamma_i \otimes x) &= [f_2 \circ \psi_{\inf(X)}](h^{\tr}(\gamma_i) \otimes x)  = (f_2 \circ \psi_X)(h^{\tr}(\gamma_i) \otimes x) \\
&= [f_2 \circ \psi_X \circ (h^{\tr} \otimes \id_{X_1})](\gamma_i \otimes x)\\
&= [f_2 \circ \psi_{h^{-\tr}.X}](\gamma_i \otimes x)\\
&= [\psi_{X}\circ (\id_{A_d} \otimes f_1)](\gamma_i \otimes x) \\
&= [\psi_{\inf(X)}\circ (\id_{A_r} \otimes f_1)](\gamma_i \otimes x).
\end{align*}
For $i > d$, we have
\begin{align*}
[f_2 \circ \psi_{\inf(X)} \circ (g^{\tr} \otimes \id_{X_1})](\gamma_i \otimes x) &= [f_2 \circ \psi_{\inf(X)}](\gamma_i \otimes x) = f_2(0) = 0 = \psi_{\inf(X)}(\gamma_i \otimes f_1(x)) \\
&=[\psi_{\inf(X)}\circ (\id_{A_r} \otimes f_1)](\gamma_i \otimes x).
\end{align*}
Hence, the diagram commutes and $g^{-\tr}.\inf(X) \cong \inf(X)$. Finally, we conclude with \cref{4.2.2}
\begin{align*}
g.X_{d,r}^- &= g.((\sigma_{K_r}^{-1} \circ \inf)(X)) \cong \sigma_{K_r}^{-1}(g^{-\tr}.\inf(X)) \cong \sigma_{K_r}^{-1}(\inf(X)) = X_{d,r}^-,
\end{align*}
and
\begin{align*}
g.X_{d,r}^+ &= g.((\sigma_{K_r} \circ \inf)(X)) \cong  \sigma_{K_r}(g^{-\tr}.\inf(X)) \cong \sigma_{K_r}(\inf(X)) = X_{d,r}^+.
\end{align*}
\end{proof}

\bigskip
\noindent 
By \cref{4.2.3}, the following definition is well-defined up to isomorphism of representations. 

\bigskip

\begin{Definition}\label{4.2.4}
    Let $X\in\rep(K_r)$ be homogeneous and $\fv \in \Gr_d(A_r)$. We define
    \[ X^{+}(\fv) \coloneqq  X^{+}(\alpha) \ \text{and} \ X^{-}(\fv) \coloneqq  X^{-}(\alpha)\]
    for any $\alpha \in \Inj_{\KK}(A_d,A_r)$ such that $\im \alpha = \fv$.
\end{Definition}

\bigskip

It follows directly from the definition that $P_0(d)^-(\fv) \cong P_1(r)$, $P_0(d)^+(\fv) = 0 = I_0(d)^-(\fv)$, and $I_0(d)^+(\fv) \cong I_1(r)$ for all $\fv \in \Gr_d(A_r)$. As we now show, apart from these cases, a homogeneous $X \in \rep(K_d)$ satisfies $X^{\pm}(\fv) \cong X^{\pm}(\fw)$ only if $\fv = \fw$.

\bigskip

\begin{proposition}\label{4.2.5}
Let $X \in \rep(K_d)$ be homogeneous, not semisimple, and $\alpha,\beta \in \Gr_d(A_r)$ be such that $X^{\pm}(\alpha) \cong X^{\pm}(\beta)$. Then $\im \alpha = \im \beta$.
\end{proposition}
\begin{proof} Let $\alpha,\beta \in \Inj_{\KK}(A_d,A_r)$ such that $X^-(\alpha) \cong X^{-}(\beta)$. In the following we suppress the exponent "-", i.e., we have $X(\alpha) \cong X(\beta)$. 
We write $X = a P_0(d) \oplus b I_0(d) \oplus Y$ with $0 \neq Y \in \rep_{1,2}(K_d)$. We have 
\[ a P_1(r) \oplus Y(\alpha) \cong X(\alpha) \cong X(\beta) = a P_1(r) \oplus Y(\beta)\]
and conclude with Krull-Remark-Schmidt that $Y(\alpha) \cong Y(\beta)$. Hence, we may assume that $X \in \rep_{1,2}(K_d)$. First we consider the special case $X(\alpha) \cong X(\iota)$. We need to show that $\im \alpha = \im \iota = A_d$ and proceed in steps. Let $A_d^\perp \coloneqq \bigoplus_{i > d}^r \KK \gamma_i$.
\begin{enumerate}
  \item[(i)] Let $g \in \GL(A_r)$ such that $g \circ \iota = \alpha$, then $A_d \cap g^{\tr}(A^\perp_d) = \{0\}$: We have
\[ X(\iota) \cong X(\alpha) \cong X(g \circ \iota) \stackrel{\ref{4.1.4}(1)}{\cong} g.X(\iota).\]
Since $X\in \rep_{1,2}(K_r)$, we conclude with \cite[(3.2.2)]{Bis20} that $\inf(X) \in \rep(K_r)$ is regular. Hence, $\inf(X) \cong (\sigma_{K_r} \circ \sigma_{K_r}^{-1})(\inf(X)) \cong \sigma_{K_r}(X(\iota))$ and
\[ \inf(X) \cong \sigma_{K_r}(X(\iota)) \cong \sigma_{K_r}(g.X(\iota)) \stackrel{\ref{4.2.2}}{\cong} g^{-\tr}.\sigma_{K_r}(X(\iota)) \cong g^{-\tr}.\inf(X).\]
Let $f \colon \inf(X) \lra g^{-\tr}.\inf(X)$ be an isomorphism. We assume that $A_d \cap g^{\tr}(A^\perp_d) \neq \{0\}$ and find $0 \neq a \in A_d^\perp$ such that $0 \neq g^{\tr}(a) \in A_d$. We conclude for $x \in \inf(X)_1 = X_1$ that 
\begin{align*}
    0 &= (f_2 \circ \psi_{\inf(X)})(a\otimes x) = (\psi_{g^{-\tr}.\inf(X)} \circ (\id_{A_r} \otimes f_1))(a \otimes x)  \\
    &= \psi_{\inf(X)} \circ (g^{\tr}(a) \otimes f_1(x)) = \psi_X(g^{\tr}(a) \otimes f_1(x)).
\end{align*} 
Since $f_1 \colon X_1 \lra X_1$ is surjective, we conclude
\[ (\dagger) \quad \psi_{X}(g^{\tr}(a) \otimes x)  = 0\]
for all $x \in X_1$. Since  $X$ is homogeneous, \cref{1.5.2} implies $\psi_X = 0$. This is a contradiction since $0 \neq X \in \rep_{1,2}(K_d)$.
\item[(ii)] Let $g \in \GL(A_r)$ such that $g \circ \iota = \alpha$, then  $A_d \cap g(A^\perp_d) = \{0\}$:
 We write
\[M_{\cB}(g^{tr}) = (\lambda_{ji})_{1 \leq j,i \leq n},\]
so for $1 \leq i \leq r$ we have $g^{\tr}(\gamma_i) = \sum^{r}_{j=1} \lambda_{ji} \gamma_j$. The assumption that the lower right hand block of $M_{\cB}(g^{\tr})$ of size $(r-d) \times (r-d)$ is not invertible, yields a non-trivial linear combination
\[ (\ast) \quad \mu_{d+1} \begin{pmatrix}
    \lambda_{d+1,d+1} \\
    \lambda_{d+2,d+1} \\
    \vdots \\
    \lambda_{r,d+1} 
\end{pmatrix} + \mu_{d+2} \begin{pmatrix}
    \lambda_{d+1,d+2} \\
    \lambda_{d+2,d+2} \\
    \vdots \\
    \lambda_{r,d+2}
\end{pmatrix}
    + \cdots + \mu_{r} \begin{pmatrix}
    \lambda_{d+1,r} \\
    \lambda_{d+2,r} \\
    \vdots \\
    \lambda_{r,r}
\end{pmatrix} =\begin{pmatrix}
    0 \\
    0 \\
    \vdots \\
    0
\end{pmatrix}.\]
We have $0 \neq \sum^r_{i=d+1} \mu_i \gamma_i \in A_d^{\perp}$ and $(\ast)$ implies that $g^{\tr}(\sum^r_{i=d+1} \mu_i \gamma_i)$ is a non-zero element in $A_d$. Hence, $A_d \cap g^{\tr}(A_d^\perp) \neq \{0\}$, in contradiction to (i). Therefore, the lower right hand block of $M_{\cB}(g) = M_{\cB}(g^{\tr})^{\tr}$ that we denote by $(c_{ji})_{d+1 \leq j,i \leq r}$ is also invertible.

Finally, we assume that $A_d \cap g(A^\perp_d) \neq \{0\}$ and find $0 \neq a = \sum^r_{i=d+1} \eta_i \gamma_i$ such that $g(a) \in A_d$. We have
\[ g(a) = \sum^r_{i=d+1} \eta_i g(\gamma_i)\]
and conclude 
\[ \begin{pmatrix}
    0 \\
    0 \\
    \vdots \\
    0
\end{pmatrix} =  \eta_{d+1} \begin{pmatrix}
    c_{d+1,d+1} \\
    c_{d+2,d+1} \\
    \vdots \\
    c_{r,d+1} 
\end{pmatrix} + \eta_{d+2} \begin{pmatrix}
    c_{d+1,d+2} \\
    c_{d+2,d+2} \\
    \vdots \\
    c_{r,d+2}
\end{pmatrix}
    + \cdots + \eta_{r} \begin{pmatrix}
    c_{d+1,r} \\
    c_{d+2,r} \\
    \vdots \\
    c_{r,r}
\end{pmatrix},\]
a contradiction.
\item[(iii)] We have $\im \alpha = A_d = \im \iota$: Let $a \in A_r \setminus \im \alpha$ and $a \in \fu$ be a $\KK$-complement of $\im \alpha$ in $A_r$. We find $g \in \GL(A_r)$ such that $g \circ \iota = \alpha$ and $g|_{A_d^\perp} \lra \fu$ is an isomorphism. By (ii), we have $A_d \cap \fu= A_d \cap g(A_d^\perp) = \{0\}$ and conclude $a \not\in A_d$. Since $\dim_{\KK} \fu = \dim_{\KK} A_d$, we conclude $\im \alpha = A_d$.
\end{enumerate}

Now we prove the general case. We let $g_{\alpha},g_{\beta} \in \GL(A_r)$ such that $g_{\alpha} \circ\iota = \alpha$ and $g_{\beta} \circ \iota = \beta$. We have $g_{\alpha}.X_{d,r} \cong X(\alpha) \cong X(\beta) \cong g_{\beta}.X_{d,r}$ and conclude
\[ X(g_{\beta}^{-1} \circ g_{\alpha} \circ \iota) \stackrel{\ref{4.1.4}(1)}{\cong} (g_{\beta}^{-1} \circ g_{\alpha}).X(\iota) =  (g_{\beta}^{-1} \circ g_{\alpha}).X_{d,r} \cong X_{d,r} \cong X(\iota).\]
Now (iii) implies $\im (g_{\beta}^{-1} \circ g_{\alpha} \circ \iota) = \im \iota$ and therefore 
\[ \im \alpha = \im (g_{\alpha} \circ \iota) = \im (g_{\beta} \circ \iota) = \im \beta.\]
This completes the proof in case $X^-(\alpha) \cong X^-(\beta)$.

Now we assume $X^+(\alpha) \cong X^+(\beta)$. As before, we conclude that we can assume that $X \in \rep_{1,2}(K_d)$. Then \cref{4.1.4} implies
\[ X^-(\alpha) \cong \tau_{K_r}^{-1}(X^+(\alpha)) \cong \tau_{K_r}^{-1}(X^+(\beta)) \cong X^-(\beta).\]
The first case implies $\im \alpha \cong \im \beta$.
\end{proof}

\bigskip
\noindent 
We summarize our findings of this section in the following result.
\bigskip

\begin{Theorem}\label{4.2.6}
Let $X \in \rep(K_d)$ be a homogeneous representation.
\begin{enumerate}
    \item If $X$ is not semisimple, then the maps
\[ \Gr_d(A_r) \lra \Iso(K_r) \ ; \ \fv \mapsto [X^{-}(\fv)] \ \text{and} \ \Gr_d(A_r) \lra \Iso(K_r) \ ; \ \fv \mapsto [X^+(\fv)]\]
are injective.
\item Let $M \in \rep(K_r)$, $\alpha \in \Inj_{\KK}(A_d,A_r)$ and $\fv := \im \alpha$. We have isomorphisms of $\KK$-vector spaces
\[ \Hom_{K_r}(X^-(\fv),M) \cong \Hom_{K_d}(X,(\sigma_{K_d} \circ \alpha^\ast)(M)) \cong \Hom_{K_d}(\sigma^{-1}_{K_d}(X),\alpha^\ast(M)),\]
\[ \Hom_{K_r}(M,X^+(\fv)) \cong \Hom_{K_d}((\sigma_{K_d}^{-1} \circ \alpha^\ast)(M),X) \cong \Hom_{K_d}(\alpha^\ast(M),\sigma_{K_d}(X)).\]
\end{enumerate}
\end{Theorem}
\begin{proof}
Follows immediately from \cref{4.2.5} and \cref{4.1.3}.
\end{proof}

\bigskip

\subsection{Application to preprojective representations}\label{S:4.3}
Let $1 \leq d < r$. In this section, we apply \cref{4.2.6} to the special case of preprojective indecomposable $K_d$-representations\footnote{For $d = 1$ we only have projective indecomposables $P_0(1)$ and $P_1(1)$. In this case, $n \in \NN_0$ means $n \in \{0,1\}$ and $P_2(1) \coloneqq  \sigma_{K_1}^{-1}(P_1(1)) = I_0(1)$.}. 
Let $i \in \NN_0$ and $d \geq 2$. We write $\dimu P_i(d) = (a_i(d),a_{i+1}(d))$ and note that $a_0(d) = 0, a_1(d) = 1$.
 We let $\delta \colon \ZZ^2 \lra \ZZ^2 \ ; \ (a,b) \mapsto (b,a)$ be the twist function on $\ZZ^2$.
Then we have
\[ \dimu P_i(d) = (a_i(d),a_{i+1}(d)) = \delta(\dimu I_i(d)).\]

Since $P_n(d) \in \rep(K_d)$ is homogeneous (see \cref{2.2.1}), we may define
\[ P_n^{+}(\fv) \coloneqq  P_n(d)^{+}(\fv) \ \text{and} \ P_n^{-}(\fv) \coloneqq  P_n(d)^{-}(\fv)\]
for all $\fv \in \Gr_d(A_r)$. In the following Lemma we collect basic properties of preprojective indecomposable Kronecker representations.

\bigskip

\begin{Lemma}\label{4.3.1} Let $1  < d$ and $n,m \in \NN_0$. 
\begin{enumerate}
    \item We have $\sigma_{K_d}^{-1}(P_n(d)) \cong P_{n+1}(d)$.
    \item We have $a_{n+2}(d) = d a_{n+1}(d) - a_{n}(d)$.
    \item We have 
\[\dim_{\KK} \Hom_{K_r}(P_n(d),P_m(d)) = \begin{cases}
0 & \text{for} \ n > m\\
a_{m-n+1}(d)\neq 0 & \text{for} \ n \leq m.
\end{cases}\]
\end{enumerate}
\end{Lemma}

\bigskip

In the following we study restrictions via elements in $\Inj_{\KK}(A_d,A_r)$ in more detail. The first result is a generalization of \cref{1.5.3}(2)(i).

\bigskip

\begin{Lemma}\label{4.3.2}
Let $M \in \rep(K_r)$ be a representation.
\begin{enumerate}
    \item Let $\alpha,\beta \in \Inj_\bmk(A_d,A_r)$ be such that $\im \alpha = \im \beta$ and let $\delta^{\ast}(M) = P_\delta \oplus R_\delta \oplus I_\delta$ be the decomposition of $\delta^{\ast}(M)$ into preprojective, regular and preinjective summands for $\delta \in \{\alpha,\beta\}$. Then $P_{\alpha} \cong P_{\beta}$, and $I_{\alpha} \cong I_{\beta}$.
    \item Let $\fv \in \Gr_d(A_r)$ and $\im \colon \Inj_{\KK}(A_d,A_r) \lra \Gr_d(A_r) ; \alpha \mapsto \im \alpha$. The following statements are equivalent:
    \begin{enumerate}[(i)]
        \item There is $\alpha \in \im^{-1}(\fv)$ such that $\alpha^{\ast}(M)$ is preprojective.
        \item For all $\alpha\in \im^{-1}(\fv)$ the representation $\alpha^{\ast}(M)$ is preprojective.
        \item For all $\alpha \in \im^{-1}(\fv)$ the representation $\alpha^{\ast}(M)$ is preprojective and $\alpha^{\ast}(M) \cong \beta^{\ast}(M)$ for all $\beta \in \im^{-1}(\fv)$.
    \end{enumerate}
\end{enumerate}
\end{Lemma}
\begin{proof}
\begin{enumerate}
    \item We have $g \coloneqq  \beta^{-1} \circ \alpha \in \GL(A_d)$. Note that $\psi_{g.\alpha^\ast(M)} = \psi_{M} \circ (\alpha \circ g^{-1} \otimes \id_{M_1}) = \psi_{M}\circ (\beta \otimes \id_{M_1})$ and therefore $\beta^\ast(M) = g.(\alpha^\ast(M))$. We conclude
\[ P_{\beta} \oplus R_{\beta} \oplus I_{\beta}  = \beta^{\ast}(M) = g.(\alpha^\ast(M)) \cong g.P_{\alpha} \oplus g.R_{\alpha} \oplus g.I_{\alpha}.\]
Since preprojective and preinjective representations are homogeneous (see \cref{2.2.1}), we obtain
\[ P_{\beta} \oplus R_{\beta} \oplus I_{\beta}  = \beta^{\ast}(M) \cong g.(\alpha^\ast(M)) \cong P_{\alpha} \oplus g.R_{\alpha} \oplus I_{\alpha}.\]
As regular representations are closed under the action of $\GL(A_d)$, the claim follows from Krull–Remak–Schmidt.

\item It suffices to prove (i) $\Rightarrow$ (iii). By (i), there is $\alpha \in \im^{-1}(\fv)$ such that $\alpha^{\ast}(M)$ is preprojective. Let $\beta \in \im^{-1}(\fv)$. Then $\im \alpha= \im \beta$ and $(1)$ implies that $P_{\alpha} \cong P_{\beta}$. Since $\dimu \beta^{\ast}(M) = \dimu \alpha^{\ast}(M) = \dimu P_{\alpha} =\dimu  P_{\beta} \leq \dimu \beta^{\ast}(M)$, we conclude that $P_{\beta} = \beta^{\ast}(M)$. Hence, $\beta^{\ast}(M)$ is preprojective with 
\[\beta^{\ast}(M) =  P_{\beta} \cong P_{\alpha} = \alpha^{\ast}(M).\]
\end{enumerate}
\end{proof}

\bigskip

\begin{Definition}\label{4.3.3}
Let $M \in \rep(K_r)$ and $\fv \in \Gr_d(A_r)$. For each $\alpha \in \Inj_{\bmk}(A_d,A_r)$ such that $\im \alpha = \fv$ we have a decomposition 
\[ \alpha^{\ast}(M) = \alpha^{\ast}(M)_{\pproj} \oplus \alpha^{\ast}(M)_{\reg} \oplus \alpha^{\ast}(M)_{\pinj}\]
into preprojective, regular and preinjective summands. We define 
\[ M|_{\fv,\pproj} \coloneqq  \alpha^{\ast}(M)_{\pproj} \quad \text{and} \quad  M|_{\fv,\pinj} \coloneqq  \alpha^{\ast}(M)_{\pinj},\]
which is well-defined up to isomorphism by \cref{4.3.2}. 

We say that $M|_{\fv}$ is \textit{(pre)projective} if $\alpha^\ast(M)$ is (pre)projective for some $\alpha \in \Inj_{\KK}(A_d,A_r)$ with $\im \alpha = \fv$. By \cref{4.3.2}, this is equivalent to $\beta^\ast(M)$ being preprojective for all $\beta \in \Inj_{\KK}(A_d,A_r)$ such that $\im \beta = \fv$ and moreover, we have an isomorphism of representations $\alpha^\ast(M) \cong \beta^\ast(M)$ in this case.
\end{Definition}

\bigskip

\begin{proposition}\label{4.3.4}
Let $M \in \rep(K_r)$, $\fv \in \Gr_d(A_r)$, $n \in \NN_0$. We write $M|_{\fv,\pproj} = \bigoplus_{i \geq 0} b_i(\fv) P_i(d)$.
\begin{enumerate}
    \item If $M|_{\fv}$ is preprojective, we have 
    \[ \dim_\bmk \Hom_{K_r}(P_n^-(\fv),M) = \sum_{i \geq n +1} b_i(\fv) \dim_\bmk \Hom_{K_d}(P_{n+1}(d), P_i(d)).\]
    \item The followings statements are equivalent.
    \begin{enumerate}
        \item[(i)] $\Hom_{K_r}(P_n^-(\fv),M) = 0$.
        \item[(ii)] $M|_{\fv}$ is preprojective and $M|_{\fv} \in \add(P_0(d),\ldots,P_n(d))$.
    \end{enumerate}
   \item We have $\dim_\bmk \Hom_{K_r}(M,P_n^+(\fv)) = \sum_{i = 0}^{n - 1} b_i(\fv) \dim_{\KK} \Hom_{K_d}(P_{i}(d),P_{n-1}(d))\footnote{Note that $n = 0$ gives us the empty sum.}$.
   \item The following statements are equivalent.
     \begin{enumerate}
        \item[(i)] $\Hom_{K_r}(M,P_n^+(\fv)) = 0$.
        \item[(ii)] $M|_{\fv,\pproj} \in \add(\{P_i(d) \mid i \geq n\})$.
    \end{enumerate}
\end{enumerate}
\end{proposition}
\begin{proof}
\begin{enumerate}
    \item Let $\alpha \in \Inj_{\KK}(A_d,A_r)$ be such that $\im \alpha = \fv$. By assumption, we have $\alpha^\ast(M) \cong M|_{\fv,\pproj}$ and conclude with \cref{4.2.6} that
    \begin{align*}
        \dim_{\KK} \Hom_{K_r}(P_n^-(\fv),M)  &= \dim_{\KK} \Hom_{K_d}(P_{n+1}(d),\alpha^\ast(M))\\
        &= \dim_\bmk \Hom_{K_d}(P_{n+1}(d),\bigoplus_{i \in \NN_0} b_i(\fv) P_i(d))\\
        &\stackrel{\ref{4.3.1}}{=} \sum_{i \geq n+1} b_i(\fv) \dim_\bmk \Hom_{K_d}(P_{n+1}(d), P_i(d)).
    \end{align*}
    \item 
    (i) $\Rightarrow$ (ii). Let $0 \neq X \in \rep(K_d)$ be such that $X$ does not have a preprojective direct summand. Then $\sigma_{K_d}^{n}(X)$ does not have $P_0(r)$ as a direct summand. In particular, we have  $(\sigma_{K_d}^{n}(X))_1 \neq 0$ and we obtain
   \begin{align*}
       \Hom_{K_d}(P_{n+1}(d),X) &\cong \Hom_{K_d}(\sigma_{K_d}^{n}(P_{n+1}(d)), \sigma_{K_d}^{n}(X)) \\
       &\cong \Hom_{K_d}(P_1(d), \sigma_{K_d}^{n}(X)) \cong (\sigma_{K_d}^{n}(X))_1 \neq 0.
   \end{align*}
   Hence, the assumption in conjunction with \cref{4.2.6} implies that every indecomposable direct summand of $\alpha^{\ast}(M)$ is preprojective, i.e., $M|_{\fv}$ is preprojective. Now we apply (1) and conclude
   \[ 0 = \sum_{i \geq n+1} b_i(\fv) \dim_{\KK} \Hom_{K_d}(P_{n+1}(d),P_i(d)).\]
   \cref{4.3.1} implies $b_i(\fv) = 0$ for all $i \geq n +1$. Therefore, we have $M|_{\fv} = M|_{\fv,\pproj} = \bigoplus_{i=0}^n b_i(\fv) P_i(d) \in  \add(P_0(d),\ldots,P_n(d))$.

   (ii) $\Rightarrow$ (i). We have $M|_{\fv} = M|_{\fv,\pproj} = \sum^n_{i=0} b_i(\fv) P_i(d)$. Hence, (1) implies 
   \[\dim_\bmk \Hom_{K_r}(P_n^-(\fv),M) =\sum_{i \geq n +1} b_i(\fv) \dim_\bmk \Hom_{K_d}(P_{n+1}(d), P_i(d)) = 0.\]
   \item Let $\alpha \in \Inj_{\KK}(A_d,A_r)$ such that $\im \alpha = \fv$. We write $\alpha^{\ast}(M) = M|_{\fv,\pproj} \oplus L$ such that $L$ does not have a preprojective direct summand. We have $\Hom_{K_d}(L,P_{n-1}(d)) = 0$ (see \cite[(VIII.2.13)]{ASS06}) and obtain 
   \begin{align*}
        \dim_\bmk \Hom_{K_r}(M,P_n^+(\fv)) &\stackrel{\ref{4.2.6}}{=} \dim_\bmk \Hom_{K_d}(\alpha^{\ast}(M),P_{n-1}(d))\\
        &= \sum_{i \in \NN_0} b_i(\fv) \dim_\bmk \Hom(P_i(d),P_{n-1}(d))\\
        &\stackrel{\ref{4.3.1}}{=} \sum_{i = 0}^{n - 1} b_i(\fv) \dim_{\KK} \Hom_{K_d}(P_{i}(d),P_{n-1}(d)).
    \end{align*}
    \item Follows immediately from (3) and \cref{4.3.1}.
\end{enumerate}
\end{proof}

\bigskip
\noindent 
We arrive at the main result of this section:

\bigskip

\begin{Theorem}\label{4.3.5}
Let $M \in \rep(K_r)$, $n \in \NN_0$ and $\fv \in \Gr_d(A_r)$. The following statements are equivalent.
\begin{enumerate}
    \item $M|_{\fv}\!=\!(-a_{n+2}(d)\! \dim_{\KK}\!M_1\!+\!a_{n+1}(d)\!\dim_{\KK}\!M_2)P_n(d)\oplus (a_{n+1}(d)\!\dim_{\KK}\!M_1\!-\!a_{n}(d)\!\dim_{\KK}\!M_2) P_{n+1}(d)$.
    \item $M|_{\fv} \in \add(P_n(d),P_{n+1}(d))$.
    \item $\Hom_{K_r}(P_{n+1}^{-}(\fv),M) = 0 = \Hom_{K_r}(M,P^+_{n}(\fv))$.
\end{enumerate}
\end{Theorem}
\begin{proof}
(1) $\Rightarrow$ (2). This is clear.

(2) $\Rightarrow$ (3).  We have $M|_{\fv} = M|_{\fv,\pproj} = b_n(\fv) P_n(d) \oplus b_{n+1}(\fv) P_{n+1}(d)$. Hence, $M|_{\fv} \in \add(\{P_i(d) \mid i \geq n\}) \cap\add(\{P_i(d) \mid i \leq n+1\})$ and \cref{4.3.4}(2)+(4) implies
\[\dim_\bmk \Hom_{K_r}(P^-_{n+1}(\fv),M) = 0 = \dim_\bmk \Hom_{K_r}(M,P_{n}^+(\fv)).\]

(3) $\Rightarrow$ (1). We apply \cref{4.3.4}(2) and conclude that $M|_{\fv}$ is preprojective and of the form $M|_{\fv} = \bigoplus_{i=0}^{n+1} b_i(\fv) P_i(d)$. Another application of \cref{4.3.4}(4) gives us $b_i(\fv) = 0$ for all $0 \leq i \leq n-1$. Hence, $M|_{\fv} \in \add(P_n(d),P_{n+1}(d))$. 
We write
\[ M|_{\fv} = b_n(\fv) P_n(d) \oplus b_{n+1}(\fv) P_{n+1}(d).\]
Recall that 
\begin{align*}
    1 &= q_d(\dimu P_n(d)) = q_d(a_n(d),a_{n+1}(d)) \\
&= a_{n+1}(d)^2 + a_n(d)(a_n(d)-d a_{n+1}(d)) \stackrel{\ref{4.3.1}}{=} a_{n+1}(d)^2 - a_n(d) a_{n+2}(d).
\end{align*} 
Hence,
\[ 
A \coloneqq  \begin{pmatrix}
a_{n}(d) & a_{n+1}(d) \\
a_{n+1}(d) & a_{n+2}(d)
\end{pmatrix}
\]
is invertible with $A^{-1} = \frac{1}{-\det(A)} \begin{pmatrix}
-a_{n+2}(d) & a_{n+1}(d) \\
a_{n+1}(d) & -a_n(d)
\end{pmatrix} = \begin{pmatrix}
-a_{n+2}(d) & a_{n+1}(d) \\
a_{n+1}(d) & -a_n(d)
\end{pmatrix}$.
We conclude 
\[
\begin{pmatrix}
b_n(\fv)\\
b_{n+1}(\fv) 
\end{pmatrix} = 
A^{-1} \begin{pmatrix}
\dim_{\KK} M_1 \\
\dim_{\KK} M_2 
\end{pmatrix}
= 
\begin{pmatrix}
- a_{n+2}(d) \dim_{\KK} M_1 + a_{n+1}(d) \dim_{\KK} M_2 \\
a_{n+1}(d) \dim_{\KK} M_1 - a_{n}(d) \dim_{\KK} M_2 \\
\end{pmatrix}.
\]
\end{proof}

\bigskip

\subsection{From the ad hoc construction to the general case}
We end this section by showing that our approach generalizes the considerations in \cite[(2.13),(2.14)]{BF24}, which are based on the constructions in \cite{Wor13a} and \cite[(2.2)]{Bis20}.

\bigskip

\begin{proposition}\label{4.4.1}
Let $n \in \NN$ and $\fv \in \Gr_d(A_r)$. The following statements hold.
\begin{enumerate}
    \item $P_n^{\pm}(\fv)$ is a regular brick.
    \item The representation $P_n^-(\fv)|_{\fv}$ is preprojective and in $\add(P_0(d),\ldots,P_{n+1}(d))$.
    \item $P_n^-(\fv)|_{\fv} \cong (r-d) a_{n+1}(d) P_0(d) \oplus P_{n+1}(d)$.
\end{enumerate}
\end{proposition}
\begin{proof}
\begin{enumerate}
 \item Follows from \cref{4.1.4}(3)(4) since preprojective indecomposable representations are bricks by \cite[(VIII.2.7)]{ASS06}
 \item  By \cref{4.1.4} and \cref{4.3.1}, we have
 \[\Hom_{K_r}(P_{n+1}^-(\fv),P_n^-(\fv)) \cong \Hom_{K_d}(P_{n+1}(d),P_n(d)) = 0\]
 and application of \cref{4.3.4} gives us \[ P_n^-(\fv)|_{\fv} = P_n^-(\fv)|_{\fv,\pproj} \in \add(P_0(d),\ldots,P_{n+1}(d)).\] 
 \item By (2), $P_n^-(\fv)|_{\fv}$ is preprojective and we have $P_n^-(\fv)|_{\fv} = \sum^{n+1}_{i=0} b_i(\fv) P_i(d)$. Hence, application of (1), \cref{4.2.6} and \cref{4.3.1}   imply 
 \[ 1 = \dim_{\KK} \End_{K_r}(P_n^-(\fv)) \cong \dim_{\KK} \Hom_{K_d}(P_{n+1}(d),P_n^-(\fv)|_{\fv}) = b_{n+1}(\fv).\]
 Hence, $P_{n+1}(d)$ appears in the direct sum decomposition of $P_n^-(\fv)|_{\fv}$ with multiplicity $1$. We have 
 \[ \dimu P_n^-(\fv) = \sigma_r^{-1}(a_n(d),a_{n+1}(d)) = (a_{n+1}(d),r a_{n+1}(d) - a_n(d)).\]
 Since $\dimu P_{n+1}(d) = (a_{n+1}(d),a_{n+2}(d))$ and $a_{n+2}(d) \stackrel{\ref{4.3.1}}{=} d a_{n+1}(d) - a_n(d)$, we conclude
 \[P_n^-(\fv)|_{\fv} \cong (r-d) a_{n+1}(d) P_0(d) \oplus P_{n+1}(d).\]
 \end{enumerate}
\end{proof}

\bigskip

\begin{proposition}\label{4.4.2}
Let $1 \leq d < r$ and $\fv \in \Gr_d(A_r)$.
\begin{enumerate}
\item The representation $P_1^{\pm}(\fv)$ is \textit{elementary}, i.e., $P_1^{\pm}(\fv)$ is regular and there is no short exact sequence $0 \lra A \lra P_1^{\pm}(\fv) \lra B \lra 0$ with $A,B \in \rep(K_r)$ regular and non-zero. 
\item We have $ P_1^+(\fv) \cong \tau_{K_r}(E(\fv))$ and $P_1^-(\fv) \cong E(\fv)$.
\item Let $1 < d < r$ and $\fv \in \Gr_d(A_r)$, then $P_1^-(\fv) \in \rep_{\proj}(K_r,d-1)$.
\end{enumerate}
\end{proposition}
\begin{proof}
\begin{enumerate}
\item Let $E \in \rep(K_r)$. Since twist and shifts of $E$ are elementary if and only if $E$ is regular, it suffices to prove that $\inf(P_1(d))$ is elementary. According to \cite[(3.2.2)]{Bis20}, $\inf(P_1(d))$ is regular. If $0 \lra A \lra \inf(P_1(d)) \lra B \lra 0$ is a short exact sequence, then either $\dim_{\KK} A_1 = 0$ or $\dim_{\KK} B_2 = 0$ since $\dimu P_1(d) = (1,d)$. Hence, $A$ or $B$ is not a non-zero regular representation.
\item Recall the definition of $E(\fv)$ from the beginning of \cref{S:4}. Let $\alpha \in \Inj_{\KK}(A_d,A_r)$ such that $\im \alpha = \fv$. By \cref{4.4.1}, we have 
\[\alpha^\ast(P_1^-(\fv)) = (r-d)a_{2}(d) P_0(d) \oplus P_{2}(d).\]
We conclude with \cite[(2.1.5)]{BF24} that $0\neq \Hom_{K_r}(E(\fv),P_1^-(\fv))$. By (1) and \cite[(2.1.3)]{BF24}, both representations are elementary. Since they also have the same dimension vector $(d,rd-1)$, we conclude with \cite[(1.4)]{KL96} that $E(\fv) \cong P_1^-(\fv)$. Moreover, we have
\[ \tau_{K_r}(E(\fv)) \cong \tau_{K_r}(P_1^-(\fv)) \stackrel{\ref{4.1.4}}{\cong} P_1^+(\fv).\]
\item Let $\fu \in \Gr_{d-1}(A_r)$. In view of \cref{4.3.5} (note that $P_1^+(d) = \{0\})$, it suffices to show that
\[0 = \Hom_{K_r}(P_1^-(\fu),P_1^-(\fv)).\] By (2), we find $\alpha \in \Inj_{\KK}(A_d,A_r)$ and $\beta \in \Inj_{\KK}(A_{d-1},A_r)$ such that 
\[ P_1^-(\fv) \cong D_{K_r}(\tau_{K_r}(\coker \overline{\alpha})) \ \text{and} \ P_1^-(\fu) \cong D_{K_r}(\tau_{K_r}(\coker \overline{\beta})).\]
Since all involved representations are regular, we conclude
\[ \Hom_{K_r}(P_1^-(\fu),P_1^-(\fv)) \cong \Hom_{K_r}(\coker \overline{\beta},\coker \overline{\alpha}).\] 
We have $\dimu \coker \overline{\beta} = (1,r-d+1)$ and $\dimu \coker \overline{\alpha} = (1,r-d)$. Note that every proper subrepresentation of $\coker \overline{\alpha}$ is projective. Since $\coker \overline{\beta}$ is regular and there are no non-zero morphisms from regular to preprojective representation (see \cite[(VIII.2.13)]{ASS06}), we conclude  with $\dim_{\KK}  \coker \overline{\beta} >\dim_{\KK} \coker \overline{\alpha}$ that $\Hom_{K_r}(\coker \overline{\beta},\coker \overline{\alpha}) = 0$.
\end{enumerate}
\end{proof}

\bigskip

\begin{corollary}\label{4.4.3}
For $M \in \rep(K_r)$ we have $\cV(K_r,d)_M = \{ \fv \in \Gr_d(A_r) \mid \Hom_{K_r}(P_1^-(\fv),M) \neq 0\}$.
\end{corollary}

\begin{proof}
This is a direct consequence of \cref{4.4.2} and \cite[(2.1.5)]{BF24}.
\end{proof}

\bigskip

\begin{Remark}\label{4.4.4}
For $d = 1$,  $\alpha \in \Inj_{\KK}(A_1,A_r)$ and $\fv \coloneqq \im \alpha$, we obtain $P_1(\fv) \cong E(\fv) \cong X(\alpha)$, where $X(\alpha)$ is the representation defined in \cite[\S 3]{Wor13a}.
\end{Remark}

\bigskip

\section{General subrepresentations and uniform representations}\label{S:5}

In this section, we use the test representation obtained from preprojective indecomposable $K_2$-representations, together with a recent result on general subrepresentations for Kronecker quivers, to prove the existence of indecomposable uniform Kronecker representations $M \in \rep(K_r)$ with
\[ M|_{K_r} = b_n(M) P_n(2) \oplus b_{n+1}(M) P_{n+1}(2)\]
for each $n \in \NN_0$ and $r \geq 3$. Throughout this section, $(V_1,V_2)$ denotes a pair of finite-dimensional $K$-vector spaces with $V_1 \oplus V_2 \neq \{0\}$ and $r \geq 3$.

\subsection{Preliminaries}\label{S:5.1}

The following gathers notation and a recent result on general subrepresentations, both of which will be used repeatedly in the sequel. We set $d \coloneqq (\dim_\bmk V_1,\dim_\bmk V_2) \neq (0,0)$.  
For $e \in \NN^2_0$ we let $\rep(K_r;V_1,V_2)_{e}$ denote the Zariski-closed subset of $\rep(K_r;V_1,V_2)$ (cf.\ \cite[(3.1)]{Sch92}) consisting of all representations that admit a subrepresentation of dimension vector $e$.  
We then define
\[
\cO(e,\not \hookrightarrow) \coloneqq \cV(K_r;V_1,V_2) \setminus \cV(K_r;V_1,V_2)_{e}.
\]
If $e$ is a positive root of $K_r$, the following result holds; see \cite[(3.4)]{Rei24} and \cite[(2.9)]{Bis25}.

\bigskip

\begin{proposition}\label{5.1.1}
The following statements are equivalent for $e \in \NN^2_0$ with $e \leq d$ (componentwise) and $q_r(e) \leq 1$.
\begin{enumerate}
    \item $\cO(e,\not \hookrightarrow) \neq \emptyset$.
    \item $\langle e,d - e\rangle_r < 0$.
\end{enumerate}
\end{proposition}

\bigskip

By definition, $\cO(e,\not \hookrightarrow) \subseteq \cV(K_r;V_1,V_2)$ is the open subset consisting of all representations in $\cV(K_r;V_1,V_2)$ that do not contain a subrepresentation of dimension vector $\be$.  In what follows, we apply \cref{5.1.1} to the dimension vectors  
\[
f^-_{n,r} \coloneqq (n+1,(r-1)(n+1)+1) 
\qquad  \text{and} \qquad 
f^+_{n,r} \coloneqq (n(r-1)-1,n),
\]  
that satisfy  
$\dimu f^{\pm}_{n,r} = \dimu P_n^{\pm}(\fv)$ for all $\fv \in \Gr_2(A_r)$ and all $n \in \NN$ (cf. \cref{S:4.3}) since $\dimu P_n(2) = (n,n+1)$.

\bigskip

\subsection{Uniform and relative projective representations}

Aim of this section is to prove that for $r \geq 3$ and $n \in \NN_0$, there exists uniform but non-homogeneous representations in $M \in \rep_{\proj}(K_r,1)$ such that
\[ M|_{K_2} \cong [(n+1) \dim_\bmk M_2 - (n+2) \dim_\bmk M_1] P_n(2) \oplus [(n+1)\dim_\bmk M_1 - n \dim_\bmk M_2] P_{n+1}(2).\] 
This, in turn, will later enable us to establish the existence of new uniform Steiner bundles.

\bigskip

\begin{proposition}\label{5.2.1} 
Let $n \in \NN$ and $(V_1,V_2)$ be a pair of $\KK$-vector spaces.
\begin{enumerate}
    \item If $n \dim_\bmk V_2 - (n+1) \dim_\bmk V_1 \geq  n(n+1) (r-2)$, then $\cO(f^-_{n,r},\not \hookrightarrow) \neq \emptyset$. 
    \item Assume that $\dim_\bmk V_2 \geq (n+1)(r-1)+1$ and let $M \in \cO(f^-_{n,r},\not \hookrightarrow) \cap \rep_{\proj}(K_r,1)$. We have $0 = \Hom_{K_r}(P_{n}^-(\fv),M)$ for all $\fv \in \Gr_2(A_r)$.
\end{enumerate}
\end{proposition}
\begin{proof}
\begin{enumerate}
\item  We can assume that $f^-_{n,r} \leq \dimu (V_1,V_2)$, otherwise $\cO(f^-_{n,r},\not \hookrightarrow) = \cV(K_r;V_1,V_2)$ and we are done. We have
\begin{align*}
    &\langle f^-_{n,r},\dimu(V_1,V_2) -  f^-_{n,r} \rangle_r = \langle (n+1,(r-1)(n+1)+1),\dimu(V_1,V_2)\rangle_r -q_r(f^-_{n,r})\\
    = &(n+1) \dim_\bmk V_1 - n\dim_\bmk V_2 -  q_r(f^-_{n,r}) \leq -(n(n+1)(r-2)) -q_r(\sigma_r(f^-_{n,r}))  \\
    = & -(n(n+1)(r-2)) - (q_2(P_n(2)) - (r-2)n(n+1)) = -1 < 0. 
\end{align*}
Since $q_r(\dimu f^-_{n,r}) = q_r(n,n+1) \leq 1$, we can apply \cref{5.1.1}.
\item Let $\fv \in \Gr_2(A_r)$. We assume that there exists $f \colon P_{n}^-(\fv) \lra M$ non-zero. Since $\rep_{\proj}(K_r,1)$ is closed under subrepresentations (see \cref{2.3.6}), we have $0 \neq \im f \in \rep_{\proj}(K_r,1)$. For $\alpha \in \Inj_{\KK}(A_d,A_r)$ such that $\im \alpha = \fv$, we have $\alpha^\ast(\im f) \in \rep_{\proj}(K_2,1)$. Hence, $\im f|_{\fv}$ is preprojective by \cite[(4.3)]{Wor13b} we write $\im f|_{\fv} = \bigoplus_{i\geq 0} b_i(\fv) P_i(2)$. Now \cref{4.3.4} implies
\[ 0 \neq \dim_\bmk \Hom_{K_r}(P_{n}^-(\fv),\im f) = \sum_{i \geq n+1} b_i(\fv) \dim_{\KK}\Hom_{K_2}(P_{n+1}(2),P_i(2)).\]
Hence, $b_{i}(\fv) > 0$ for some $i \geq n+1$. In particular, $n+1 = \dim_\bmk (P_{n}^-(\fv))_1 \geq \dim_\bmk (\im f)_1 \geq \dim_\bmk (P_i(2))_1 = i \geq n+1$. We conclude $i = n + 1 = \dim_\bmk (\im f)_1$. By assumption, we have
\[ \ell \coloneqq \dim_\bmk (\im f)_2 \leq \dim_\bmk (P_n^-(\fv))_2 = (n+1)(r-1)+1 \leq \dim_\bmk M_2.\]
Therefore, find a subrepresentation $Y \subseteq M$ such that
$Y \cong [{(n+1)(r-1)+1-\ell}]P_0(r)$ and $Y \cap \im f = \{0\}$. Hence, $\dimu (Y \oplus \im f) = (n+1,(n+1)(r-1)+1) = f^-_{n,r}$, contradicting our assumption.
\end{enumerate}
\end{proof}

\bigskip
Given $e \leq d = \dimu(V_1,V_2)$, we consider the open subset
\[
\cO(e,\not \twoheadrightarrow) \coloneqq \cV(K_r;V_1,V_2) \setminus \cV(K_r;V_1,V_2)_{\dimu(V_1,V_2) - \be}
\]
consisting of all representations that do not have a factor of dimension vector $\be$. Since the duality $D_{K_r} \colon \rep(K_r) \lra \rep(K_r)$ interchanges quotients and subrepresentations, we have
\begin{align*}
    \cO(e,\not \twoheadrightarrow) 
    &= \cV(K_r;V_1,V_2) \setminus \cV(K_r;V_1,V_2)_{\dimu(V_1,V_2) - \be} \\
    &\cong \cV(K_r;V_2,V_1) \setminus \cV(K_r;V_2,V_1)_{\delta(\be)},
\end{align*}
where $\delta \colon \ZZ^2 \lra \ZZ^2$ denotes the twist. 

\bigskip

\begin{proposition}\label{5.5.2} 
Let $n \in \NN$ and $(V_1,V_2)$ be a pair of $\KK$-vector spaces.
\begin{enumerate}
    \item If $(n+1)\dim_{\KK} V_1 - n\dim_{\KK} V_2 \geq n(n+1)(r-2)$
    and $\Delta_{(V_1,V_2)} \geq 0$, then $\bigcap^n_{l=1} \cO(f^+_{l,r},\not \twoheadrightarrow) \neq \emptyset$.
    \item Assume that $\dim_\bmk V_1 \geq n(r-1)-1$. 
\begin{enumerate}
    \item[(i)] Let $M \in \cV(K_r;V_1,V_2)$ and assume that there is $l \in \{1,\ldots,n\}$, a representation $U \in \rep(K_r)$ with $\dim_{\KK} U_1 \leq l(r-1)-1$, $\dim_{\KK} U_2 = l$ and an epimorphism $M \twoheadrightarrow U$. Then $M  \notin \bigcap^n_{i=1} \cO(f^+_{i,r},\not \twoheadrightarrow)$.
    \item[(ii)] We have $0 =  \Hom_{K_r}(M,P_n^+(\fv))$
for all $M \in \bigcap^n_{l=1} \cO(f^+_{l,r},\not \twoheadrightarrow)$ and all $\fv \in \Gr_2(A_r)$.
\end{enumerate}
\item If $(n+1)\dim_\bmk V_1 - n \dim_\bmk V_2 \geq n(n+1)(r-2)$ and $\Delta_{(V_1,V_2)} \geq r-1$, then $\dim_\bmk V_1 \geq n(r-1)-1$ and $\bigcap^n_{l=1} \cO(f^+_{l,r},\not \twoheadrightarrow) \neq \emptyset$. Moreover, we have $0 = \dim_\bmk \Hom_{K_r}(M,P_n^+(\fv))$
for all $M \in \bigcap^n_{l=1} \cO(f^+_{l,r},\not \twoheadrightarrow)$ and all $\fv \in \Gr_2(A_r)$.
\end{enumerate}

\bigskip

\end{proposition}
\begin{proof}
\begin{enumerate}
    \item Let $1 \leq l \leq n$. We have
\begin{align*}
   (l+1) \dim_\bmk V_1 - l \dim_\bmk V_2 &= (n+1) \dim_\bmk V_1 - n \dim_\bmk V_2 + (n-l) \Delta_{(V_1,V_2)}\\
    &\geq n(n+1)(r-2) \geq l(l+1)(r-2).
    \end{align*}
and
\begin{align*}
    q_r(\delta(f^+_{l,r})) &= q_r(f^+_{l,r}) =  q_r(\sigma_{r}^{-1}(f^+_{l,r})) = q_r(\dimu P_l(2)) =  q_2(\dimu P_l(2)) - (r-2)l(l+1) \\
    &= 1-  (r-2)l(l+1),
\end{align*}
where $\delta \colon \ZZ^2 \lra \ZZ^2 ; (x,y) \mapsto (y,x)$. We conclude
\begin{align*}
\langle \delta(f^+_{l,r}),\dimu(V_2,V_1) -  \delta(f^+_{l,r})\rangle_r  &= l \dim_{\KK} V_2 - (l+1) \dim_{\KK} V_1 -(1-  (r-2)l(l+1)) \\
  &  \leq -l(l+1)(r-2) -1 + (r-2)l(l+1) < 0.
\end{align*}
Moreover, we have $q_r(\delta(f^+_{l,r})) \leq 1$. As in the proof of \cref{5.2.1} we may assume without loss of generality $f^+_{l,r} \leq \dimu(V_1,V_2)$. Now \cref{5.1.1} and taking duals imply $\cO(f^+_{l,r},\not \twoheadrightarrow) \neq \emptyset$. The statement follows since $\cV(K_r;V_1,V_2)$ is irreducible and we take an intersection of finitely many non-empty open sets.
\item  
\begin{enumerate}
    \item[(i)] 
     Denote by $K = (K_1,K_2,\psi_{M}|_{A_r \otimes_{\KK} K_1})$ the kernel of the epimorphism $M \twoheadrightarrow U$ under consideration. We have
    \begin{align*}
        \dim_\bmk K_1 &= \dim_{\KK} V_1 - \dim_{\KK} U_1 \geq \dim_{\KK} V_1 - (l(r-1)-1) \\
        &\geq \dim_{\KK} V_1 - (n(r-1)-1) \geq  0.
    \end{align*}
   In particular, we find a subspace $W_1 \subseteq K_1$ of dimension $\dim_{\KK} V_1 - (l(r-1)-1)$. We consider the subrepresentation $W$ of $K$ corresponding to $(W_1,K_2)$. We have $\dimu W = (\dim_{\KK} V_1 - (l(r-1)-1),\dim_\bmk V_2 - l)$ and therefore $\dimu (M/W) = (l(r-1)-1,l) = f^+_{l,r}$. Hence, $M \not\in  \cO(f^+_{l,r},\not \twoheadrightarrow) \supseteq \bigcap^{n}_{i =1} \cO(f^+_{i,r},\not \twoheadrightarrow)$.
\item[(ii)] Let $M \in \bigcap^n_{l=1} \cO(f^+_{l,r},\not \twoheadrightarrow)$ and assume that $0 \neq \Hom_{K_r}(M,P_{n}^+(\fv))$ for some $\fv \in \Gr_2(A_r)$. We find $l \in \{1,\ldots,n\}$ minimal such that $0 \neq \Hom_{K_r}(M,P^+_{l}(\fv))$, and a non-zero morphism $f \colon M \lra P^+_{l}(\fv)$. We set $U \coloneqq \im f \subseteq P^+_{l}(\fv)$ and note that $\dim_\bmk U_2 \neq 0$, since $I_0(r)$ is not a direct summand of $P^+_{l}(\fv)$ by \cref{4.4.1}. 

We write $U|_{\fv,\pproj} = \bigoplus_{i \in \NN_0} b_i(\fv) P_i(2)$ and conclude with \cref{4.3.4}(3) 
        \[ (\ast) \quad 0 \neq \dim_\bmk \Hom_{K_r}(U,P^+_{l}(\fv)) = \sum^{l-1}_{i=0} b_i(\fv) {\dim_\bmk \Hom_{K_2}(P_{i}(2),P_{l-1}(2))}.\]
If $l = 1$, we have $\dimu P^+_{l}(\fv) =  f^+_{1,r} = (r-2,1)$ and therefore $0 < \dim_\bmk U_2 \leq 1 = l$, i.e., $\dim_\bmk U_2 = l$ and $\dimu U = (k,l)$ with $k \leq r-2 = l(r-1)-1$. This is a contradiction to (i). Hence, $l > 1$. Minimality of $l$ implies  $\Hom_{K_r}(M,P_{l-1}^+(\fv)) = 0$ and left exactness of $\Hom_{K_r}(-,P_{l-1}^+(\fv))$ in conjunction with  \cref{4.3.4}(3) implies 
         \[ 0 = \dim_\bmk \Hom_{K_r}(U,P_{l-1}^+(\fv)) = \sum_{i = 0}^{l-2} b_i(\fv)  \underbrace{\dim_\bmk \Hom_{K_2}(P_{i}(2),P_{l-2}(2))}_{\neq 0}.\]
        Hence, $b_i(\fv) = 0$ for all $0 \leq i \leq l-2$ and $(\ast)$ simplifies to 
        \[0 \neq b_{l-1}(\fv) \Hom_{K_2}(P_{l-1}(2),P_{l-1}(2)) = b_{l-1}(\fv).\]
        In particular, $\dim_{\KK} U_2 \geq \dim_{\KK} (P_{l-1}(2))_2 = l$.
         Since 
\[\dimu U \leq \dimu P_{l}^+(\fv) = (l(r-1)-1,l),\] we get $\dim_{\KK} U_1 \leq l(r-1)-1$ and $\dim_\bmk U_2 = l$. This is a contradiction to (i).
\end{enumerate}
\item By (1), we have 
$\bigcap^{n}_{l=1} \cO(f^+_{l,r},\not \twoheadrightarrow) \neq \emptyset$ and 
\[\dim_{\KK} V_1 \geq n(n+1)(r-2)+n \Delta_{(V_1,V_2)} \geq n(n+1)(r-2) +n(r-1) \geq n(r-1)-1.\] Hence, the conditions of (2) are satisfied and the statement follows from (2)(ii).
\end{enumerate}
\end{proof}

\bigskip
\noindent Now we can prove the main result of this section.

\bigskip

\begin{Theorem}\label{5.2.3}
Let $n \in \NN$, $r \geq 3$ and $(V_1,V_2)$ be a pair of $\KK$-vector spaces such that
\[  (n+1)\dim_\bmk V_1 - n \dim_\bmk V_2 \geq n(n+1)(r-2) \quad (\mathrm{I}) \]
and 
\[ (n+1) \dim_\bmk V_2 - (n+2) \dim_\bmk V_1 \geq (n+1)(n+2)(r-2) \quad (\mathrm{II}).\] 
There exists a non-empty open subset $\cO \subseteq \rep_{\proj}(K_r,1) \cap \cV(K_r;V_1,V_2)$ such that every representation in $M \in \cO$ is a non-homogeneous brick, uniform and
    \[ M|_{K_2} \cong [(n+1) \dim_\bmk V_2 - (n+2) \dim_\bmk V_1] P_n(2) \oplus [(n+1)\dim_\bmk V_1 - n \dim_\bmk V_2] P_{n+1}(2).\]
\end{Theorem}
\begin{proof}
Note that $(\mathrm{I}) + (\mathrm{II})$ gives us 
\[ (\ast) \quad \Delta_{(V_1,V_2)} \geq (2n+2)(n+1)(r-2) \geq 8(r-2) \geq r - 1.\]

This shows that the assumptions of \cref{2.3.1}, \cref{5.5.2}(1) and \cref{5.2.1}(1) are satisfied. In particular,  
\[ X\coloneqq  \rep_{\proj}(K_r,1) \cap \bigcap^n_{i=1} \cO(f^+_{i,r},\not \twoheadrightarrow) \cap \cO(f^-_{n+1,r},\not \hookrightarrow)\]
is a non-empty and open subset of $\cV(K_r;V_1,V_2)$. Let $M \in X$ and $\fv \in \Gr_2(A_r)$, then $M \in \rep_{\proj}(K_r,1)$, $M|_{\fv}$ is preprojective (see \cite[(4.3)]{Wor13b}), and \cref{5.5.2}(3) gives us $0 = \Hom_{K_r}(M,P_n^+(\fv))$. We conclude with \cref{4.3.4}(4) and $n \geq 1$ that $P_0(2)$ is not a direct summand of $M|_{\fv}$. Hence, $P_0(r)$ is not a direct summand of $M$.

Let $N$ be a direct summand of $M$, then $N_1 \neq 0$. The assumption $N \in \rep_{\proj}(K_r,2)$ yields with \cref{2.3.1} that $\Delta_N(2) \geq \min \{ 2(r-2),\dim_{\KK} N_1(r-2)\} \geq r-2 \neq 0$ and $N|_{\fv} = \Delta_{N}(2) P_0(2) \oplus (\dim_{\KK} N_1) P_1(2)$ is a direct summand of $M|_{\fv}$, a contradiction. We conclude with \cref{2.3.6}(2) that $M$ can not have a preprojective representation as a direct summand. Since $\rep_{\proj}(K_r,1)$ does not contain non-zero preinjective representations (see \cref{2.3.6}), every indecomposable direct summand of $M$ is therefore regular. Hence, we may apply \cref{2.1.1} and conclude $q_r(\dimu(V_1,V_2)) \leq 0$. Now \cref{2.2.2} gives us a non-empty open subset $\cO' \subseteq \cV(K_r;V_1,V_2)$ consisting of bricks that are not homogeneous. We set
\[ \cO := \cO' \cap X \neq \emptyset\]
and note that $(\mathrm{I})$ implies $\dim_\bmk V_1 \geq n(n+1)(r-2) + n \Delta_{(V_1,V_2)} \geq n(n+1)(r-2) + n(r-1)$ and $\Delta_{(V_1,V_2)} \geq (2n+2)(n+1)(r-2)$ established in $(\ast)$, gives
\begin{align*}
    \dim_\bmk V_2 &\geq \dim_\bmk V_1 + (2n+2)(n+1)(r-2) \\
    &\geq n(n+1)(r-2) + n(r-1) + (2n+2)(n+1)(r-2) \\
    &= (3n+2)(n+1)(r-2) + n(r-1) \\
    &\geq 10(r-2) + n(r-1) \geq 2(r-1)+ 1 + n(r-1)\\
    &= (n+2)(r-1)+1.
\end{align*} 
Hence, we may apply \cref{5.2.1}(2) and \cref{5.5.2}(3) to conclude that for every element in $\cO$ the condition of \cref{4.3.5}(3) is satisfied. Finally, we recall that $a_{n}(2) = n$ for all $n \in \NN_0$.
\end{proof}

\bigskip
\noindent The following result shows a possible application of \cref{5.2.3}.

\bigskip

\begin{corollary}\label{5.2.4}
Let $n \in \NN$ and $r \geq 3$, fix $s \geq 2(n+1)^2(r-2)$ and set $\ell \coloneqq s -2(n+1)^2(r-2) \in \NN_0$.
\begin{enumerate}
    \item For each 
\[c \in [n((n+1)(r-2)+s),n((n+1)(r-2)+s)+\ell],\]
and every pair of vector space $(V_1,V_2)$ with dimension vector $\dimu(V_1,V_2) = (c,s+c)$, there is a non-empty open subset $\cO \subseteq \rep_{\proj}(K_r,1) \cap \cV(K_r;V_1,V_2)$ such that every representation in $M \in \cO$ is a non-homogeneous brick, uniform and
    \[ M|_{K_2} \cong [(n+1)\dim_\bmk V_2 - (n+2) \dim_\bmk V_1] P_{n}(2) \oplus [(n+1)\dim_\bmk V_1 - n \dim_\bmk V_2] P_{n+1}(2).\]
    \item For $s = 2(n+1)^2(r-2)$ and $c = n((n+1)(r-2)+s)$ we obtain
\[ M|_{K_2} \cong [(n+2)(n+1)(r-2)] P_n(2) \oplus [n(n+1)(r-2)] P_{n+1}(2).\]
In particular, this implies $\dim_\bmk M_1 = n(n+1)(2n+3)(r-2)$ and $\Delta_M = (n+1)(2n+2)(r-2)$.
\end{enumerate}
\end{corollary}
\begin{proof}
\begin{enumerate}
    \item 
   We compute 
\begin{align*}
    (n+1)\dim_{\KK} V_1 - n\dim_{\KK} V_2 &=  c - ns \geq n(n+1)(r-2), \text{and} \\
    (n+1)\dim_{\KK} V_2 - (n+2)\dim_{\KK} V_1 &= (n+1)s - c\\ &\geq (n+1)s-n((n+1)(r-2)+s)-\ell\\
    &= s-\ell - n(n+1)(r-2)\\
    &= 2(n+1)^2(r-2) - n(n+1)(r-2) \\
    &=(n+1)(n+2)(r-2).
\end{align*}
Hence, the conditions of \cref{5.2.3} are satisfied.
\item Follows from the proof of (1).
\end{enumerate}
\end{proof}

\bigskip
\noindent We also record the following Lemma.
\bigskip

\begin{corollary} \label{5.2.5}
Let $n \in \NN$, $r \geq 3$ and $(V_1,V_2)$ be a pair of $\KK$-vector spaces.
\begin{enumerate} 
    \item Assume that $n \dim_\bmk V_2 - (n+1)\dim_\bmk V_1 \geq n(n+1)(r-2)$ and $\dim_\bmk V_2 \geq (n+1)(r-1)+1$. A general representation $M$\footnote{By \qq{general representation} we mean that there exists a non-empty open subset $O$ of $\cV(K_r;V_1,V_2)$ such that every representation in $O$ has the described property.} of $\cV(K_r;V_1,V_2)$ is in $\rep_{\proj}(K_r,1)$ and for each $\fv \in \Gr_2(A_r)$ we have     
    \[ M|_{\fv} = M|_{\fv,\pproj} \in \add(\{ P_i(2) \mid 0 \leq i \leq n\}).\]
     \item Assume that $(n+1)\dim_\bmk V_1 - n \dim_\bmk V_2 \geq n(n+1)(r-2)$ and $\Delta_{(V_1,V_2)} \geq r - 1$. A general representation $M$ of $\cV(K_r;V_1,V_2)$ is in $\rep_{\proj}(K_r,1)$ and for each $\fv \in \Gr_2(A_r)$ we have   
    \[ M|_{\fv} = M|_{\fv,\pproj} \in \add(\{ P_i(2) \mid i \geq n\}).\]
\end{enumerate}
\end{corollary}
\begin{proof}
\begin{enumerate}
      \item By assumption, we have $\Delta_{(V_1,V_2)} = \dim_\bmk V_2 - \dim_\bmk V_1 \geq (n+1)(r-2) + \frac{\dim_\bmk V_1}{n} \geq  (n+1)(r-2) \geq 2(r-2) \geq r - 1$. We combine \cref{5.2.1}, \cref{2.3.1} and \cref{4.3.4} and obtain a non-empty open subset $\cO \subseteq \cV(K_r;V_1,V_2) \cap \rep_{\proj}(K_r,1)$ such that for all $M \in \cO$ we have $M|_{\fv} = M|_{\fv,\pproj} \in \add(\{ P_i(2) \mid 0 \leq i \leq n\})$.
     \item This follows from \cref{5.5.2}(3), \cref{2.3.1} and \cref{4.3.4}.
\end{enumerate}
\end{proof}

\bigskip

\section{Consequences for Steiner bundles} \label{S:6}

In this section, we transfer our findings from the previous section to $\StVect(\PP(A_r))$ and show that several \qq{wild} phenomena of $\Vect(\PP^{r-1})$ already occur in the at first glance, innocent-looking category of Steiner bundles on $\PP(A_r) \cong \PP^{r-1}$ for $r \geq 3$.

\subsection{The type of a uniform Steiner bundle}

Let $\cG \in \Vect(\PP(A_r))$ be a Steiner bundle.  
Recall from \cref{S:1.1} that there exists a uniquely determined sequence $(b_i(\cG))_{i \in \NN_0}$ of natural numbers such that 
\[
O_{\cG} := \bigl\{ \fv \in \Gr_2(A_r) \ \big| \ 
\cG|_{\fv} \cong \bigoplus_{i \in \NN_0} b_i(\cG)\,\cO_{\PP(\fv)}(i) \bigr\}
\]
is a non-empty open subset of $\Gr_2(A_r)$.  
Moreover, $O_{\cG} = \Gr_2(A_r)$ holds if and only if $\cG$ is uniform.

\bigskip

\begin{Definition}\label{6.1.1}
Let $\cF \in \StVect(\PP(A_r))$ be a Steiner bundle. We call $\cF$ a \textit{$k$-type uniform bundle} or \textit{of $k$-type}, provided $\cF$ is uniform and $k = \max \{ i \in \NN_0\mid b_i(\cF) \neq 0\}$. 
\end{Definition}

\bigskip

Uniform Steiner bundles of $k$-type exist for all $k \in \NN_0$ (see \cite[(2.8),(5.2)]{MMR21}). According to \cref{1.5.5}, Steiner bundles of $0$-type are just direct sums of $\cO_{\PP(A_r)}$, 
while the uniform Steiner bundles of $1$-type are exactly those of the form 
$\widetilde{\Theta}(M)$ with $M \in \rep_{\proj}(K_r,2)$ non-semisimple. We begin this section by showing how our result can be applied to construct non-homogeneous Steiner bundles that are $1$-uniform.  

\bigskip

\subsection{1-type uniform Steiner bundles}\label{S:6.2}
The following result shows that uniform Steiner bundles that are non-homogeneous can be obtained by elementary operations, thereby slightly sharpening \cite[(5.3.4)]{BF24}.

\bigskip

\begin{Lemma}\label{6.2.1}
Let $2 \leq d < r$ and $X \in \rep_{\proj}(K_d,1)$ be indecomposable and not simple. Then the following statements hold.
\begin{enumerate}
    \item The representation $X_{d,r}^- = (\sigma_{K_r}^{-1} \circ \inf)(X) \in \rep(K_r)$ is indecomposable, quasi-simple, not homogeneous and $X_{d,r}^- \in \rep_{\proj}(K_r,d-1) \setminus \rep_{\proj}(K_r,d)$. 
    \item If $d \geq 3$ and $N \in (X_{d,r}^- \to)$, then the Steiner bundle $\widetilde{\Theta}(N)$ is uniform of $1$-type but not homogeneous.
\end{enumerate}
\end{Lemma}
\begin{proof}
\begin{enumerate}
    \item By \cref{4.1.4}, the representation $X_{d,r}^-$ is indecomposable, regular and \cite[(3.2.2)]{Bis18} implies that $X_{d,r}^{-}$ is quasi-simple. 
    Since $\psi_{\inf(X)} \neq 0$ and $\psi_{\inf(X)}(\gamma_r \otimes -) = 0$, we conclude with \cref{1.5.2} that $X_{d,r}^-$ can not be homogeneous. 

Given a representation $N \in \rep(K_s)$ and $a \in A_s$, we denote by $a_N$ the $\KK$-linear map 
\[ a_N \colon N_1 \lra N_2 \ ; \ n \mapsto \psi_N(a \otimes n).\]
Let $\fv \in \Gr_{r-(d-1)}(A_r)$. We have $\{0\} \neq \fv \cap A_d$ and find $0 \neq a\in \fv \cap A_d$.
Since $X \in \rep_{\proj}(K_d,1)$, we have $\{0\}=\ker a_X = \ker a_{\inf(X)} $ and conclude  $\{0\} = \bigcap_{b \in \fv} \ker b_{\inf(X)}$. On the other hand, $A_{d}^\perp \coloneqq \bigoplus^r_{d+1} \KK \gamma_i \in \Gr_{r-d}(A_r)$ satisfies 
\[ \bigcap_{b \in A_{d}^\perp} \ker b_{\inf(X)} = X_1\neq \{0\}.\]
By \cite[(2.6)]{Bis25}, this is equivalent to $X_{d,r}^- \in \rep_{\proj}(K_r,d-1) \setminus \rep_{\proj}(K_r,d)$. 
\item By (1), the representation $X_{d,r}^-$ is not homogeneous and $d \geq 3$ implies  $X_{d,r}^- \in \rep_{\proj}(K_r,d-1) \subseteq \rep_{\proj}(K_r,2)$. Now we apply \cref{1.5.5} and \cref{2.2.1}.
\end{enumerate}
\end{proof}

\bigskip

\subsection{1-type uniform Steiner bundles of minimal rank.}
Let $n,c \in \NN_{\geq 2}$. In \cite{MMR21} the authors study uniform Steiner bundles that are not homogeneous over an algebraically closed field of characteristic $0$ and show that each $1$-type uniform Steiner bundle $\cE \in \StVect(\PP^n)$ 
\[ 0 \to \cO_{\PP^n}(-1)^{c} \to \cO_{\PP^n}^{s+c} \to \cE \to 0
\]
satisfies $c+2(n-1) \leq s = \rk(\cE)$ (note that this result follows from \cref{1.3.3}, \cref{1.5.5} and \cref{2.3.1}). Then they show that this lower bound is sharp by constructing for all $n \geq 2$ and all $c \geq 2$ a family of Steiner bundles $U_{n,c} \in \StVect(\PP^n)$ with and $c_1(U_{n,c}) = c$, $\rk(U_{n,c}) = c + 2(n-1)$. Moreover, they show that each $U_{n,c}$ has the following properties (see \cite[(4.6)]{MMR21}):
\begin{itemize}
\item[(i)] $U_{n,c}$ is uniform of $1$-type.
\item[(ii)] $U_{n,c}|_{\PP^2} \cong U_{2,c} \oplus \cO^{2n-4}_{\PP^2}$ for a suitable projective plane $\PP^2 \subseteq \PP^n$.	
\end{itemize} 
Moreover, they prove that the constructed $U_{2,c} \in \StVect(\PP^2)$ are non-homogeneous for $c \geq 4$, and, combining this with $(i),(ii)$, conclude that each $U_{n,c}$ is a uniform but non-homogeneous Steiner bundle of $1$-type. We provide an alternative proof of the last statement which avoids relying on the structure of the bundles $U_{2,c}$ and instead uses the theorem of Van de Ven, thereby extending the result to all uniform $1$-type bundles on $\PP^2$ of minimal rank.

\bigskip

\begin{Theorem}\label{6.3.1} Let $\Char(\KK) = 0$, $c \geq 4$ and $\cF \in \StVect(\PP(A_3))$ be a uniform Steiner bundle of $1$-type with $c_1(\cF) = c$, $\rk(\cF) = c + 2$. Then $\cF$ is not homogeneous.
\end{Theorem} 
\begin{proof}
    We write $\cF = \widetilde{\Theta}(M)$ with $M \in \rep_{\proj}(K_3,1)$. According to \cref{1.5.5}, we have $M \in \rep_{\proj}(K_3,2)$.  In view of \cref{1.5.3}, it suffices to show that $M$ is not homogeneous. We have $\dimu M = (c,2c+2)$ and $\Delta_M(2) = 2c+2 - 2c = 2 = 2(3-2)$. Since $\dim_{\KK} M_1 = c \geq 4 > 3 = 2+1$, we conclude with \cite[(2.3.3)]{BF24} that $M$ is not projective. Another application of \cite[(2.3.3)]{BF24} implies that $M$ is a brick. By \cite[(2.6)]{Bis25}, we have $\sigma_{K_3}(M) \in \rep_{\proj}(K_3,1)$ with $\dimu \sigma_{K_r}(M) = (c-2,c)$ and \cref{4.2.2} implies that $\sigma_{K_r}(M)$ is homogeneous.  Hence, the Steiner bundle $\widetilde{\Theta}(\sigma_{K_r}(M))$ is a simple and homogeneous Steiner bundle of rank $2$ with first Chern class $c-2$. Now the Theorem of Van de Ven \cite[(2.2.2)]{OSS80} implies that $\widetilde{\Theta}(\sigma_{K_r}(M))$ is isomorphic to a twist of the tangent bundle $\cT_{\PP(A_3)}$, i.e., $\widetilde{\Theta}(\sigma_{K_r}(M)) \cong \cT_{\PP(A_3)}(a)$ for some $a \in \ZZ$.
    One checks that $\cT_{\PP(A_3)}(a) \in \StVect(\PP(A_3))$ implies $a = -1$. Hence, $\widetilde{\Theta}(\sigma_{K_r}(M)) \cong \cT_{\PP(A_3)}(-1)$ and $c_1(\cT_{\PP(A_3)}(-1)) = c-2$. This is a contradiction since the Euler sequence
    $0 \lra \cO_{\PP(A_3)}(-1) \lra \cO_{\PP(A_3)}^{3} \lra \cT_{\PP(A_3)}(-1) \lra 0$
    gives us $c_1(\cT_{\PP(A_r)}(-1)) = 1 \neq c-2$.
\end{proof}

\bigskip

\begin{Remark}\label{6.3.2}
Let $\Char(\KK) = p > 0$. Together with Rolf Farnsteiner \cite{BF25} we have proven that every simple and homogeneous Steiner bundle of rank $r-1$ on $\PP(A_r)$ is isomorphic to $\cT_{\PP(A_r)}(-1)$. This shows that \cref{6.3.1} also holds for positive characteristic.
\end{Remark}

\bigskip

\subsection{k-type uniform Steiner bundles} 

In this section, we study the $k$-type of uniform but non-homogeneous Steiner bundles for $r \geq 3$ and $k \geq 2$.  
The first systematic approach to this problem was carried out by Marchesi and Miró-Roig \cite{MMR21}, who constructed such bundles $\cF$ of $k$-type for every $k \geq 2$, all satisfying $\supp(\cF) = \{0,1,\ldots,k\}$. 
In contrast, we prove the existence of non-homogeneous indecomposable $k$-uniform Steiner bundles with disconnected splitting type, that is, with $\supp(\cF)$ not forming an interval in $\NN_0$. Moreover, we show that the gaps in the splitting type can be made arbitrarily large.

\bigskip

\begin{Theorem}\label{6.4.1}
Let $n \in \NN_0$.
\begin{enumerate}
\item For $n \in \NN$, fix $s \geq 2(n+1)^2(r-2)$ and set $\ell \coloneqq s -2(n+1)^2(r-2) \in \NN_0$ as well as \[c \in [n((n+1)(r-2)+s),n((n+1)(r-2)+s)+\ell].\]
 The general Steiner bundle $\cF \in \Vect(\PP(A_r))$ with first Chern class $c$ and of rank $s$ is simple, uniform of $k$-type with support $\{n,n+1\}$, and not homogeneous.
\item There exists a simple, uniform, but non-homogeneous Steiner bundle of $(n+1)$-type with support $\supp(\cF) = \{n,n+1\}$.
  \item For $n \geq 2$, there exists an indecomposable, uniform, but non-homogeneous Steiner bundle $\cF$ of $(n+1)$-type with support $\supp(\cF) = \{0,1,n,n+1\}$.
 In particular, there is no upper bound for the size of the gaps.
\end{enumerate}
\end{Theorem}
\begin{proof}
\begin{enumerate}
  \item This is a direct consequence of \cref{5.2.4}, \cref{1.3.3}, \cref{1.5.3} and \cref{1.5.4}.
  \item In view of (1), it suffices to treat the case $n = 0$, which follows from \cref{2.3.1} and \cref{2.2.6}.
  \item We consider $M$ as in \cref{5.2.4}(1). Then $M \in \rep_{\proj}(K_r,1) \setminus \rep_{\proj}(K_r,2)$ is regular (see \cref{1.5.5} and \cref{2.3.6}) and quasi-simple since $M$ is a brick (see \cite[(9.2),(9.4)]{Ker94}). We choose $\ell \in \NN_{\geq 2}$ and set $E \coloneqq M_{[\ell]}$ as well as  $X \coloneqq \tau_{K_r}^{-1}(M_{[\ell-1]})$. The situation may be illustrated as follows:     
 
%%%%%%%%%%%%%%%%%%%%%%%%%%%%%%%%%%%%%%%% ZA_\infty Anfang %%%%%%%%%%%%%%%%%%%%%%%%%%
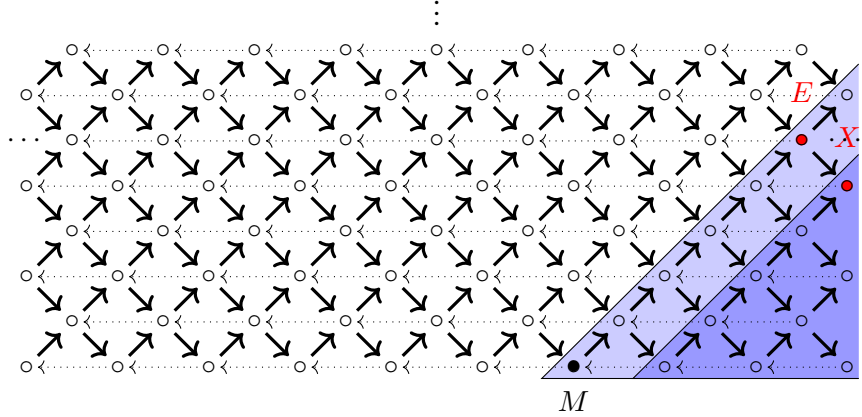
\begin{center}
\begin{tikzpicture}[very thick, scale=1]
                    [every node/.style={fill, circle, inner sep = 1pt}]

\def \n {9} % #Knoten Reihe  - 1
\def \m {3} % #Knoten Spalte - 1
\def \translation {1} % 1 Für Translation

\def \ab {0.15} % Abstand Pfeil und Knoten
\def \Pab {0.6} % Halber Abstand Horizontal

\def \rcone {1} % 1 für rechten Kegel
\def \rdist {2} % Anzahl der quasi-einfachen die eingeschlossen werden - 1
\def \rcolor {blue} %  white für keine Farbe

\def \rrcone {1} %1 für einen zweiten rechten Kegel links von rcone
\def \rrdist {3} % Anzahl der quasi-einfachen die eingeschlossen werden - 1
\node[color=black]  at (\n*\Pab*2-\Pab*2*4+1.2,-\ab+0.16) {$\bullet$};
\node[color=black]  at (\n*\Pab*2-\Pab*2*4+4.2,-\ab+3.15) {$\textcolor{red}{\bullet}$};
\node[color=black]  at (\n*\Pab*2-\Pab*2*4+4.8,-\ab+2.55) {$\textcolor{red}{\bullet}$};

\node[color=black]  at (\n*\Pab*2-\Pab*2*4+4.2,-\ab-0.3+4.1) {\textcolor{red}{$E$}};
\node[color=black]  at (\n*\Pab*2-\Pab*2*4+4.8,-\ab-0.3+3.5) {\textcolor{red}{$X$}};

\node[color=black]  at (\n*\Pab*2-\Pab*2*4+1.2,-\ab-0.3) {$M$};

\foreach \a in {0,...,\n}{
\foreach \b in {0,...,\m}{
  
   \ifthenelse{\a = \n \and \b < \m}{
   \node[color=black] ({\a,\b,5})at ({\a*2*\Pab},{\b*2*\Pab}) {$\circ$};
     }
     {
      \ifthenelse{\b = \m \and \a < \n}{
      \node[color=black] ({\a,\b}) at ({\a*2*\Pab+\Pab},{\b*2*\Pab+\Pab}) {$\circ$};
      \node[color=black] ({\a,\b,5})at ({\a*2*\Pab},{\b*2*\Pab}) {$\circ$};
      }
      {
    
     \ifthenelse{\b = \m \and \a = \n}
     {\node[color=black] ({\a,\b,5})at ({\a*2*\Pab},{\b*2*\Pab}) {$\circ$};}
    {\node[color=black] ({\a,\b}) at ({\a*2*\Pab+\Pab},{\b*2*\Pab+\Pab}) {$\circ$};
    \node[color=black] ({\a,\b,5})at ({\a*2*\Pab},{\b*2*\Pab}) {$\circ$};

      }
      }
      }
    }
    }

\foreach \s in {0,...,\n}{
\foreach \t in {0,...,\m}
{  
 \ifthenelse{\t = \m \and \s < \n}{
    \draw[->] (\s*2*\Pab+\ab,\t*2*\Pab+\ab) to (\s*2*\Pab+\Pab-\ab,\t*2*\Pab+\Pab-\ab); 
    \draw[->] (\s*2*\Pab+\Pab+\ab,\t*2*\Pab+\Pab-\ab) to (\s*2*\Pab+2*\Pab-\ab,\t*2*\Pab+\ab); 

  }{
  
  \ifthenelse{\s = \n \and \t < \m}{
  
  }
  {
  \ifthenelse{\s = \n \and \t = \m}{
   
  }{
   \draw[->] (\s*2*\Pab+\ab,\t*2*\Pab+\ab) to (\s*2*\Pab+\Pab-\ab,\t*2*\Pab+\Pab-\ab); 
   \draw[->] (\s*2*\Pab+\Pab+\ab,\t*2*\Pab+\Pab+\ab) to (\s*2*\Pab+2*\Pab-\ab,\t*2*\Pab+2*\Pab-\ab);
   \draw[->] (\s*2*\Pab+\ab,\t*2*\Pab+2*\Pab-\ab) to (\s*2*\Pab+\Pab-\ab,\t*2*\Pab+\Pab+\ab); 
   \draw[->] (\s*2*\Pab+\Pab+\ab,\t*2*\Pab+\Pab-\ab) to (\s*2*\Pab+2*\Pab-\ab,\t*2*\Pab+\ab);    
   }
   }
  
    }
    }
    }

\ifthenelse{\isodd{\m}}
%% IF
 { 
  \node[color=black] (Dots1) at (0,\m*\Pab+2*\Pab) {$\cdots$};
  \node[color=black] (Dots2) at (\n*2*\Pab,\m*\Pab+2*\Pab) {$\cdots$};
   \ifthenelse{\isodd{\n}}{
  \node[color=black] (Dots3) at (0.5*\n*2*\Pab,2*\m*\Pab+2*\Pab) {$\vdots$};}
  {\node[color=black] (Dots3) at (0.5*\n*2*\Pab,2*\m*\Pab+\Pab) {$\vdots$};} 
  }
%% Else
  {
  \node[color=black] (Dots1) at (0,\m*\Pab+\Pab) {$\cdots$};
  \node[color=black] (Dots2) at (\n*2*\Pab,\m*\Pab+\Pab) {$\cdots$};
  \ifthenelse{\isodd{\n}}{
  \node[color=black] (Dots3) at (0.5*\n*2*\Pab,2*\m*\Pab+2*\Pab) {$\vdots$};}
  {\node[color=black] (Dots3) at (0.5*\n*2*\Pab,2*\m*\Pab+\Pab) {$\vdots$};}
  }
 
\ifthenelse{\translation = 1}{
   \foreach \s in {0,...,\n}{
   \foreach \t in {0,...,\m}{ 
   \ifthenelse{\s = 0}{}{
      \ifthenelse{\s = \n}{\draw[->,dotted,thin] (\s*2*\Pab-\ab,\t*2*\Pab) to (\s*2*\Pab-2*\Pab+\ab,\t*2*\Pab); }{
   \draw[->,dotted,thin] (\s*2*\Pab-\ab,\t*2*\Pab) to (\s*2*\Pab-2*\Pab+\ab,\t*2*\Pab); 
   \draw[->,dotted,thin] (\s*2*\Pab-\ab+\Pab,\t*2*\Pab+\Pab) to (\s*2*\Pab-2*\Pab+\Pab+\ab,\t*2*\Pab+\Pab); 
   }
   }}
}}
{}  %ELSE

\begin{scope}[on background layer]
 
\ifthenelse{\rrcone = 1}{
        \draw[fill = \rcolor!20] (\n*\Pab*2+\ab,-\ab) node[anchor=north]{}
  -- (\n*\Pab*2-\rrdist*\Pab*2-0.7*\Pab,-\ab) node[anchor=north]{}
  -- (\n*\Pab*2+\ab,\rrdist*\Pab*2+0.7*\Pab) node[anchor=south]{};
    }
  {}

\ifthenelse{\rcone = 1}{
        \draw[fill= \rcolor!40](\n*\Pab*2+\ab,-\ab) node[anchor=north]{}
  -- (\n*\Pab*2-\rdist*\Pab*2-0.7*\Pab,-\ab) node[anchor=north]{}
  -- (\n*\Pab*2+\ab,\rdist*\Pab*2+0.7*\Pab) node[anchor=south]{};
    }
  {}   
   
\end{scope}
\end{tikzpicture}
\end{center}
General theory yields a short exact sequence
     \[ 0 \lra M \lra E \lra X  \lra 0.\]
By \cref{2.3.6}, we have $X \in \rep_{\proj}(K_r,r-1) \subseteq  \rep_{\proj}(K_r,2)$ and \cref{1.5.5} implies that $\widetilde{\Theta}(X)$ is of $1$-type with $\supp(\widetilde{\Theta}(X)) = \{0,1\}$. 
Let $\alpha \in \Inj_{\KK}(A_2,A_r)$, then $\alpha^\ast(X) \in \rep(K_2)$ is projective and the exact sequence 
\[ 0 \lra \alpha^\ast(M) \lra \alpha^\ast(E) \lra \alpha^\ast(X) \lra 0\]
 splits. We define $\cF \coloneqq \widetilde{\Theta}(E)$. Twofold application of \cref{1.5.3} yields
     \[ \hat{\alpha}^\ast(\cF) \cong \widetilde{\Theta}(\alpha^\ast(E)) \cong \widetilde{\Theta}(\alpha^\ast(M)) \oplus \widetilde{\Theta}(\alpha^\ast(X)) \cong \hat{\alpha}^\ast(\widetilde{\Theta}(M)) \oplus \hat{\alpha}^\ast(\widetilde{\Theta}(X)).\]
     Since $\widetilde{\Theta}(M)$ and $\widetilde{\Theta}(X)$ are uniform, we conclude that $\cF$ is uniform and 
     \[\supp(\cF) = \supp(\widetilde{\Theta}(M)) \cup \supp(\widetilde{\Theta}(X)) = \{0,1,n,n+1\}.\]
\end{enumerate}
\end{proof}

\bigskip

\noindent The following results may be of independent interest for future investigations.

\bigskip

\begin{proposition}\label{6.4.2}
Let $c,s,n \in \NN$. The following statements hold.
\begin{enumerate}
    \item If $c \leq n[s - (n+1)(r-2)]$ and $s + c \geq (n+1)(r-1) + 1$, then a general\footnote{See \cref{5.2.5} for the definition.} Steiner bundle $\cF$ of rank $s$ and with first Chern class $c$ satisfies that for each $\fv \in \Gr_2(A_r)$ every direct summand $\cO_{\PP(\fv)}(i)$ of $\cF|_{\fv}$ satisfies $i \leq n$.
    \item If $c \geq n[(n+1)(r-2)+s]$ and $s \geq r-1$, then a general Steiner bundle $\cF$ of rank $s$ and with first Chern class $c$ satisfies that for each $\fv \in \Gr_2(A_r)$ every direct summand $\cO_{\PP(\fv)}(i)$ of $\cF|_{\fv}$ satisfies $i \geq n$.
    \item Let $1 \leq d < r$. If $(d-1)c - s < d(r-d)$, $s \geq r - 1$ and $\cF$ is a Steiner bundle of rank $s$ and with first Chern class $c$, then there exists $\alpha \in \Inj_{\KK}(A_d,A_r)$ such that $\hat{\alpha}^\ast(\cF)$ has $\cO_{\PP(A_d)}$ as a direct summand. 
    \end{enumerate}
\end{proposition}
\begin{proof}
We consider a pair of vector spaces such that $\dim_\bmk V_1 = c$ and $\dim_\bmk V_2 = s + c$. 
\begin{enumerate}
\item We have
\[ \dim_\bmk V_2 = s+c \geq (n+1)(r-1) + 1,\]
and 
\begin{align*}
     n \dim_\bmk V_2 - (n+1) \dim_\bmk V_1 &= n(s+c) - (n+1)c  \\
     &= ns - c \geq ns-  n[s-(n+1)(r-2)] \\
     &= n(n+1)(r-2).
\end{align*}
By \cref{5.2.5}(1), we find a non-empty open subset $\cO \subseteq \rep_{\proj}(K_r,1)$ of $\cV(K_r;V_1,V_2)$ such that for each $M \in \cO$ and every $\fv \in \Gr_2(A_r)$, we have $M|_{\fv} \in \add(P_0(2),\ldots,P_n(2))$. We conclude with \cref{1.5.4} that $\widetilde{\Theta}(M)|_{\fv} \in \add(\cO_{\PP(\fv)},\ldots,\cO_{\PP(\fv)}(n))$.
\item We have  
\begin{align*}
     (n+1) \dim_{\KK} V_1 - n \dim_\bmk V_2 &= (n+1)c - n(c+s) \\
     &= c - ns \geq n[(n+1)(r-2)+s] - ns \\
     &= n(n+1)(r-2)
\end{align*}
and $\dim_\bmk V_2 - \dim_\bmk V_1 = s \geq r - 1$. The statement follows as in (1) using \cref{5.2.5}(2).
\item Let $M \in \rep_{\proj}(K_r,1)$ with $\dimu M = (c,s+c)\in \NN^2$ and $\widetilde{\Theta}(M) \cong \cF$. Then $d \dim_\bmk M_1 - \dim_\bmk M_2 = (d-1)c - s < d(r-d)$. We conclude with \cite[(2.12)]{Bis25} that $D_{K_r}(M) \not\in \rep_{\esp}(K_r,d)$. Dualizing implies that there is $\fv \in \Gr_d(A_r)$ such that the map
\[ \psi_{M}|_{\fv \otimes_{\KK} M_1} \colon \fv \otimes_{\KK} M_1\lra M_2\] is not surjective.
Let $\alpha\in \Inj_{\KK}(A_d,A_r)$ be such that $\im \alpha = \fv$, then $\alpha^\ast(M)$ has $P_0(d)$ as a direct summand. 
Since $\widetilde{\Theta}$ commutes with direct sums, $\cO_{\PP(A_d)}$ is a direct summand of $\hat{\alpha}^\ast(\cF)$.
\end{enumerate}
\end{proof} 

\bigskip

\subsection{Jumping lines and almost-uniform Steiner bundles}
Uniform vector bundles $\cF \in \Vect(\PP^n)$ are characterized by the property that their set of jumping lines $\cJ_{\cF}$ is empty (see \cref{S:1.1}). A natural generalization of this notion, due to Ellia \cite{Ell17}, is the following:

\bigskip

\begin{Definition}\label{6.5.1}
A vector bundle $\cF \in \Vect(\PP(A_r))$ is called \textit{almost-uniform}, provided $\cJ_{\cF}$ is non-empty and finite. 
\end{Definition}

\bigskip

Using the functor $\widetilde{\Theta} \colon \rep_{\proj}(K_r,1) \lra\StVect(\PP(A_r))$ we are able to give a handy criterion to check for almost-uniform Steiner bundles.
In general, it is difficult to compute the generic decomposition of a given representation. However, we have the following useful criterion.

\bigskip

\begin{proposition}\label{6.5.2}
Let $0 \neq M \in \rep_{\proj}(K_r,1)$.
\begin{enumerate}
    \item If there is $\fv \in \Gr_2(A_r)$ such that $0 = \Hom_{K_r}(P_1^-(\fv),M)$, then
   \[ M_{\gen} = \Delta_{M}(2) P_0(2) \oplus (\dim_{\KK} M_1)P_1(2)\]
   and
\[\cJ_{\widetilde{\Theta}(M)} = \cJ_M  =\{ \fu \in \Gr_2(A_r) \mid \Hom_{K_r}(P_1^{-}(\fu),M) \neq 0\}.\]
\item If $\{ \fu \in \Gr_2(A_r) \mid \Hom_{K_r}(P_1^-(\fu),M) \neq 0\}$ is finite and non-empty, then the Steiner bundle $\widetilde{\Theta}(M)$ is almost-uniform and 
\[\cJ_{\widetilde{\Theta}(M)} =\{ \fu \in \Gr_2(A_r) \mid \Hom_{K_r}(P_1^{-}(\fu),M) \neq 0\}.\]
\end{enumerate}
\end{proposition}
\begin{proof}
\begin{enumerate}
    \item Since $\Hom_{K_r}(P_1^-(\fv),M) = 0$, we conclude with \cref{4.4.3} and \cite[(2.5.5)]{BF24} that 
\[ M_{\gen} = \Delta_M(2) P_0(2) \oplus (\dim_{\KK} M_1)P_1(2)\]
and $O_M = \{ \fu \in \Gr_2(A_r) \mid M|_{\fu} = \Delta_M(2) P_0(2) \oplus (\dim_{\KK} M_1)P_1(2)\}$. Another application of \cref{4.4.3} gives us $O_M = \{ \fu \in \Gr_2(A_r) \mid \Hom_{K_r}(P_1^-(\fu),M) = 0\}$. Hence, \[ \cJ_{\TilTheta(M)} = \cJ_M = \{ \fu \in \Gr_2(A_r) \mid \Hom_{K_r}(P_1^-(\fu),M) \neq 0\}.\]
\item Since $\{ \fu \in \Gr_2(A_r) \mid \Hom_{K_r}(P_1^-(\fu),M) \neq 0\} \neq \emptyset$ is finite, there is $\fv \in \Gr_2(A_r)$ such that $\Hom_{K_r}(P_1^-(\fv),M) = 0$. Now we apply (1).
\end{enumerate}
\end{proof}

\bigskip

A main result of \cite{Ell17} is that for $n \in \NN$ there exists an almost-uniform vector bundle on $\PP^n$ of rank $2n-1$ with exactly one jumping line. The given examples \cite[(Proposition 16)]{Ell17} are precisely the Steiner bundles that one obtains from the representations $P_1^-(\fv)$, $\fv \in \Gr_2(A_r)$: 

\bigskip

\begin{proposition}\label{6.5.3}
Let $r \geq 3$ and $\fv \in \Gr_2(A_r)$. The Steiner bundle $\widetilde{\Theta}(P_1^-(\fv))$ satisfies
\begin{enumerate}
    \item $\rk(\widetilde{\Theta}(P_1^-(\fv))) = 2(r-1)-1$, and
    \item $\cJ_{\widetilde{\Theta}(P_1^-(\fv))} = \{ \fv\}$.
\end{enumerate}
\end{proposition}
\begin{proof}
Let $\fu \in \Gr_2(A_r)$. Recall from \cref{4.4.2} and  \cref{S:5} that $P_1^-(\fu)$ is an elementary representation in $\rep_{\proj}(K_r,1)$ and $\dimu P_1^-(\fu) = (2,2(r-1)+1)$. By \cref{1.3.3}, the Steiner bundle $\widetilde{\Theta}(P_1^-(\fv)) \in \StVect(\PP(A_r))$ has rank $2(r-1)-1$.
By \cref{4.2.6}, we have $P_1^-(\fu) \cong P_1^-(\fw)$ if and only if $\fu = \fw$.
Since $\dimu P_1^-(\fu) = \dimu P_1^-(\fw)$, we conclude with \cite[(1.4)]{KL96} that $\Hom_{K_r}(P_1^-(\fu),P_1^-(\fv)) \neq 0$ if and only if $\fu = \fv$. Now we apply \cref{6.5.2}.
\end{proof}

\bigskip

Taking direct sums of the above Steiner bundles shows that every finite subset $\emptyset \neq X \subseteq \Gr_2(A_r)$ can be realized as a set of jumping lines. However, this construction does not give any insight into the \qq{size} of the category consisting of Steiner bundles with $X$ as the set of jumping lines.

In the following we make use of a construction of Bongartz (see \cite{Bon81},\cite[(VI.2.4)]{ASS06} and \cite{Luk24}) and the process of simplification (see \cite{KL91}) due to Ringel, to show that the category of Steiner bundles with set of jumping lines $X$ corresponds to a wild subcategory in $\rep_{\proj}(K_r,1)$.

Throughout this section we assume that $r \geq 3$ and fix representations $X_1,\ldots,X_n,Y \in \rep(K_r)$ such that $s_i \coloneqq \dim_{\KK} \Ext^1_{K_r}(Y,X_i) \neq 0$ for  $i \in \{1,\ldots,n\}$. We stick to the notation introduced in \cite[\S III.3]{SY11}. 
We fix short exact sequences \[ (\mathbb{E}_{X_i,j} \colon 0 \lra X_i \stackrel{f_{X_i,j}}{\lra} {E_{X_i,j}} \stackrel{g_{X_i,j}}{\lra} Y\lra 0)_{j=1}^{s_{i}} \]
such that $([\mathbb{E}_{X_{i},j}])_{j=1}^{s_i}$ is a $\KK$-basis of the space of extensions $\ext^1_{K_r}(Y,X_{i}) \cong \Ext^1_{K_r}(Y,X_{i})$ for all $i \in \{1,\ldots,n\}$. 
We define 
\[f \coloneqq \bigoplus_{i=1}^n \bigoplus_{j=1}^{s_i} f_{X_i,j} \colon \bigoplus^n_{i=1} X_i^{s_i} \lra \bigoplus^n_{i=1}\bigoplus^{s_{i}}_{j=1}{E_{X_i,j}}  \ \text{and} \ g \coloneqq \bigoplus^n_{i=1} \bigoplus^{s_{i}}_{j=1}{E_{X_i,j}}  \lra \bigoplus_{i=1}^n Y^{s_i}\]
and obtain a short exact sequence
\[  0 \lra \bigoplus^n_{i=1} X_i^{s_i} \stackrel{f}{\lra} \bigoplus^n_{i=1}\bigoplus^{s_{i}}_{j=1}{E_{X_i,j}}  \stackrel{g}{\lra} \bigoplus_{i=1}^n Y^{s_i} \lra 0.\]
We consider the diagonal embedding $\nabla \colon Y \lra \bigoplus_{i=1}^n Y^{s_i}$ and obtain a pullback diagram
\[
\xymatrix{
   0 \ar[r]& \bigoplus^n_{i=1} X_i^{s_i} \ar@{=}[d]   \ar^{u}[r] & E    \ar^{v}[r] \ar^{w}[d]& Y \ar[r] \ar[d]^{\nabla} & 0 \\
    0 \ar[r] & \bigoplus^n_{i=1} X_i^{s_i} \ar^{f}[r] & \bigoplus^n_{i=1} \bigoplus^{s_{i}}_{j=1}{E_{X_i,j}}   \ar^{g}[r]  & \bigoplus_{i=1}^n Y^{s_i} \ar[r] & 0.
}
\]
The upper exact row
\[ \mathbb{E} \colon 0 \lra \bigoplus^n_{i=1} X_i^{s_i} \stackrel{u}{\lra} E \stackrel{v}{\lra} Y \lra 0\]
is called a $\{X_1,\ldots,X_n\}$-\textit{universal short exact sequence ending in $Y$.}

\bigskip

\begin{proposition}[Bongartz' construction]\label{6.5.4}
    Let $1 \leq \ell \leq n$. The connecting homomorphism
    \[\delta \colon \Hom_{K_r}(\bigoplus^n_{i=1} X_{i}^{s_i},X_\ell) \lra \Ext^1_{K_r}(Y,X_\ell)\]
    obtained by applying $\Hom_{K_r}(-,X_{\ell})$ to $\mathbb{E}$ is surjective.
\end{proposition}
\begin{proof}
See \cite[(5.2)]{Luk24} for proof in case of the dual construction or adapt the arguments given in \cite[(VIII.2.4)]{SY17} appropriately.
\end{proof}

\bigskip

\begin{corollary}\label{6.5.5}
In addition, assume that $X_1,\ldots,X_n,Y$ are pairwise $\Hom$-orthogonal.
\begin{enumerate}
    \item If $X_\ell$ is a brick, then $\Hom_{K_r}(E,X_\ell) = 0$.
    \item If $X_1,\ldots,X_n,Y$ are bricks, then $E$ is a brick that satisfies $\Hom_{K_r}(X_i,E) \neq 0 = \Hom_{K_r}(Y,E)$ and $\Hom_{K_r}(Z,E) = 0$ for every representation $Z$ that is $\Hom$-orthogonal to all $X_1,\ldots,X_n,Y$.
\end{enumerate}
\end{corollary}
\begin{proof}
\begin{enumerate}
\item We apply $\Hom_{K_r}(-,X_\ell)$ to $\mathbb{E}$ and obtain the exact sequence
\[ 0 \lra \Hom_{K_r}(Y,X_\ell) \stackrel{\Hom_{K_r}(v,X_{\ell})}{\lra} \Hom_{K_r}(E,X_{\ell}) \lra \Hom_{K_r}(\bigoplus^n_{i=1} X_i^{s_i},X_{\ell}) \stackrel{\delta}{\lra} \Ext^1_{K_r}(Y,X_\ell) \lra 0.\]
Note that $\Hom_{K_r}(\bigoplus^n_{i=1} X_i^{s_i},X_\ell) \cong \Hom_{K_r}(X_\ell^{s_{\ell}},X_\ell)$ is an $s_\ell$-dimensional space and therefore has the same dimension as $\Ext^1_{K_r}(Y,X_\ell)$. Since $\delta$ is surjective, $\delta$ is also injective and $\Hom_{K_r}(v,X_{\ell})$ is an isomorphism. The statement follows since $\Hom_{K_r}(Y,X_\ell) =0$.
\item We have a short exact sequence
\[ 0 \lra \Hom_{K_r}(Y,Y) \lra \Hom_{K_r}(E,Y) \lra \Hom_{K_r}(\bigoplus^n_{i=1} X_{i}^{s_i},Y) = 0\]
and conclude $\dim_{\KK}\Hom_{K_r}(E,Y) = 1$. The short exact sequence
\[ 0 \lra \Hom_{K_r}(E,\bigoplus^n_{i=1} X_{i}^{s_i}) \lra \Hom_{K_r}(E,E) \lra \Hom_{K_r}(E,Y)\]
in conjunction with (1) implies that $E$ is a brick. 

Clearly, we have $\Hom_{K_r}(X_i,E) \neq 0$ for all $i \in \{1,\ldots,n\}$. Let $Z$ be $\Hom$-orthogonal to every element in $\{X_1,\ldots,X_n,Y\}$, then the exact sequence
\[ 0 \lra \Hom_{K_r}(Z,\bigoplus^n_{i=1} X_i^{s_i}) \lra \Hom_{K_r}(Z,E) \lra \Hom_{K_r}(Z,Y) \]
gives us $\Hom_{K_r}(Z,E) = 0$. Finally, let $Z = Y$ and consider a morphism $h \in \Hom_{K_r}(Y,E)$. Then $h \circ v \in \End_{K_r}(E) = \KK \id_{E}$. Since $\ker v \cong \bigoplus^n_{i=1} X_i^{s_i} \neq \{0\}$, we conclude $0 = h \circ v$ and the surjectivity of $v$ gives $h = 0$.
\end{enumerate}
\end{proof}

\bigskip
\noindent We are now equipped with all the necessary tools to prove the main result of this section.

\bigskip

\begin{Theorem}\label{6.5.6}
Let $r \geq 3$, $1  < d < r$ and $\emptyset \neq X \subseteq \Gr_d(A_r)$ be a finite set. There exists a regular brick $E_X \in \rep(K_r)$ such that
\begin{enumerate}
    \item $\cV(K_r,d)_{E_X} = X$,
    \item $\cV(K_r,d-1)_{E_X} = \emptyset$, and
    \item $\dimu E_X = (|X| \cdot (d(r-d) -1)+1)\cdot (d,rd-1)$.
\end{enumerate}
\end{Theorem}
\begin{proof}
Let $\fu \neq \fw \in \Gr_d(A_r)$. By \cref{4.2.6} and \cref{4.4.2}, the representations $P_1^-(\fu)$ and $P_1^-(\fw)$ are non-isomorphic elementary representations that have the same dimension vector by construction. We apply \cite[(1.4)]{KL96} and conclude that $P_1^-(\fu)$ and $P_1^-(\fw)$ are $\Hom$-orthogonal bricks. We write $X = \{\fv_1,\ldots,\fv_n\}$ and fix $\fu \in \Gr_d(A_r) \setminus X$. We let $X_i \coloneqq P_1^-(\fv_i)$ for all $i \in \{1,\ldots,n\}$ and $Y \coloneqq P_1^-(\fu)$. Since $X_i$ and $Y$ are regular, we conclude with the Euler-Ringel form and Kac's Theorem
\[ - \dim_{\KK} \Ext^1_{K_r}(Y,X_i) = q_r(\dimu Y) < 0,\]
hence $s_i \coloneqq \dim_{\KK} \Ext^1_{K_r}(Y,X_i) \neq 0$. The previous considerations give us a short exact sequence
\[ \mathbb{E} \colon  0 \lra \bigoplus^n_{i=1} X_i^{s_i} \stackrel{f}{\lra} E_X \stackrel{g}{\lra} Y \lra 0\]
such that $E_X$ enjoys the properties of the middle term considered in \cref{6.5.5}. In particular, $E_X$ is a brick. Let $\fw \in \Gr_d(A_r) \setminus (X \cup \{\fu\})$, then $P^-_1(\fw)$ is $\Hom$-orthogonal to $X_1,\ldots,X_n,Y$ and we conclude with \cref{6.5.5}(2) that
\[ \Gr_d(A_r) \setminus (X \cup \{\fu\}) \subseteq  \{ \fw \in \Gr_d(A_r) \mid \Hom_{K_r}(P_1^-(\fw),E_X) = 0\}.\]
Moreover, $\Hom_{K_r}(X_i,E_X) \neq 0 = \Hom_{K_r}(Y,E_X)$ for $i \in \{1,\ldots,n\}$ implies 
\[ \Gr_d(A_r) \setminus X =  \{ \fw \in \Gr_d(A_r) \mid \Hom_{K_r}(P_1^-(\fw),E_X) = 0\}.\]
We conclude with \cref{4.4.3} that 
\[ \cV(K_r,d)_{E_X} = \{ \fm \in \Gr_d(A_r) \mid \Hom_{K_r}(P_1^-(\fm),E_X) \neq  0\} = X.\]
Recall from \cref{4.4.2} that $X_1,\ldots,X_n,Y$ are in $\rep_{\proj}(K_r,d-1)$. Since $\rep_{\proj}(K_r,d-1)$ is closed under extensions, we conclude $E_X \in \rep_{\proj}(K_r,d-1)$ and therefore $\cV(K_r,d-1)_{E_X} = \emptyset$.

Finally, we compute the dimension vector of $E_X$. Let $i \in \{1,\ldots,n\}$. Since $X_i$ and $Y$ are $\Hom$-orthogonal and $\dimu X_i = \dimu Y = \dimu P_1^{-}(\fv) = (d,rd-1)$, we obtain
\begin{align*}
    s_i &= \dim_{\KK} \Ext^1_{K_r}(Y,X_i) = -q_r(\dimu Y) = -q_r(\sigma_{r}(d,rd-1)) \\
    &= -q_r(1,d) = d(r-d)-1.
\end{align*}
Hence, 
\begin{align*}
    \dimu E_X &= |X| \cdot (d(r-d) -1) \cdot \dimu P_1^{-}(\fv) + \dimu P_1^{-}(\fv)  = (|X| \cdot (d(r-d) -1)+1)\cdot \dimu P_1^{-}(\fv).
\end{align*} 
\end{proof}

\bigskip

The case $d=2$, when reformulated in the setting of Steiner bundles on $\PP(A_r)$, gives rise to the following statement.

\bigskip

\begin{Theorem}\label{6.5.7}
Let $r \geq 3$ and $\emptyset \neq X \subseteq \Gr_2(A_r)$ be a finite subset. There exists an abelian and wild subcategory $\cD$ of $\rep_{\proj}(K_r,1)$ such that for every $M \in \cD$ the corresponding Steiner bundle $\TilTheta(M)$ satisfies $\cJ_{\TilTheta(M)} = X$. In particular, ${\TilTheta(M)}$ is almost-uniform.
\end{Theorem} 
\begin{proof}
We fix the regular brick $E_X$ constructed in \cref{6.5.6} and consider the simplification category $\cE(\{E_X\})$ (cf. \cite[Section 1]{KL91}). Since $E_X \in \rep_{\proj}(K_r,1)$ and $\rep_{\proj}(K_r,1)$ is closed under extensions, we conclude $\cE(\{E_X\}) \subseteq \rep_{\proj}(K_r,1)$. Since $E_X$ is regular, Kac's Theorem implies that $q_r(\dimu E_X) < 0$ and therefore 
\[ \dim_{\KK} \Ext^1_{K_r}(E_X,E_X) = - q_r(\dimu E_X) + 1 \geq 2.\]
Now the Remark following \cite[(1.4)]{KL91} implies that $\cD \coloneqq \cE(\{E_X\})$ is a wild subcategory of $\rep_{\proj}(K_r,1)$. Let $0 \neq M \in \cE(\{E_X\}) \subseteq \rep_{\proj}(K_r,1)$. By definition, we find $n\in \NN$ and a filtration
\[ 0 = M_0 \subsetneq M_1 \subsetneq M_2 \subsetneq \cdots \subsetneq M_{n-1} \subsetneq M_n = M\]
such that $M_i/M_{i-1} \cong E_X$ for all $i \in \{1,\ldots,n\}$. By \cref{6.5.6}, we have $\cV({K_r},2)_{E_X} = X$ and conclude with \cite[(2.1.2)]{BF24} that $\cV({K_r},2)_{M} = X$. Now \cref{4.4.3} implies that $\{\fu \in \Gr_2(A_r) \mid \Hom_{K_r}(P_1^-(\fu),M) \neq 0 \} = \cV(K_r,2)_{M}= X$ is a non-empty finite set and it follows 
\[ \cJ_{\TilTheta(M)} = \cJ_{M} \stackrel{\ref{6.5.2}}{=} \{\fu \in \Gr_2(A_r) \mid \Hom_{K_r}(P_1^-(\fu),M) \neq 0 \} = X.\]
\end{proof}

\bigskip

\begin{Remark}\label{6.5.8}
According to \cite[(1.4)]{KL91} and \cite[(9.2),(9.4)]{Ker94}, every indecomposable representation $M \in \cE(\{E_X\})$ is quasi-simple in a regular component $\cC_M$ and $\cC_M \neq \cC_N$ for all $M \not\cong N \in \cE(\{E_X\})$.
\end{Remark}

\section*{Acknowledgement}
Most of the results presented in this article are from my habilitation thesis \cite{Bis25b}, which I wrote at the University of Kiel.

I would to thank Frank Lukas for introducing us to the Bongartz' construction in the general context of hereditary algebras and explaining his work to us.

\printbibliography
\end{document}